\documentclass[11pt,reqno]{amsart}
\usepackage{graphicx}
\usepackage{amscd}
\usepackage{amsmath,amsthm}
\usepackage{amsfonts}
\usepackage{amssymb}
\usepackage{color}
\usepackage{hyperref}
\usepackage{tikz-cd}
\usepackage{float}
\usepackage{graphicx}
\usepackage{multirow}
\usepackage{geometry}
\DeclareMathOperator{\Span}{span}
\DeclareMathOperator{\diag}{diag}

\DeclareMathOperator{\Div}{div}
\DeclareMathOperator{\ord}{ord}
\DeclareMathOperator{\Gcd}{gcd} 

\DeclareMathOperator{\supp}{Supp}
\DeclareMathOperator{\Aut}{Aut}

\DeclareMathOperator{\mult}{mult}
\newenvironment{psmallmatrix}
  {\left(\begin{smallmatrix}}
  {\end{smallmatrix}\right)}

\newtheorem*{introtheorem}{Theorem}
\newtheorem{theorem}{Theorem}[section]
\newtheorem{lemma}{Lemma}[section]
\newtheorem{sublemma}{Sublemma}[section]
\newtheorem{corollary}{Corollary}[section]
\newtheorem{proposition}{Proposition}[section]
\newtheorem{definition}{Definition}[section]
\newtheorem{remark}{Remark}[section]
\newtheorem{convention}{Convention}

\numberwithin{equation}{section}

\theoremstyle{definition}
\newtheorem{example}[theorem]{Example}

\makeatletter
\@namedef{subjclassname@2020}{%
  \textup{2020} Mathematics Subject Classification}
\makeatother

\begin{document}

\bibliographystyle{plain}
\author{Quo-Shin Chi, Zhenxiao Xie \& Yan Xu}
\thanks{This work was partially supported by NSFC No. 12171473 for the last two authors. The second author was also partially supported by the Fundamental Research Funds for Central Universities, and the third author was also partially supported by by NSFC No. 12301065, the Fundamental Research Funds for the Central
Universities, Nankai University (63241426) and Natural Science Foundation of Tianjin (22JCQNJC00880).}
\title{Fano 3-folds and classification of  
constantly curved holomorphic $2$-spheres of degree $6$ in the complex Grassmannian $G(2,5)$} 

\begin{abstract}  Up to now the only known constantly curved sextic curve, i.e., holomorphic $2$-sphere of degree 6, in the complex $G(2,5)$ has been the first associated curve of the Veronese curve of degree 4, which indicates that such curves are rare to find. Exploring the rich interplay between the 
%Riemann sphere and projectively equivalent Fano $3$-folds of index 2 and degree $5$, 
ramification of harmonic sequences in differential geometry and 
algebro-geometric properties of projectively equivalent Fano $3$-folds of index 2 and degree $5$, we invoke the moduli space structure  %%\cite{Takagi-2, Takagi-1} 
of sextic curves in the Fano $3$-fold often referred to as $V_5$ %%of Mukai and Umemura 
%%in the literature
to confirm the rarity of constancy of curvature, by establishing that the harmonic sequence of %a generic curve of degree $6$ in $V_5$ 
a generic sextic curve in $G(2,5)$ 
is totally unramified. %in the sense of harmonic sequences 
%%and hence not of constant curvature. %\cite{He-Jiao-Zhou2015}. 
This paper proposes to investigate from the Galois viewpoint the way ramification can appear in relation to the constancy of curvature among nongeneric sextic curves in $G(2,5)$. 

To this end, we break it into two cases. %We divide the fine structure of constantly curved holomorphic $2$-spheres into two cases. 
The first is when the %%hololomorphic 
sextic curve is $GL(5,{\mathbb C})$-equivalent to a curve $\gamma\subset V_5$  not living in the $PSL_2$-invariant tangent developable surface $\bf S$ of $V_5$, where we may lift $\gamma$ to a Galois cover in the ${\mathbb C}P^3$ containing $PSL_2$. By studying the branch points of the Galois covering in connection with the %points of 
intersection of $\gamma$ %of degree $6$ 
and $\bf S$ in $V_5$, we categorize such $\gamma$ further into two sub-families, namely, the family consisting of those $\gamma$ ramified at the singular locus of $\bf S$ somewhere, to be labeled as the {\em generally ramified family}, %%thanks to its relation to the notion of ramification in the theory of harmonic sequences, 
and the family complementary to it. In  
the second case when the %%hololomorphic 
$2$-sphere is $GL(5,{\mathbb C})$-equivalent to a $\gamma$ living in %the $PSL_2$-invariant developable surface of $V_5$,   
$\bf S$, we show by the $PSL_2$-invariant theory that it necessarily belongs to the generally ramified family. 
\iffalse
To understand the fine structure of constantly curved holomorphic $2$-spheres of degree $6$, we break it into two cases. The first is when the hololomorphic $2$-sphere is $GL(5,{\mathbb C})$-equivalent to a sextic curve $\gamma$ living in the $PSL_2$-invariant tangent developable surface $\bf S$ of $V_5$. By invoking the classical $PSL_2$-invariant theory, we show that we may cover $\gamma$ by a line in ${\mathbb CP}^3$ that contains $PSL_2$, so that in particular $\gamma$ is tangent to the singular locus of $\bf S$ at some point. In the second case when the hololomorphic $2$-sphere is $GL(5,{\mathbb C})$-equivalent to a sextic curve $\gamma$ not living in $\bf S$, we may lift $\gamma$ to a Galois cover in the same ${\mathbb C}P^3$.
By studying the ramified points of the Galois covering in relation to the points of intersection of $\bf S$ and the sextic curve, we categorize such sextic curves into two families, namely, the family consisting of curves tangent to the singular locus of $\bf S$ at some point, to be labeled as the generally ramified family thanks to its relation to the notion of ramification in the theory of harmonic sequences, and the family complementary to it. 
\fi

We prove through elaborate $PSL_2$-transvectant and engaged unitary analyses that, up to the ambient unitary %%and internal M\"{o}bius 
equivalence, the moduli space of constantly curved sextic curves in $G(2,5)$ that are $GL(5,{\mathbb C})$-equivalent to those in the generally ramified family, is semialgebraic of dimension 2, all members of which barring the above Veronese curve are
%%the second fundamental form of all such $2$-spheres are not of constant norm, and hence they 
nonhomogeneous. Many explicit examples can be constructed.

We also outline  in general the structure of the Galois covers of the sextic curves  in the family complementary to the generally ramified family. It appears to suggest, through all classes of rational Galois covers we have completely classified, each dependent on a single parameter, that the constantly curved sextic curves in $G(2,5)$ that are $GL(5,{\mathbb C})$-equivalent to the ones in this family, be nongeneric among all constantly curved ones in $G(2,5)$. 

\end{abstract}

%\subjclass[2020]{53C42, 53C55}
\maketitle

%\no
%{\small{}
%\iffalse
\vskip -0.6cm
{\footnotesize 
{\bf MSC(2020): Primary 53C42, 53C55; Secondary 14M15, 14J45, 14H45}

{\bf Keywords.} holomorphic curves, Gaussian curvature, linear section, quintic del Pezzo 3-fold
}
%\vskip -0.6cm
%\indent
%\fi
\section{Introduction}
Minimal surfaces constitute one of the most enduring topics in Differential Geometry that not only enjoys its deep links with partial differential equations, complex analysis, and algebraic curves, but also finds intriguing connections to the physical world.
%%Compared to the relatively complete research in the Euclidean space, less work has been done in ambient spaces with positive sectional curvature, especially the complex projective space $\mathbb{C}P^{n}$ and the complex Grassmannian $G(k,m)$. We are concerned in this paper with the classification problem of the constantly curved holomorphic $2$-spheres of degree $6$ in $G(2,5)$.
%%This type problem attracted considerable attention in the both physic and mathematical communities. 
In 1980, Din and Zakrzewski \cite{Din-Zakrzewski1980} classified complex projective $\sigma$-models, or, mathematically, harmonic maps from the $2$-sphere to the ambient projective space, to be the (projectivized) basis elements of a Frenet frame of a holomorphic ${\mathbb C}P^1$ into the ambient space. Subsequently, Burstall and Wood \cite{Burstall-Wood1986}, Chern and Wolfson \cite{Chern-Wolfson1987}, and Uhlenbeck \cite{Uhlenbeck1989} independently generalized it to other ambient spaces by different methods. 

Of all minimal surfaces, %%To obtain some rigidity or classification results, dditional topological and geometric constrains are required, such as being 
those of  constant curvature in different ambient spaces form a model class that have continually drawn attention, such as   %$\mathbb{S}^n$, 
%%minimal surfaces of constant curvature are investigated by 
Calabi \cite{Calabi1967}, Wallach \cite{Wallach1970}, Do Carmo-Wallach \cite{Do Carmo-Wallach1971}, Chen \cite{B.Y.-Chen1972}, Barbosa \cite{Barbosa1975}, Kenmotsu \cite{Kenmotsu1976}, and Bryant \cite{Bryant1985} in the real space forms, Kenmotsu \cite{Kenmotsu1985},  Bando-Ohnita \cite{Bando-Ohnita1987}, Bolton-Jensen-Rigoli-Wood \cite{Bolton-Jensen-Rigoli-Woodward1988}, Chi-Jensen-Liao \cite{Chi-Jensen-Liao1995}, and Kenmotsu \cite{Kenmotsu2000} in the complex projective spaces, and Yau \cite{S.T.-Yau1974} in K\"{a}hler manifolds  of nonnegative constant holomorphic sectional curvature. In particular, constantly curved minimal $2$-spheres in the real space forms are Bor\r{u}vka spheres \cite{O.-Boruvka1933}, up to rigid motion. Similarly, constantly curved minimal $2$-spheres in the complex projective spaces are, up to rigid motion, the (projectivized) basis elements of the Frenet frame of the Veronese curve of constant curvature, where the proof followed from %the theory of harmonic sequence and 
Calabi's rigidity principle \cite{Calabi1953} that states that if the isometric embedding from one complex manifold into the complex projective space  exists, then it is unique up to rigid motion. %%A first glimpse, this problem falls into the categories of Mathematical Physics and Algebraic Geometry. However, the expected dimension of the moduli space in Gromov-Witten theory would not coincide with the right one \cite[p.~6]{Rose2014}. Moreover, limited research has been done in the case of degree bigger or equal than $4$ \cite{Chung-Hong-Lee2018}. It should be pointed out that a number of issues related to constructing and classifying explicit examples are still unclear.

%%Game changes after taking metric into consideration. It was 
\iffalse
On the other hand, Calabi, in as early as the 1950s, initiated the study of isometric embeddings from complex manifolds into the complex projective space \cite{Calabi1953}. %%a K\"{a}hler geometry version of the Nash embedding problem. 
His rigidity principle states that if the embedding exists, then it is unique up to an action of the isometry group of $\mathbb{C}P^n$; in particular, there follows  the rigidity result that a holomorphic $2$-sphere of constant curvature in the complex projective space must be the Veronese curve, up to the ambient unitary and internal M\"{o}bius symmetries (see also~\cite{Bolton-Jensen-Rigoli-Woodward1988} for the extension to the minimal case). 
\fi
%%The key ingredient is the so called {\emph{diastasis}}}, a distance function obtained by the normalized potential for K\"{a}hler metric which works well in analytic case.
%%This has been proven successful, for instance, in classifying minimal $2$-spheres of constant curvature in $\mathbb{C}P^{n}$. In 1988, Bloton, Jensen, Rigoli and Woodward \cite{Bolton-Jensen-Rigoli-Woodward1988} proved that they must be the so called Veronese spheres. 

The rigidity principle of Calabi no longer holds for general ambient spaces. Motivated by the Grassmannian $\sigma$-models introduced by Din and Zakrzewski \cite{Din-Zakrzewski1981} and the rigidity principle, the first named author and Zheng \cite{Chi-Zheng1989} classified the noncongruent, constantly curved holomorphic $2$-spheres of degree $2$ in $G(2, 4)$ into two $1$-parameter families, by exploring the method of moving frames and Cartan's theory of higher order invariants \cite{Jensen1977}. Later on, Li and Yu \cite{ZQLi-ZuHuanYu1999} %and Li and Jin \cite{ZQLi-MMJin2008}
 classified all constantly curved minimal $2$-spheres in $G(2,4)$, using the Pl\"{u}cker embedding and the theory of harmonic sequences.

The next simplest ambient space is the complex Grassmannian $G(2,5)$. By analyzing a $2\times 5$ matrix representation of a holomorphic ${\mathbb C}P^1$, constantly curved holomorphic $2$-spheres in $G(2,5)$ are divided into two classes by  Jiao and Peng, the {\em singular} and the {\em nonsingular} ones (a technical condition different from the usual geometric meaning, see Section \ref{JP} for definition).  They classified nonsingular constantly curved holomorphic $2$-spheres of degree less than or equal to $5$ in $G(2,5)$, and proved the nonexistence of such spheres with degree $6\leq d\leq 9$ \cite{Jiao-Peng2004,Jiao-Peng2011}.  For the singular category, however, as the degree increases the computational complexity involved in their method rises dramatically. It is thus technically difficult to apply the method to construct singular  $2$-spheres in general. Subsequently, there have emerged several partial classifications (e.g. under the condition of total unramification or homogeneity) of constantly curved holomorphic (minimal) $2$-spheres in $G(2,5)$ or $G(2,n)$ in general; see \cite{He-Jiao-Zhou2015, Peng-Xu2015} and the references therein. 

Constantly curved holomorphic $2$-spheres in $G(2,4)$ and $G(2,5)$ have also been studied by Delisle, Hussin and Zakrzewski in \cite{Delisle-Hussin-Zakrzewski2013} from the viewpoint of Grassmannian $\sigma$-models, where the classification results they obtained coincide with those mentioned above. Moreover, they posed a conjecture about the upper bound of the degrees of constantly curved holomorphic $2$-spheres in the Grassmannians. This conjecture was affirmed by them in the case of $G(2,5)$, for which the upper bound equals $6$ (see also a recent paper \cite{He2022} with more detailed proof by He).

\iffalse
The next simplest ambient space is the complex Grassmannian $G(2,5)$. Jiao and Peng classified {\emph{nonsingular}} (see Section \ref{JP} for definition) constantly curved holomorphic $2$-spheres of degree less than or equal to $5$ in $G(2,5)$, and proved the nonexistence of such spheres with degree $6\leq d\leq 9$ \cite{Jiao-Peng2004,Jiao-Peng2011} by analyzing a $2\times 5$ matrix representation of the curve. The advantage of their method is that one can construct new explicit nonsingular examples in a straightforward manner. However, as the degree increases, the computational complexity rises dramatically. It is thus technically difficult to apply the method to construct {\em singular} (in the above sense) $2$-spheres in general. Subsequently, in 2015, He, Jiao and Zhou \cite{HeJiaoZhou2015} classified holomorphic $2$-spheres of constant curvature in $G(2,5)$, under an additional assumption referred to by the authors as {\em total unramification}, via the theory of harmonic sequence. Recently, He \cite{He2022} proved that nonsingularity and total unramification could be removed for degree $7\leq d\leq 9$ in $G(2,5)$, and so established the nonexistence of constantly curved holomorphic $2$-spheres of such degrees, %%by a delicate linear algebra analysis, 
providing evidence to Delisle-Hussin-Zakrzewski's conjectures \cite{Delisle-Hussin-Zakrzewski2013} on the upper and lower bounds of the degrees of constantly curved $2$-spheres in the Grassmannians. 
\fi

At the critical degree $d=6$, however, there does exist a singular %and ramified
 (in the above sense) constantly curved holomorphic $2$-sphere of degree $6$ in $G(2,5)$, namely,%%(see section \ref{secramifiedpoints}), %%(see \eqref{knowning example} and Remark \ref{onlykonwnexample})  
%\small{
%\vskip -0.35cm
\begin{equation}\label{eq-standard}
\begin{pmatrix}
      1 & 2z & \sqrt{6}z^2 & 2z^3 & z^4 \\
      0 & 1 & \sqrt{6}z & 3z^2 & 2z^3 \\
    \end{pmatrix},
\end{equation}
%\vskip -0.33cm
\noindent 
referred to in this paper as the {\em standard} Veronese curve in $G(2,5)$. To the authors' knowledge, it has been the only known example in the literature. Surprisingly, we will show in this paper that  the moduli space of constantly curved holomorphic $2$-spheres in $G(2,5)$ has a $2$-dimensional semialgebraic component, modulo rigid motion, out of which many explicit examples can be constructed.

\iffalse
referred to in this paper as the {\em standard} Veronese curve in $G(2,5)$. To the authors' knowledge, it has been the only known example in the literature. Surprisingly, we will show in this paper that the moduli space of constantly curved $2$-spheres in $G(2,5)$ is a $2$-dimensional semialgebraic set, up to ambient $U(5)$ and internal $PSL_2$ (M\"{o}bius) symmetries.
\fi
%%There are two main motivations for this paper: to investigate whether it is unique and to establish a conceptional method. Our main result is following 

Different from all existing methods, to see whether there are constantly curved holomorphic examples of degree 6 other than the standard Veronese curve in $G(2,5)$, let us return to our paper~\cite{Chi-Xie-Xu2021} for motivation, where we investigated constantly curved holomorphic (and minimal) $2$-spheres of degree $d$ in the complex hyperquadric. Such a holomorphic $2$-sphere is a rational normal curve of degree $d$ sitting in a projective $d$-plane, so that the $2$-sphere lies in the intersection of the $d$-plane and the hyperquadric called a {\em linear section} of the hyperquadric, which is itself a quadric (may be singular). Thus, the moduli space of such $2$-spheres is essentially a fibered space over the base space that is a semialgebraic subset  of the variety of linear sections of the hyperquadric.

In the same vein, albeit more sophisticated, via the Pl\"{u}cker embedding, a holomorphic $2$-sphere of degree 6 contained in $G(2,5)\subset{\mathbb C}P^9$ is a rational normal curve (a sextic curve) sitting in a projective $6$-plane $\bf L$ in ${\mathbb C}P^9$; thus, the curve lies in the linear section ${\bf L}\cap G(2,5)$. 
 Castelnuovo ~\cite{Castelnuovo1891} showed that generic (see Section \ref{muf3f} for definition)
such linear sections constitute the intriguing class of Fano 3-folds of index 2 and degree 5 all of which are projectively equivalent (see also ~\cite{Piontkowski-Van1999} for a detailed modern account and Section \ref{muf3f} for a quick overview).  

Employing $PSL(2,\mathbb{C})$-representations, Mukai and Umemura \cite{Mukai-Umemura1983} constructed a beautiful Fano $3$-fold of index 2 and degree $5$, to be denoted by ${\mathcal H}_0^3$ henceforth (often denoted by $V_5$ in the literature), which can be identified naturally with the linear section of $G(2,5)$ cut out by the $6$-plane ${\bf L}_0$ containing the above standard Veronese curve, whose tangent developable surface ${\bf S}\subset {\mathcal H}_0^3$ plays a crucial role in the sequel, where ${\bf L}_0$ turns out to be precisely the projectivization ${\mathbb P}(V_6)$ of the irreducible $PSL(2,\mathbb{C})$-module $V_6$ of dimension 7. This fits ideally in our differential-geometric framework for computation when the condition of constant curvature is engaged. %%following the same spirit as in our recent paper \cite{Chi-Xie-Xu2021}. 

Recall a holomorphic curve $F:M\rightarrow G(2,5)\subset{\mathbb C}P^9$ is {\em unramified} at $p$ if the tangent line to $F$ at $p$ does not lie entirely in $G(2,5)$, in which case $F$ is {\em totally unramified} at $p$ if the curve $[dF\wedge dF] \subset {\mathbb C}P^4$ is unramified as a projective curve at $p$. Now, transforming by $GL(5,\mathbb{C})$, we can use the sextic curves in $\mathcal{H}^3_0$ to parameterize holomorphic $2$-spheres of degree $6$ in $G(2,5)$. Takagi and Zucconi's work on the Moduli space (Hilbert scheme) \cite{Takagi-2, Takagi-1} of sextic curves in $\mathcal{H}^3_0$,  in which %In their works, 
the intersection properties between a general sextic curve in $\mathcal{H}^3_0$ and lines and conics were investigated, turns out to characterize %be helpful, Utilizing an algebro-geometric characterization in relation to 
the total unramification of harmonic sequences (see Theorem~\ref{thm-ramichar}), from which %it facilitates us to 
we obtain that 
%{\bf Theorem 1.}
%{\em 
generic holomorphic $2$-spheres of degree $6$ in $G(2,5)$ are not of constant curvature (see Theorem~\ref{thm-generic}). 
%} 
%Here genric mean differs from a general sextic curve by a $GL(5,C)$-transformation. 

%The proof of the aforementioned theorem is based on a novel algebro-geometric characterization for totally unramification of harmonic sequences (see Theorem~\ref{thm-ramichar}). 

To  find constantly curved nongeneric holomorphic $2$-spheres of degree $6$ in $G(2,5)$, we explore the way ramification occurs from the standpoint of Galois covers. We approach this by separating the analysis into two distinct cases. When a sextic curve $\gamma\subset {\mathcal H}_0^3$ does not live in $\bf S$, by exploring Mukai and Umemura's orbit decomposition structure of ${\mathcal H}_0^3$, we may        %%extremal in the sense of Castelnuovo \cite{Arbarello-Cornalba-Griffiths-Harris1985}, 
lift $\gamma$ to a Galois cover in the natural %%$\mathbb{P}(M_2(\mathbb{C}))\cong 
 $\mathbb{C}P^3$ containing $PSL(2,{\mathbb C})$ (see Section \ref{sec5.2}).               %%where $M_2({\mathbb C})$ is the $2\times 2$ complex matrix space, 
Studying the Galois covering at its branch points that cover the points of intersection of $\bf S$ and the sextic curve, enables us to categorize all sextic curves not contained in $\bf S$ into two classes, namely, the more flexible one labeled as the {\em generally ramified family}, consisting of those sextic curves ramified at the singular locus of $\bf S$ somewhere, and the more rigid complementary one labeled as the {\em exceptional transversal family}.

In contrast, when a sextic curve $\gamma$ lives in $\bf S$ but not in its singular locus, through the $PSL_2$-invariant theory, we may explicitly lift $\gamma$ to a line in the same ${\mathbb C}P^3$, so that in particular $\gamma$ also falls in the generally ramified family.

It turns out that a constantly curved holomorphic $2$-sphere of degree $6$, $GL(5,{\mathbb C)}$-equivalent to a sextic curve lying in the generally ramified family, spans a $6$-plane $\bf L$ differing from ${\bf L_0}$ by  a diagonal transformation of $GL(5,{\mathbb C})$. This is done through elaborate $PSL_2$-invariant transvectant and engaged unitary analyses (see Section \ref{sextic}) %%\subset {\mathbb C}P^9$, which states that the curve is ramified at $p$ when its tangent line in ${\mathbb C}P^9$ at $p$ lies entirely in $G(2,5)$, 
to yield the following.

\vskip 0.3cm
\noindent
{\bf Theorem.}
{\em The moduli space ${\mathcal M}$ of constantly curved holomorphic $2$-spheres of degree $6$ in $G(2,5)$, which are $GL(5,{\mathbb C})$-equivalent to sextic curves living in the generally ramified family, is a $2$-dimensional semialgebraic set, up to the ambient $U(5)$-equivalence.}  %% and intrinsically the internal M\"{o}bius reparametrization. }%Moreover, in each Fano $3$-fold, there is at most one constantly curve holomorphic $2$-sphere, up to complex conjugation.}
\vskip 0.3cm

The moduli space structure facilitates the computation to verify that the second fundamental form of all members of ${\mathcal M}$ are not of constant norm, %nonparallel, 
and thus all but the standard Veronese curve are nonhomogeneous.

Of particular interest are three points in the moduli space $\mathcal{M}$, for each of which the corresponding Fano 3-fold contains a unique constantly curved holomorphic $2$-sphere of degree $6$, whereas the Fano 3-fold corresponding to a point other than the three in $\mathcal{M}$ contains exactly two distinct constantly curved holomorphic $2$-spheres conjugated to each other in an appropriate sense (see Sections ~\ref{sec7} and \ref{lastsection}).

\iffalse
 The collection of such $2$-spheres in the linear sections ${\bf L}\cap G(2,5)$ as depicted in the preceding theorem is referred to in this paper as the {\em diagonal family}.
%%Theorem \ref{main thm1} (see also Theorem \ref{classification theorem}).
An in-depth algebro-geometric investigation of the diagonal family enables us to characterize their moduli space in Section \ref{sec5}. %% (see Proposition \ref{prop-exist}). 
%% (see also Theorem \ref{moduli space})
\fi

\iffalse
\begin{introtheorem}
The moduli space of the diagonal family is a $2$-dimensional semialgebraic set, up to extrinsically the ambient unitary $U(5)$-equivalence and complex conjugation, and intrinsically the internal M\"{o}bius reparametrization. Moreover, in each Fano $3$-fold, there is at most one constantly curve holomorphic $2$-sphere, up to complex conjugation.
\end{introtheorem}
\fi

Our approach facilitates the explicit construction of many new examples, through algebro-geometric means, of constantly curved $2$-spheres of degree 6.  %% except for the standard Veronese curve. %%[Moreover, for these $2$-spheres, the the norm squared integral of the second fundamental form is calculated, using which we can characterize an interesting new example (see Proposition~\ref{prop-willmore}).  TO BE MODIFIED] %This is a consequence of the Gauss equation.
 %%(see also Proposition \ref{nonhomogenous})
%%\begin{proposition}\label{pro1}
%%The holomorphic $2$-spheres of the diagonal family are non-homogeneous except for the standard Veronese curve.%%$PSL_2\cdot u^6$.
%% \end{proposition} 
%%Notice that the minimal homogeneous $2$-spheres in $G(2,5)$ were classified by C.K. Peng and X.W. Xu \cite{Peng-Xu2015}. In particular, the only holomorphic homogeneous $2$-sphere of of degree $6$ is the Veronese curve $PSL_2\cdot u^6$. Our classification is geometric. %%To the authors' knowledge, the diagonal family is original and appears first time in the literature.

%%By numerical evidence, we conjectured that the moduli space of the diagonal family is topologically a unit disk in the last section. We will explore it in the future.

%We end the paper with interesting individual as well as 1-parameter families of new examples in Section \ref{lastsection}, and pose two moduli space questions.

Based on Felix Klein's work \cite{Klein1956}, we have completely classified all Galois covers of genus zero and their corresponding sextic curves in ${\mathcal H}_0^3$ for the exceptional transversal family (see Table 2, Section \ref{sec5.4}), which consists of a few $1$-parameter examples and hence at most finitely many such of constant curvature in ${\mathcal H}_0^3$ by considering total unramification. (Since the classification is long, we only indicate a couple of examples in the current paper. See Section \ref{exceptional}.)
It is tempting to suggest, %, %as a Galois cover of genus $>1$ carrries only finitely many automorphisms, 
up to $U(5)$-equivalence, that there would be at most finitely many $1$-parameter examples of constantly curved $2$-spheres in the transversal exceptional family.

The paper is organized as follows. Section $2$ is devoted to recalling the representation theory of  $PSL(2,\mathbb{C})$,  as well as Jiao and Peng's classification of nonsingular (in their sense) constantly curved holomorphic $2$-spheres in $G(2,5)$. %, and ramified points of holomorphic curves in $G(2,5)$ which will play an important role in the sequel. 
In Section 3, we introduce briefly the theory of generic linear sections of $G(2,5)$ and the Fano 3-fold ${\mathcal H}_0^3$ constructed by Mukai and Umemura. In Section \ref{Sec4}, we %%invoke the moduli space (Hilbert scheme) structure \cite{Takagi-2} of sextic curves in ${\mathcal H}_0^3$, to 
show that generic holomorphic $2$-spheres of degree $6$ in $GL(2,5)$ are not constantly curved.  %% as a consequence that such curves are totally unramified in the sense of harmonic sequences and hence not constantly curved \cite{He-Jiao-Zhou2015}. 
%%To understand the fine structure of constantly curved holomorphic $2$-spheres of degree $6$ in $G(2,5)$, 
In Section \ref{sec-para}, %%we first take care of the case when the $2$-sphere is $GL(5,{\mathbb C})$-equivalent to a curve $\gamma$ living in $\bf S,$ where it is shown that $\gamma$ can be lifted to a line in ${\mathbb C}P^3$; 
when a sextic curve $\gamma\subset {\mathcal H}_0^3$ is not in ${\bf S}$, 
we introduce its Galois lift in ${\mathbb C}P^3$, where Galois analyses lead to the generally ramified family that also includes the case when $\gamma\subset {\bf S} $ is not in the singular locus of $\bf S$, as outlined above. Starting from Section~~\ref{sextic}, we explore $PSL_2$-transvectant and engaged unitary analyses in preparation for the existence and uniqueness (Theorem \ref{prop-exist}) of constantly curved holomorphic $2$-spheres of degree $6$ in the generally ramified family in Section \ref{sec7}, and for the moduli space structure of the aforementioned Theorem (Theorem \ref{moduli space}) and related results in Section \ref{lastsection}, from which interesting individual as well as 1-parameter families of new examples are exhibited.

\section{Priliminaries}

\subsection{Irreducible representations of $SL_2(\mathbb{C})$.} ~

Let $V_n$ be the the space of binary forms of degree $n$ in two variables $u$ and $v$, on which $SL_2({\mathbb C})$ (to be denoted by $SL_2$)
acts by 
\begin{equation}\label{sl2}
\aligned
&SL_2\times V_n \rightarrow V_n,\quad (g,f) &\mapsto  (g\cdot f)(u,v)\triangleq f (g^{-1}\cdot (u,v)^t). 
\endaligned
\end{equation}
It is well-known that $V_n,~n\in\mathbb{Z}_{\geq 0}$, are the only finite-dimensional irreducible representations of $SL_2$.

Choose the following basis of $V_n$,
\begin{equation}\label{basis of repre of su2}
e_{l}\triangleq \tbinom{n}{l}^{\frac{1}{2}}u^{n-l}v^{l},~~~l=0,\ldots,n.
\end{equation}
Under this basis, write 
\begin{equation}\label{irre repre in matrix form}
(e_{0},\ldots,e_{n})\, \rho^{n}(g)\triangleq (g\cdot e_{0},g\cdot e_1,\ldots,g\cdot e_{n}). 
\end{equation}
The representation  $\rho^{n}(g): SL_2\rightarrow  GL(n+1;\mathbb{C})$ induces the wedge-product representation
\begin{equation}
\label{action on wedge}
SL_2\times V_n\wedge V_n\rightarrow V_n\wedge V_n, \quad
 (g, e_{k}\wedge e_{l})\mapsto (g\cdot e_{k}) \wedge (g\cdot e_{l}), ~~~0\leq k,l\leq n.
 \end{equation}
For the sake of clarity, we view $V_n\wedge V_n$ as the space of anti-symmetric matrices $\wedge^2 \mathbb{C}^{n+1}$, by identifying $e_{k}\wedge  e_{l}$ with the anti-symmetric matrix $E_{kl}-E_{lk}\in M_{n+1}(\mathbb{C})$, where the only nonvanishing entry of $E_{kl}$ is $1$ at the $(k,l)$ position,~$0\leq k<l\leq n$. With the basis $\{e_k\wedge e_l~|~0\leq k<l\leq n\}$ (see \eqref{basis of repre of su2}), it is not difficult to obtain the wedge-product representation in matrix form,
\begin{equation*}
\rho^n\wedge \rho^n: SL_2\times \wedge^2 \mathbb{C}^{n+1}\rightarrow \wedge^2 \mathbb{C}^{n+1}, \quad
 (g, A)\mapsto ~(\rho^{n}(g))\cdot A\cdot (\rho^{n}(g))^{t}.
\end{equation*}
The Clebsch-Gordan formula states that  (assume $m\geq n$) 
\[V_m\otimes V_n\cong V_{m+n}\oplus V_{m+n-2}\oplus\cdots \oplus V_{m-n}.\]
Moreover, for any given number $p\in [0, n]$, the projection $V_m\otimes V_n  \rightarrow V_{m+n-2p}$ can be formulated by 
\begin{equation}\label{transvectant}
(f,h)\mapsto (f,h)_{p}\triangleq \frac{(m-p)!(n-p)!}{m!n!}\sum_{i=0}^{p}(-1)^{i}\binom{p}{i}\frac{\partial^p f}{\partial u^{p-i}\partial v^i}\frac{\partial^p h}{\partial u^{i}\partial v^{p-i}},
\end{equation}
which is $PSL_2$-equivariant and is called the \emph{$p$-th transvectant} \cite[Eq (2.1), p.~16]{PopoviciuDraisma2014}. Moreover,
\begin{equation}\label{eq-clebsch}
V_n\wedge V_n \cong V_{2n-2}\oplus V_{2n-6}\oplus \ldots \oplus V_{r},
\end{equation}
where $r$ is the remainder of $2n-2$ divided by $4$, and the projections are the same as \eqref{transvectant}. 

Gordan proved that the binary sextic $V_6$ has $5$ invariants and $26$ covariants given by a finite number of iterated transvectants \cite[Table 1.1, p.~12; Theorem 2.1.3, p.~18]{PopoviciuDraisma2014}, among which $(f,f)_2,~(f,f)_4,~(f,f)_6$  \cite[Sections 4.5, 5.6]{PopoviciuDraisma2014} appear in geometry in an unexpected way (see Proposition \ref{orbits defined by transvectant}).

\subsection{Holomorphic $2$-spheres in $G(2,5)$}\label{JP}~

We briefly review some basic facts of constantly curved holomorphic $2$-spheres in the complex Grassmannian $G(2,5)$, and along the way  introduce those \emph{nonsingular} ones that Jiao and Peng \cite{Jiao-Peng2004} defined and classified. 

Throughout, we equip $G(2,5)$ with the standard K\"ahler metric 
%%$$\mathrm{g}\triangleq tr\left((I_2+P\, P^*)^{-1}dP\, (I_5+P^*\, P)^{-1}dP^*\right),$$
%%where $P\in G(2,5)$ is seen as a $2\times 5$ matrix, which is 
induced from the Fubini-Study metric of ${\mathbb C}P^9$ when
$G(2,5)$ is realized as a subvariety of $\mathbb{C}P^9$ by the Pl\"{u}cker embedding, 
\[i: G(2,5)\rightarrow \mathbb{P}(\wedge^2 \mathbb{C}^5)\cong \mathbb{C}P^9,\;\Span\{u,v\}\mapsto [u\wedge v].\]
Explicitly, let $\{\epsilon_0,\epsilon_1,\ldots,\epsilon_4\}$ be a basis of $\mathbb{C}^5$. Then  $\{\epsilon_i\wedge \epsilon_j\,|\, 0\leq i<j\leq 4\}$ forms a basis of $\wedge^2 \mathbb{C}^5$ so that $p=[\sum_{i,j} p_{ij}\,\epsilon_i\wedge \epsilon_j]$ belongs to $G(2,5)$ if and only if $p\wedge p=0$, which is equivalent to 
{\small\begin{equation}\label{G(2,5)}
\aligned
&p_{01}p_{23}-p_{02}p_{13}+p_{03}p_{12}=0,\quad p_{01}p_{24}-p_{02}p_{14}+p_{04}p_{12}=0,\\&p_{01}p_{34}-p_{03}p_{14}+p_{04}p_{13}=0,\quad
p_{02}p_{34}-p_{03}p_{24}+p_{04}p_{23}=0,\\&p_{12}p_{34}-p_{13}p_{24}+p_{14}p_{23}=0.
\endaligned
\end{equation}}
\begin{remark}\label{rk-plucker}
It follows from the definition that $G(2,5)$ is $PSL_2$-invariant under the wedge-product action $\rho^4\wedge \rho^4$ given in \eqref{action on wedge}. 
\end{remark}

Let  $\varphi:\mathbb{C}P^1\rightarrow G(2,5)$ be a holomorphic $2$-sphere. 
It follows from the Normal Form Lemma \cite{Piontkowski1996} that there exist two holomorphic curves $f,g:\mathbb{C}P^1\rightarrow \mathbb{C}P^4$, such that $\varphi=\Span\{f,g\}$.  Explicitly, choosing an affine coordinate $z$ on $\mathbb{C}P^1$, we can write 
$f(z)=(f_0(z),\ldots,f_4(z))$ and $g(z)=(g_0(z),\ldots,g_4(z))$ as row vectors with polynomial entries except at some isolated points. 

In view of Remark~\ref{rk-plucker}, we obtain that $\varphi$ is of constant curvature $K$ if and only if $i\circ \varphi$ is of constant curvature $K$ under the Pl\"ucker embedding. This guarantees that the rigidity principle of Calabi can be employed to study constantly curved holomorphic $2$-spheres in $G(2,5)$, which we rephrase as follows for reference. 

\begin{lemma}\label{constant curvature}
Let $f:\mathbb{C}P^1\rightarrow \mathbb{C}P^{n}$ be a holomorphic $2$-sphere of degree $d$. The following are equivalent.
\begin{itemize} \item[(1)]  
The Gauss curvature $K$ of $f$ is $\frac{4}{d}$ \emph{.} Furthermore, up to the action of $U(n+1)$ and M\"{o}bius reparametrization, $f$ is given by the Veronese sphere
\begin{equation}\label{Veronse sphere}
Z_d(z)\triangleq [1:\sqrt{d}z:\cdots:\sqrt{\tbinom{d}{k}}z^k:\cdots:z^d]^t.
\end{equation}
\item[(2)] 
There is an affine chart $z\in{\mathbb C}$ over which $|f|^2=(1+|z|^2)^d$.
\item[(3)] 
There is an affine chart $z\in{\mathbb C}$ over which $ f=\sum_{k=0}^{d} \sqrt{\tbinom{d}{k}}A_k z^k$, and $\{A_0, A_1,\cdots, A_6\}$ forms an orthonormal basis of the $d$-plane spanned by $f$.
\end{itemize}
\end{lemma}

For a constantly curved holomorphic $2$-sphere $\varphi:\mathbb{C}P^1\rightarrow G(2,5)$, it is known  \cite{Jiao-Peng2004,  ZQLi-ZuHuanYu1999} that $\varphi$ can be parameterized as $\varphi=\big(\varphi_1(z)^t,\varphi_2(z)^t\big)^t$ with 
\begin{equation}\label{eq-spara}
\hspace{-3mm}
\varphi_1(z)\!\!=\!\!\big(\!1,0,\varphi_{12}(z),\varphi_{13}(z),\varphi_{14}(z)\big),~ \varphi_2(z)\!\!=\!\!\big(\!0,1,\varphi_{22}(z),\varphi_{23}(z),\varphi_{24}(z)\big),\!\!
\end{equation}
where $\varphi_{1i}(z)$ and $\varphi_{2i}(z)$ ($2\leq i\leq 4$) are polynomials vanishing at $z=0$. In the sequel, \eqref{eq-spara} will be called a \emph{standard parameterization} of $\varphi$.  We point out that this kind of parameterization is not unique. In fact, if $\{\varphi_1,~\varphi_2\}$ is a standard parameterization of $\varphi$, then $\{\alpha \varphi_1+\beta \varphi_2, -\bar{\beta}\varphi_1+\bar{\alpha}\varphi_2\}$ is also a standard parameterization after rotating $\epsilon_0$ and $\epsilon_1$ while maintaining $|\alpha|^2+|\beta|^2=1$. 

In  \cite{Jiao-Peng2004}, a holomorphic $2$-sphere  $\varphi:\mathbb{C}P^1\rightarrow G(2,5)$ is called \emph{nonsingular} if there exists a standard parameterization $\{\varphi_1,~\varphi_2\}$ of $\varphi$, such that
$[\varphi_1(\infty)]\neq [\varphi_2(\infty)]$ in $\mathbb{C}P^4$.  Otherwise, $\varphi$ is called  \emph{singular}.  It is easy to verify that $\varphi$ is nonsingular if and only if there exists a standard parameterization $\{\varphi_1, \varphi_2\}$ of $\varphi$, such that
\begin{equation}\label{eq6}
\deg\varphi=\deg \varphi_1 +\deg \varphi_2.
\end{equation}

Using a standard parameterization, one can construct explicitly nonsingular examples as was done by Jiao and Peng in \cite{Jiao-Peng2004}. Indeed, under the nonsingular assumption, Jiao and Peng in the paper proved the following nonexistence result. 

\begin{theorem}
There does not exist {\emph{nonsingular}} holomorphic constantly curved $2$-spheres of degree $6$ in $G(2,5)$.
\end{theorem}

The idea goes as follows. By contradiction, otherwise, it would follow from  \eqref{eq6} that we had only three possibilities that
$(\deg \varphi_1,\deg \varphi_2)=(5,1),~(4,2),~(3,3).$
In each case, we obtained vectors $A_k,~0\leq k\leq 6$, where
$i\circ \varphi=\varphi_1\wedge \varphi_2\triangleq\sum_{k=0}^{6} \sqrt{\tbinom{d}{k}}A_k z^k$,
in terms of  undetermined coefficients of $\varphi_1$ and $\varphi_2$ to violate item (3) of Lemma \ref{constant curvature}. 

As the degree of $\varphi$ increases, however, the number of undetermined coefficients rises dramatically, so that it is technically difficult to apply the method to construct {\em singular} $2$-spheres.

It is readily verified that the Veronese curve \eqref{eq-standard} given in the introduction is singular in terms of Jiao and Peng's definition, where a standard parameterization in the sense of \eqref{eq-spara} can be chosen to be
\begin{equation}\label{eq-standard1}
\begin{pmatrix}
      1 & 0 & -\sqrt{6}z^2 & -4z^3 & -3z^4 \\
      0 & 1 & \sqrt{6}z & 3z^2 & 2z^3 \\
    \end{pmatrix}. 
\end{equation}
We point out that this example is smooth (nonsingular) in the usual algebro-geometric sense, which is indeed what we are after.

\subsection{Reducible and Irreducible holomorphic curves in $G(2,5)$}\label{secramifiedpoints}~

For later purposes, we develop the extrinsic geometry of holomorphic curves in $G(2,5)$ from the viewpoint of developable surfaces.

Let $f:M\rightarrow G(2,5)$ be a holomorphic map from a Riemann surface $M$. Composed with the Pl\"{u}cker embedding, $ F\triangleq  i\circ f$ is a holomorphic curve in $\mathbb{C}P^9=\mathbb{P}(\wedge^2{\mathbb{C}^5})$. Since $F$ lies in $G(2,5)$,  we have $F\wedge F\equiv 0$, whose derivative with respect to a local complex coordinate $z$ yields that $F\wedge {\partial F}/{\partial z}=0$. Consider the \emph{tangent developable surface} $\mathcal{D}$ of $F$ in $\mathbb{C}P^9$, spanned by $F$ and its tangent line ${\partial F}/{\partial z}$, 
\[\mathcal{D}\triangleq \{[u\,F+v\,{\partial F}/{\partial z}]~|~[u:v]\in\mathbb{C}P^1\}.\]

\begin{lemma}\label{main prop}
The following are equivalent.
\begin{itemize}
  \item[(1)] The tangent developable surface $\mathcal{D}$ of $F$ lies in $G(2,5)$.
  \item[(2)] ${\partial F}/{\partial z}\wedge {\partial F}/{\partial z}\equiv 0.$ 
\end{itemize}
\end{lemma}
The lemma follows by differentiating $(u\,F+v\,{\partial F}/{\partial z})\wedge (u\,F+v\,{\partial F}/{\partial z})=0$ while employing $F\wedge {\partial F}/{\partial z}=0$.

Inspired by the first item in Lemma \ref{main prop}, we call a holomorphic curve $f:M\rightarrow G(2,5)$ {\emph{reducible}}, if the tangent developable surface $\mathcal{D}$ of $F= i \circ f$ also lies in $G(2,5)$; otherwise, we call $f$ \emph{irreducible}. 
If $f:M\rightarrow G(2,5)$ is irreducible, then ${\partial F}/{\partial z}\wedge  {\partial F}/{\partial z}$ has isolated zeroes, which we call \emph{ramified points} (with multiplicity) and $f$ is said to be {\em ramified} at these points.

\begin{remark}
In the theory of harmonic sequences, a holomorphic curve $f:M\rightarrow G(2,5)$ is called reducible if the rank of the next term $f_1$ is strictly less than $2$; see \cite{Jiao-YXu2018}. This definition coincides with the above definition. We thank Professor L. He for helpful discussions about it. 
\end{remark}

It was proven in \cite{Fei-Jiao2011} that  a constantly curved reducible holomorphic $2$-sphere of degree $6$ is rigid, which is unitarily equivalent to the {\em standard} Veronese curve \eqref{eq-standard} in $G(2,5)$. As a result, we need only consider irreducible holomorphic $2$-spheres in $G(2,5)$ in the sequel.
\iffalse
{\color{red}Moreover, in our language here, the totally unramified curve is defined equivalently as following.
\begin{definition}\cite[Def 2.4, p.~24]{He-Jiao-Zhou2015}
Let $f:M\rightarrow G(2,5)$ be a holomorphic curve, and $F:=i\circ f:M\rightarrow\mathbb{C}P^9$ be the 
Pl\"{u}cker embedding of $f$. We call $f$ is \emph{totally unramified}, if $f$ has no ramified points and $\frac{\partial F}{\partial z}\wedge \frac{\partial F}{\partial z}:M\rightarrow \mathbb{C}P^4$ is an immersion.
\end{definition}

Using He, Jiao and Zhou's classification, we obtain the following.
\begin{theorem}\cite[Cor 5.4, p.~43]{He-Jiao-Zhou2015}\label{totally unramified no constant curvature}
There is no totally unramified holomorphic two-spheres of constant curvature and degree $6$ in $G(2,5)$
\end{theorem}
}
\fi

\section{Algebro-geometric preparation}
\subsection{Generic linear sections of $G(2,5)$ and Fano 3-folds of index 2 and degree 5}\label{muf3f}~

To motivate, a holomorphic $2$-sphere of degree 6 in $G(2,5)$ lies in a $6$-plane $\bf L$ in $\mathbb{P}(\wedge^2 \mathbb{C}^5)\cong \mathbb{C}P^{9}$, and so it lives in the intersection ${\bf L}\cap G(2, 5)$ called a {\em linear section} of $G(2,5)$. 
The dual $2$-plane of ${\bf L}$ in $(\wedge^2{\mathbb C}^5)^{\ast}$ is given by a linear system
\begin{equation}\label{LS}
[\lambda A+ \mu B + \tau C],~~~\;\;\;\;\;\;~~~[\lambda:\mu:\tau]\in{\mathbb C}P^2,
\end{equation} 
where $A,B,C$ are fixed skew-symmetric matrices of size $5\times 5$ identified with elements in $(\wedge^2{\mathbb C}^5)^{\ast}$. Following \cite{Piontkowski-Van1999}, we say that $\bf L$ is {\em generic} if all matrices in the linear system are of rank 4, and the associated cut ${\bf L}\cap G(2,5)$ is referred to as a {\em generic} linear section. Let us look at a concrete example next. 

By the Clebesch-Gordan formula \eqref{eq-clebsch},  we obtain that $\wedge^2 \mathbb{C}^5 \cong V_6 \oplus V_2$. Here, we identify $V_6$ with a $SL_2$-invariant subspace of $5\times 5$ anti-symmetric matrices by 
{\small
 \begin{align}\label{standard6plane}
 \begin{split}
\sum_{i=0}^{6}\sqrt{\tbinom{6}{i}}a_i u^{6-i}v^i&\mapsto \begin{psmallmatrix}
0 & a_0 & a_1 & \sqrt{\frac{3}{5}}a_2 & \frac{1}{\sqrt{5}}a_3\\
-a_0 & 0 & \sqrt{\frac{2}{5}}a_2  & \frac{2}{\sqrt{5}}a_3 & \sqrt{\frac{3}{5}}a_4\\
-a_1 & -\sqrt{\frac{2}{5}}a_2 & 0 & \sqrt{\frac{2}{5}}a_4 & a_5 \\
-\sqrt{\frac{3}{5}}a_2 & -\frac{2}{\sqrt{5}}a_3 & -\sqrt{\frac{2}{5}}a_4 &0 &a_6\\
-\frac{1}{\sqrt{5}}a_3 & -\sqrt{\frac{3}{5}}a_4 & -a_5 & -a_6 & 0
\end{psmallmatrix}.
\end{split}
\end{align}}
Let $\{e_i\}$ be an orthonormal basis of $\mathbb{C}^5$. An orthonormal basis of $V_6$ is given by 
{\small\begin{align}\label{basisOfV6}
\begin{split}
E_0\triangleq &e_0\wedge e_1,\quad E_1\triangleq e_0\wedge e_2,\quad E_2\triangleq \sqrt{3/5}\,e_0\wedge e_3+\sqrt{2/5}\,e_1\wedge e_2,\\
E_3\triangleq &{1}/{\sqrt{5}}\,e_0\wedge e_4+{2}/{\sqrt{5}}\,e_1\wedge e_3,\quad E_4\triangleq  \sqrt{3/5\,}e_1\wedge e_4+\sqrt{2/5\,}e_2\wedge e_3,\\
E_5\triangleq &e_2\wedge e_4,\quad E_6\triangleq e_3\wedge e_4.
\end{split}
\end{align}}

It is readily checked that $uv(u^4-v^4)$ (respectively, $u^6$) in $V_6$ corresponds to $(E_1-E_5)/\sqrt{6}$ (respectively, $E_0$).
Note that, the dual plane to $V_6$ is given by a linear system of the form in \eqref{LS}, where
\begin{equation}\label{Name}
A\triangleq \sqrt{6}p_{03}-3p_{12}=0,~~~B\triangleq 2p_{04}-p_{13}=0,~~~C\triangleq \sqrt{6}p_{14}-3p_{23}=0.
\end{equation}
It is also readily checked that the rank of $[\lambda A+ \mu B + \tau C]$ is $4$ for every $[\lambda:\mu:\tau]\in{\mathbb C}P^2$. Therefore, as a linear section, 
\[\mathcal{H}_0^3\triangleq {\mathbb P}(V_6)\cap G(2,5),\] is generic. 

Note also that the space ${\mathbb P}(V_6)$ is the $6$-plane spanned by the standard Veronese curve in \eqref{eq-standard}, which is precisely the orbit $PSL_2\cdot u^6$ confirmed by a computation with
$
(E_0,\cdots,E_6)\cdot Z_6(z)
$, where $Z_6$ is given in \eqref{Veronse sphere}, to see that they are agreeable.

We include a short outline of the following well-known fact for the reader's convenience. Our reference is \cite{Piontkowski-Van1999}.

\begin{theorem} \label{thm-equi}
All generic linear sections ${\bf L}\cap G(2,5)$ are  $PGL(5,\mathbb{C})$-equivalent to ${\mathcal H}_0^3$.
\end{theorem}

To begin, the {\em Pfaffian} of a $(2n)\times(2n)$ skew-symmetric matrix $M$ with entries $a_{ij}$ is defined to be
\[ pf(M)\triangleq \sum_\sigma sgn(\sigma)\, a_{i_1\,j_1}a_{i_2\,j_2}\cdots a_{i_n\,j_n},\]
where $\sigma: \{1,2,\cdots, 2n\}\rightarrow \{i_1,j_1,i_2,j_2,\cdots,i_n,j_n\}$, in order, runs over permutations of $\{1,2,\cdots, 2n\}$ satisfying $i_s<j_s, 1\leq s\leq n,$ and $i_1<i_2<\cdots<i_n$. The Pfaffian enjoys the property that if $N$ is a $(2n+1)\times(2n+1)$ skew-symmetric matrix of rank $2n$, then the 1-dimensional kernel of $N$ is spanned by the vector $(v_1,\cdots,v_{2n+1})$, where $v_i$ is the diagonal Pfaffian of the $(2n)\times(2n)$ skew-symmetric matrix obtained by deleting the $i$th row and column.

Now, since the dual 2-plane of a generic 6-plane $\bf L$ in $\mathbb{P}(\wedge^2({\mathbb C}^5))$ is a linear system $[\lambda A + \mu B +\tau C]$, $[\lambda:\mu:\tau]\in{\mathbb C}P^2$, all of whose $5\times 5$ skew-symmetric matrices are of rank 4, we can use the associated diagonal Phaffians to define the {\em center} map
 $$
{\bf c}:[\lambda:\mu:\tau]\in {\mathbb C}P^2\rightarrow \text{projectivized center of}\; [\lambda A + \mu B +\tau C]\in {\mathbb C}P^4.
 $$
 It is then verified that the center map is an embedding of ${\mathbb C}P^2$ into ${\mathbb C}P^4$ of degree 2, and thus the image of ${\bf c}$, called the {\em projected Veronese surface}, is a generic projection from the standard Veronese surface in ${\mathbb C}P^5$ to ${\mathbb C}P^4$. Consequently, any two such $2$-plane linear systems are $PGL(5,{\mathbb C})$-equivalent, and so are the corresponding linear sections. In fact, ${\bf L}\cap G(2,5)$ is the closure of all lines in ${\mathbb C}P^4$ intersecting the associated projected Veronese surface in three distinct points. 
 
 Exploring the center map $c$, the authors in \cite{Piontkowski-Van1999} also obtained the automorphism group of a generic linear section ${\bf L}\cap G(2,5)$.  \begin{theorem} \label{thm-automorphism}
The automorphism group of a generic linear section ${\bf L}\cap G(2,5)$ is $PSL_2$.  
\end{theorem}

Generic linear sections ${\bf L}\cap G(2,5)$ constitute all Fano $3$-folds of index 2 and degree $5$, first classified by Castelnuovo \cite{Castelnuovo1891}, a typical one of which is to be denoted by ${\mathcal H}^3$ henceforth; here, the degree is that of the Fano 3-fold as a subvariety of ${\mathbb C}P^9$, and the index is the difference between its degree and codimension in $G(2,5)$, so that its anti-canonical bundle is $\simeq {\mathcal O}(2)$. To reference, we call ${\mathcal H}_0^3={\mathbb P}(V_6)\cap G(2,5)$ introduced earlier the {\em standard}\, Fano 3-fold.

We point out that the automorphism group of a Fano $3$-fold of index 2 and degree 5 has also been studied by Mukai and Umemura in \cite{Mukai-Umemura1983} from the viewpoint of algebraic group actions. By considering the action of $PSL_2$ on ${\mathbb P}(V_6)$, they proved that the closure of $PSL_2 \cdot uv(u^4-v^4)$ is precisely ${\mathcal H}^3_0$. In the same paper, they also obtained the following beautiful orbit decomposition structure on ${\mathcal H}^3_0$. 

\begin{theorem}%\cite{Mukai-Umemura1983}
\label{OrbitDecomp}
{\small\begin{align*}
\aligned
{\mathcal H}^3_0 =\overline{PSL_2 \cdot uv(u^4-v^4)}
 =PSL_2 \cdot uv(u^4-v^4) \sqcup PSL_2 \cdot u^5v \sqcup PSL_2 \cdot u^6.
\endaligned
\end{align*}}
\end{theorem}

 \begin{remark}\label{f12}
 In the above orbit decomposition,  $PSL_2 \cdot uv(u^4-v^4)$ is of dimension $3$, which is parameterized as
{\small \begin{equation}\label{f}
\aligned
&f_1: PSL_2\mapsto {\mathbb P}(V_6), \quad
[\begin{pmatrix}
a & b\\
c & d
\end{pmatrix}] \mapsto [\begin{pmatrix}
a & b\\
c & d
\end{pmatrix} \cdot uv(u^4-v^4)]=[a_0:a_1:\cdots:a_6],\\
&a_0\triangleq -\sqrt {6}{d}^{5}c+\sqrt {6}d{c}^{5},\quad a_1\triangleq {d}^{4} \left( ad+5\,bc \right) -c^4(5\,ad+bc),\\
&a_2\triangleq -b{d}^{3} \left( ad+2\,bc \right) \sqrt {10}+a{c}^{3} \left( 2\,ad+bc
 \right) \sqrt {10}
, \\&a_3\triangleq {b}^{2}{d}^{2} \left( ad+bc \right) \sqrt {30}-{a}^{2}{c}^{2} \left( a
d+bc \right) \sqrt {30}
,\\
&a_4\triangleq -{b}^{3}d \left( 2\,ad+bc \right) \sqrt {10}+{a}^{3}c \left( ad+2\,bc
 \right) \sqrt {10}
,\\&a_5\triangleq b^4(5\,ad+bc)-{a}^{4} \left( ad+5\,bc \right) ,\quad
a_6\triangleq -\sqrt {6}{b}^{5}a+\sqrt {6}b{a}^{5}.
\endaligned
\end{equation}}
\!\!Similarly, the orbit $PSL_2 \cdot u^6$ is parameterized as 
{\small \begin{equation}\label{companion3}
\aligned
[\begin{pmatrix}a&b\\c&d\end{pmatrix}]\mapsto [d^6:-\sqrt{6}bd^5:\sqrt{15}b^2d^4:-\sqrt{20}b^3d^3:\sqrt{15}b^4d^2:-\sqrt{6}b^5d:b^6].
\endaligned
\end{equation}}
 
 It is precisely the Veronese curve $Z_6$ in \eqref{Veronse sphere}.  Its tangent developable surface constitutes the closure of the $2$-dimensional orbit {\em (}see \cite%[p.~497]
 {Mukai-Umemura1983}{\em )},
  \[\overline{PSL_2 \cdot u^5v}=  PSL_2 \cdot u^5v \sqcup PSL_2 \cdot u^6,\]
where $PSL_2 \cdot u^5v$ has the following parameterization 
 \begin{align}\label{companion2}
\aligned
&f_2: PSL_2\mapsto {\mathbb P}(V_6), \quad
[\begin{pmatrix}
a & b\\
c & d
\end{pmatrix}] \mapsto [\begin{pmatrix}
a & b\\
c & d
\end{pmatrix}\cdot u^5v]=[b_0:b_1:\cdots:b_6],\\
&b_0\triangleq -\sqrt {6}{d}^{5}c,\quad b_1\triangleq {d}^{4} \left( ad+5\,bc \right) ,\quad b_2\triangleq -b{d}^{3} \left( ad+2\,bc \right) \sqrt {10},\\
&b_3\triangleq {b}^{2}{d}^{2} \left( ad+bc \right) \sqrt {30},\quad b_4\triangleq -{b}^{3}d \left( 2\,ad+bc \right) \sqrt {10},\\&b_5\triangleq b^4\,(5ad+bc),\quad
b_6\triangleq -\sqrt {6}{b}^{5}a.
\endaligned
\end{align}

\iffalse

{\color{red}
It is well-known, since $f_2$ is of bidegree $(1,5)$ in $((a,c), (b,d))$, that $[f_2]$ defines a morphism into the tangent developable surface, the closure of  $PSL_2\cdot u^5v$, by
$$
[f_2]: [a:c]\times [b:d]\in {\mathbb C}P^1\times {\mathbb C}P^1\longmapsto [b_0:b_1: \cdots:b_6]\in {\mathbb C}P^6,
$$
which is bijective and ramifies at the diagonal $\Delta$ of the domain, where $\Delta$ is mapped to $PSL_2\cdot u^6$, the $1$-dimensional $PSL_2$-orbit {\rm (}see \cite[p.31]{Ein-Lazarsfeld}{\rm )}.

Similarly, we define $f_2^*$ through $f_2$ by swapping the pair $(a,b)$ and the pair $(c,d)$. Then 
generically $[f_1]=[f_2-f_2^*]$ is a point on the line connecting the two points $[f_2]$ and $[f_2^*]$.
}

\fi

\end{remark}

Meanwhile, using the invariants and covariants of the binary sextic (see \eqref{transvectant}), we remark that the above orbits have another $SL_2$-invariant characterization. %, which is well known 
\begin{proposition}\label{orbits defined by transvectant}
Given $f=\sum_{i=0}^{6}\sqrt{\tbinom{6}{i}}a_i u^{6-i}v^i$ defining $[f]\in \mathbb{P}(V_6)$, we have
\begin{enumerate}
\item[(1)] $[f]$ lies in $\mathcal{H}_0^3=\mathbb {P}(V_6)\cap G(2,5)=\overline{PSL_2 \cdot uv(u^4-v^4)}$ if and only if the $4$-th transvectant $(f,f)_4=0$,

\item[(2)] $[f]$ lies in the closed $2$-dim orbit $\overline{PSL_2 \cdot u^5v}$ if and only if the $4$-th and $6$-th transvectants $(f,f)_4$ and $(f,f)_6$ vanish, and

\item[(3)] $[f]$ lies in the $1$-dimensional orbit $PSL_2 \cdot u^6$ if and only if the $2$nd transvectant $(f,f)_2=0$.
\end{enumerate}

\end{proposition}
For later purposes, we quote the following well known calculations:  $(f,f)_2=\text{Hess}(f)/450$. 
\begin{equation}\label{eQ}
(f,f)_6=
2a_0a_6 - 2a_1a_5 + 2a_2a_4 - a_3^2.
\end{equation}%\begin{proof}
%To prove (1), 
%write 
$$(f,f)_4=\sum\limits_{i=0}^4 \sqrt{\tbinom{4}{i}}t_i u^{4-i}v^i,~~~~~~\quad\text{where}$$ 
%%\begin{align}
\begin{equation}\label{five equations in aj fano}
%%\begin{split}
\aligned
&t_0\triangleq\frac{2}{15}(\sqrt{15}a_0a_4-\sqrt{30}a_1a_3+3a_2^2),\quad t_1\triangleq \frac{\sqrt{6}}{15}(5a_0a_5-\sqrt{15}a_1a_4+\sqrt{2}a_2a_3),\\
&t_2\triangleq \frac{\sqrt{6}}{15}(5a_0a_6-3a_2a_4+2a_3^2),\quad t_3\triangleq \frac{\sqrt{6}}{15}(5a_1a_6-\sqrt{15}a_2a_5+\sqrt{2}a_3a_4),\\\
&t_4\triangleq \frac{2}{15}(\sqrt{15}a_2a_6-\sqrt{30}a_3a_5+3a_4^2).
\endaligned
\end{equation}
%%\end{split}
%%\end{align}
%%By \eqref{G(2,5)}, \eqref{standard6plane} and \eqref{basisOfV6}, we find 
We point out that $\frac{1}{\sqrt{6}}f\wedge f=\sum\limits_{i=0}^4 t_{4-i} e_0\wedge \cdots \wedge \widehat{e_{i}}\wedge \cdots \wedge e_4\in \mathbb{P}(\wedge^4(\mathbb{C}^5))$, where the notation $\widehat{e_{i}}$ means that we omit the term $e_{i}$. %Hence $f$ lies in $G(2,5)$ if and only if $(f,f)_4=0$.

%In regard to (2), on $\mathbb{P}(V_6)$, the $6$-th transvectant is given by 
\iffalse
\begin{equation}\label{eQ}
(f,f)_6=
2a_0a_6 - 2a_1a_5 + 2a_2a_4 - a_3^2.
\end{equation}
Moreover, it is directly checked that \eqref{eQ} vanishes on the closed $2$-dimensional orbit by substituting \eqref{companion3} and \eqref{companion2} into \eqref{eQ}.
\fi
Meanwhile, since ${\mathcal H}_0^3$ and the $5$-quadric $Q_5$ defined by 
\begin{equation}\label{5-quadric}
Q_5\triangleq\{[f]\in {\mathbb P}(V_6): (f,f)_6=0\}
\end{equation}
are both $PSL_2$-invariant in ${\mathbb P}(V_6)$, the closed $2$-dimensional $PSL_2$-orbit is precisely $Q_5\cap {\mathcal H}_0^3$.%, as the former is the only 2-dimensional $PSL_2$-invariant orbit in the orbit decomposition by Theorem \eqref{OrbitDecomp} (see also \cite[Example 2.7.2, p.~32]{PopoviciuDraisma2014} for an algebraic explanation).

Note also that %For (3), note that $450(f,f)_2=\frac{\partial^2 f}{\partial^2 u}\frac{\partial^2 f}{\partial^2 v}-(\frac{\partial^2 f}{\partial u\partial v})^2$ is the Hessian of $f$, which 
$(f,f)_2$ vanishes if and only if $f$ is the $6$-th power of a linear form; see \cite[Prop 5.3, p.~71]{kung-Rota1984} for an algebraic reason. %For instance, $f=(du-bv)^6$ corresponds to the point \eqref{companion3}. 
%\end{proof}

\begin{remark}\label{isotropy group} The isotropy groups of the  two orbits of ${\mathcal H}^3_0$ of dimension $\geq 2$ are given below.\\
{\rm(1)}\,The open orbit $PSL_2\cdot uv(u^4-v^4):$ Its isotropy group at $ uv(u^4-v^4)$ is the projective binary octahedral group of order $24$, isomorphic to $S_4$, consisting of the following elements {\rm (}$\xi\triangleq e^{2k\pi\sqrt{-1}/8},~k
= 0,1,\ldots,3${\rm )}{\rm :}
{\small\[\begin{pmatrix}
  \xi & 0 \\
  0 & 1/\xi \\
\end{pmatrix},~~\begin{pmatrix}
  0 & \xi \\
  -1/\xi & 0 \\
\end{pmatrix},~~
1/\sqrt{2}\cdot\begin{pmatrix}
  1/\xi & -1/\xi \\
  \xi & \xi \\
\end{pmatrix},\]~~ \[1/\sqrt{2}\cdot\begin{pmatrix}
  \sqrt{-1}/\xi & -1/\xi \\
  \xi & -\sqrt{-1}\xi \\
\end{pmatrix},~~ 1/\sqrt{2}\cdot\begin{pmatrix}
  -1/\xi & -1/\xi \\
  \xi & -\xi \\
\end{pmatrix},~~ 1/\sqrt{2}\cdot\begin{pmatrix}
  -\sqrt{-1}/\xi & -1/\xi \\
  \xi & \sqrt{-1}\xi\\
\end{pmatrix}.\]}

\noindent {\rm(2)}\,The $2$-dimensional orbit $PSL_2\cdot u^5v:$ Its isotropy group at $u^5v$ is 
\[{\small\{\begin{pmatrix}
                                                                                           a & 0 \\
                                                                                                0 & 1/a \\
                                                                                              \end{pmatrix}|~a\in\mathbb{C}^{\ast}
  \}~~\mod \pm I_2.}\]

\end{remark}

For later computational purposes, we prove the following.

\begin{lemma}\label{commutediag}
Let $A$ be a matrix in $SL_2$. Then
\begin{equation}\label{usefullemma}
\rho^4(A)\cdot (E_0,E_1,\ldots, E_6)=(E_0,E_1,\ldots, E_6)\,\rho^6(A),
\end{equation}
where the left-hand side with a dot is the $\wedge^2$-action of $\rho^4(A)$ on $V_6\subset \wedge^2(\mathbb{C}^5)$ and the right-hand side without a dot is a matrix multiplication.
\end{lemma}

\begin{proof}  Since the Clebsch-Gordon transvectant $\pi\triangleq f\wedge g\rightarrow(f,g)_1$ in~\eqref{transvectant} is $SL_2$-equivariant, we obtain from the commutativity of the diagram

\begin{equation}\label{diagram}
\begin{tikzcd}
V_4 \wedge  V_4\arrow{r}{\rho^4(A)} \arrow{d}{\pi}& V_4 \wedge  V_4\arrow{d}{\pi}\\
V_6  \arrow{r}{\rho^6(A)} & V_6
\end{tikzcd}
\end{equation}
that $\rho^6(A):V_6\rightarrow V_6$ is induced from the $\wedge^2$-action of  $\rho^4({A})$ (see \eqref{action on wedge}). 
\end{proof}

\section{Generic holomorphic $2$-spheres of degree $6$ in $G(2,5)$}\label{Sec4}
In this section, we prove that generic holomorphic $2$-spheres of degree $6$ in $G(2,5)$ are not of constant curvature. Here, a  holomorphic $2$-sphere of degree 6 is called generic if it differs from a  {\em general} (in the sense given in \cite[Condition 3.20]{Takagi-1}) rational normal curve of degree $6$ (a sextic curve) in the standard Fano 3-fold $\mathcal{H}^3_0$ by a transformation in $GL(5,\mathbb{C})$.  
%\subsection{General rational normal curves of degree $d$ in $\mathcal{H}^3_0$}
%In this subsection, 

Firstly, we review some results of general sextic curves in $\mathcal{H}^3_0$ referred to as the  quintic del Pezzo 3-fold and denoted by $V_5$ in \cite{Takagi-2, Takagi-1}. In these two papers, Takagi and Zucconi investigated the moduli space (Hilbert scheme) of sextic curves  in $\mathcal{H}^3_0$. (Their results are more general; we only invoke the special case when the curve degree is $6$.) Let $H^6$ be the Hilbert scheme whose general points parameterize sextic curves  in  $\mathcal{H}^3_0$. The following results (see Corollary~3.10 in \cite{Takagi-2}, Proposition~2.3.1, Proposition~2.3.3 and Proposition~2.3.4 in \cite{Takagi-1}) were proved.   
\begin{proposition} \label{prop-Takagi}
%{\bf (Takagi-Zucconi)} 
The closure of $H^6$ 
%of the Hilbert Scheme of rational normal curves of degree $d$ in $\mathcal{H}^3_0$ 
is an irreducible variety of dimension $12$. Moreover, for a general sextic curve $C_6$ in $\mathcal{H}^3_0$, \\
{\rm (1)} $C_6$ intersects the closure of the $2$-dimensional orbit $\overline{PSL_2\cdot u^5v}$ simply, \\ 
{\rm (2)} there exist at most finitely many bi-secant lines of $C_6$ in $\mathcal{H}^3_0$, and any one of them intersects $C_6$ simply, and\\
{\rm (3)} $Q|_{C_6}$ has no point of multiplicity greater than $2$ for any multi-secant conic $Q$.
\end{proposition} 

It turns out the above proposition can be used to prove that general sextic curves in $\mathcal{H}^3_0$ are totally unramified in the sense of harmonic sequences \cite{Chern-Wolfson1987}, from which we can derive that generic holomorphic $2$-spheres of degree $6$ in $G(2,5)$ are not of constant curvature.  Recall (below Lemma \ref{main prop}) that a holomorphic $2$-sphere $F:\mathbb{C}P^1 \rightarrow G(2,5)$ is unramified if $F'\wedge F'$ is nowhere vanishing, in which case it is called {\em totally unramified} if, furthermore, the curve $[F'\wedge F']: \mathbb{C}P^1 \rightarrow {\mathbb P}(\Lambda^4\mathbb{C}^5)\cong \mathbb{C}P^4$ is unramified as a projective curve, which is equivalent to saying that $F''\wedge F'$ is nowhere parallel to $F'\wedge F'$. 
%$F=f\wedge g:\mathbb{C}P^1 \rightarrow G(2,5)$ is called totally unramified if $\text{Span}\{f,g,f',g'\}$ is of rank $4$ everywhere and ,
Our key observation is the following interesting algebro-geometric characterization of total unramification. 
\begin{theorem}\label{thm-ramichar}
Let $F:\mathbb{C}P^1 \rightarrow G(2,5)$ be a holomorphic $2$-sphere of degree $6$. \\
{\rm (1)} $F'\wedge F'$ is zero at a point $p$ if and only if the tangent line of $F$ at $p$ lies in $G(2,5)$. \\
{\rm (2)} Assume $F'\wedge F'$ is nonzero at $p$. If $[F'\wedge F']$ is ramified at $p$, then there exists a conic $Q$ tangent to $F$ at $p$ such that $Q|_{F}$ has multiplicity no less than $3$ at $p$. 
\end{theorem}
\begin{proof}
The conclusion in (1) follows from $F\wedge F'=0$ so that  $$(F+tF')\wedge (F+tF')=t^2 F'\wedge F',~~~t\in \mathbb{C}.$$ 

For the conclusion in (2), we assume that $F=f\wedge g$. Since $F'\wedge F'$ does not vanish at $p$, we can choose a basis $\{e_1, e_2, \cdots, e_5\}$ of $\mathbb{C}^5$ such that 
%$$F(p)=e_1\wedge e_2,~~~F'(p)=e_1\wedge e_3- e_2\wedge e_4,$$ 
\begin{equation}\label{eq-F01p}
F(p)=e_1\wedge e_2,~~~F'(p)=e_1\wedge e_3- e_2\wedge e_4,
\end{equation}
and 
\begin{equation}\label{eq-F2p}
F''(p)=f''(p)\wedge e_2-2e_3\wedge e_4+e_1\wedge g''(p).
\end{equation}

If $[F'\wedge F']$ is ramified at $p$, then there exist two complex number $\alpha$ and $\beta$ such that  $$(\alpha F'(p) +\beta F''(p))\wedge F'(p)=0.$$ It follows that 
\begin{equation}\label{eq-fg2p}
f''(p), g''(p)\in \{e_1, e_2, e_3, e_4\}.
\end{equation}

Consider the $2$-plane $P_2$ spanned by $\{F(p), F'(p), F''(p)\}$ and its intersection with $G(2,5)$. Using \eqref{eq-F01p}$\sim$\eqref{eq-fg2p}, it is easy to verify that 
$[F(p)+xF'(p)+yF''(p)]$ lies in $G(2,5)$ if and only if 
\begin{equation}\label{eq-ramiconic}
-4y+2x^2+\lambda y^2+\mu xy=0
\end{equation}
for two constants $\lambda$ and $\mu$, which means that the intersection $P_2\cap G(2,5)$ is exactly a conic. We denote this conic by $Q$. 

We choose a local coordinate $z$ near $p$ such that $z(p)=0$. It follows from the Taylor expansion of $F$ at $z=0$ that, near $p$, $F$ can be parameterized as 
$[1:z:\frac{z^2}{2}:\frac{z^3}{3!}: \cdots: \frac{z^d}{d!}],$
with respect to the frame 
$\{F(p), F'(p), F''(p), \cdots, F^{(d)}(p)\}$
on the $d$-plane containing $F$. Substituting $x=z$ and $y=\frac{z^2}{2}$ into the left-hand side of \eqref{eq-ramiconic}, we have 
$$-4y+2x^2+\lambda y^2+\mu xy=z^3(\frac{\mu}{2}+\frac{\lambda}{4}z),$$
which implies $Q|_{F}$ has multiplicity no less than $3$ at $p$. 
\end{proof}
 Proposition~\ref{prop-Takagi}, Theorem~\ref{thm-ramichar}, and an easy construction of a totally unramified sextic curve in $G(2,5)$ whose curvature is not constant,      %%He-Jiao-Zhou's classification in the  totally unramified case \cite[Corollary 5.4, p.~43]{He-Jiao-Zhou2015}, 
imply that a generic holomorphic $2$-sphere of degree $6$ is totally unramified, and so we obtain the main result of this section. 
\begin{theorem}\label{thm-generic}
Generic holomorphic $2$-spheres of degree $6$ %%{\rm(}with $3\leq d\leq 6${\rm)} in $G(2,5)$ 
in $G(2,5)$ are not of constant curvature.   
\end{theorem}

{\section{Galois covering of the holomorphic $2$-spheres of degree $6$ in $G(2,5)$}\label{sec-para}

We see from the preceding section that holomorphic $2$-spheres of degree $6$ with constant curvature in $G(2,5)$ are nongeneric. To understand better how and when the ramification in the sense of harmonic sequences can appear, %the fine structure of such $2$-spheres, 
we look at it from the Galois point of view. We divide the discussion according to whether the curve lies in the closed $2$-dimensional $PSL_2$-orbit as follows. 

\subsection{The case when the curve lies in the closed $2$-dimensional orbit} ~ 

Part of the following theorem is known to algebraists \cite{weyman}. We give a straightforward proof pertaining to our geometric situation here. 
\begin{theorem}\label{lem-case2}
Let $F:\mathbb{C}P^1\rightarrow \mathcal{H}_0^3 $ be a rational normal curve of degree $6$. Assume $F$ lies in the closed $2$-dimensional orbit $\overline{PSL_2\cdot u^5v}$ but does not coincide with the $1$-dimensional orbit. Then $F$ can be lifted to a projective line $\phi:\mathbb{C}P^1\rightarrow \mathbb{C}P^3$ in the diagram

\begin{equation}\label{diagram in quadric}
\begin{tikzcd}
&\mathbb{C}P^3 \arrow[dashed]{d}{f_2}\\
\mathbb{C}P^1  \arrow{r}{F} \arrow{ur}{\phi}& \mathbb{C}P^6,
\end{tikzcd}
\end{equation}
where $f_2$ is given in \eqref{companion2}. Moreover, $F$ intersects the $1$-dimensional orbit.
\end{theorem}

\begin{proof} We give a proof based on the $PSL_2$-invariant theory.

Firstly, we show the existence of the lift $\phi$. Assume that $F=\sum\limits_{i=0}^6 a_i(z)\sqrt{\tbinom{6}{i}}u^{6-i}v^i$, where $a_i(z)$ are polynomials of $z$ with $a_i(z)\neq 0$ because $F$ is linearly full. Then by \eqref{five equations in aj fano}, we obtain %the graph structure
\begin{align}\label{a4,a5,a6graph1}
\begin{split}
a_4&=\frac{\sqrt{10}a_1a_3-\sqrt{3}a_2^2}{\sqrt{5}a_0},~\quad\quad\quad a_5=\frac{\sqrt{30}a_1^2a_3-3a_1a_2^2-\sqrt{2}a_0a_2a_3}{5a_0^2},\\
a_6&=\frac{3\sqrt{10}a_1a_2a_3-3\sqrt{3}a_2^3-2\sqrt{5}a_0a_3^2}{5\sqrt{5}a_0^2}.
\end{split}
\end{align}
By \eqref{eQ} and \eqref{a4,a5,a6graph1}, $F$ lying in the closed $2$-dimensional orbit $\overline{PSL_2\cdot u^5v}$ is equivalent to
\begin{equation}\label{curves lies in closed 2-dim orbit}
0=Q_5=-\frac{9}{5}a_3^2-\frac{2\sqrt{2}(\sqrt{15}a_1^2-9a_0a_2)a_1}{5a_0^2}a_3-\frac{2(8\sqrt{15}a_0a_2-15a_1^2)a_2^2}{25a_0^2}.
\end{equation}
We can directly write down the lift $\phi=[\begin{pmatrix}a&b\\c&d\end{pmatrix}]$ by assigning
\begin{equation}\label{lift when curve in quadric}
a=\frac{\sqrt{6}a_1}{a_0}+\frac{-5\sqrt{10}a_1a_2+15\sqrt{5}a_0a_3}{10a_1^2-4\sqrt{15}a_0a_2},~~~b=\frac{-\sqrt{10}a_1a_2+3\sqrt{5}a_0a_3}{10a_1^2-4\sqrt{15}a_0a_2},
\end{equation}
$c=-1,$ and $d=1$.  In fact, given $\phi\cdot u^5v=(u+av)(u-bv)^5$, we derive that 
\begin{equation}\label{why lift in quadric case is right}
\sum\limits_{i=0}^6 g_i(z)\sqrt{\tbinom{6}{i}}u^{6-i}v^i\triangleq F-a_0 (u+av)(u-bv)^5=0 
\end{equation}
in $\mathbb{C}(a_0,a_1,a_2)[a_3]/(Q_5)$, viewing $a_0,\ldots,a_3$ as independent variables. To see this, by direct computations, $g_0=g_1=0$, and 
$g_i=r_i\cdot Q_5,~2\leq i\leq 6,$
for some polynomials $r_i\in \mathbb{C}(a_0,a_1,a_2)[a_3]$ with degree $\deg_{a_3}(r_i)=i-2,~2\leq i\leq 6$, which can be obtained by Euclid's division algorithm; for example, $r_2=\frac{\text{coeff}(g_2,a_3,2)}{\text{coeff}(Q_5,a_3,2)}$, etc. The only thing to remark is that 
\begin{equation}\label{quadric but not in 1dim}
5a_1^2-2\sqrt{15}a_0a_2\not \equiv 0
\end{equation}
in \eqref{lift when curve in quadric}. Otherwise, $a_2=\frac{\sqrt{15}a_1^2}{6a_0}$, and then by \eqref{curves lies in closed 2-dim orbit}, $0=Q_5=-\frac{(\sqrt{30}a_1^3-18a_0^2a_3)^2}{180a_0^4}$ to yield $a_3=\frac{\sqrt{30}a_1^3}{18a_0^2}$, so that by the graph structure \eqref{a4,a5,a6graph1} we obtain $F=\frac{(\sqrt{6}a_0u+a_1v)^6}{216a_0^5}$, which contradicts the assumption that $F$ does not coincide with the $1$-dimensional orbit. (In the following Remark \eqref{why we find the lift}, we will motivate the choice of $a$ and $b$ given in the lift \eqref{lift when curve in quadric}.) Moreover, $F$ intersects the $1$-dimensional orbit $PSL_2\cdot u^6$ at the zeros of $a+b$.

In conclusion, we have the above commutative diagram \eqref{diagram in quadric}. Next, we show that $\phi(\mathbb{C}P^1)$ is a projective line. We may assume that $a,b,c,d$ are polynomials of an affine coordinate $z$, after factoring out the common denominator. Set $\alpha\triangleq \Gcd(a,c)$, $\beta\triangleq \Gcd(b,d)$, and  $A\triangleq \begin{pmatrix}a/\alpha&b/\beta\\c/\alpha&d/\beta\end{pmatrix}$. Then 
\[\phi\cdot u^5v=\begin{pmatrix}a&b\\c&d\end{pmatrix}\cdot u^5v=(A\begin{pmatrix}\alpha&0\\0&\beta\end{pmatrix})\cdot u^5v=(\beta^5\alpha) \;A\cdot u^5v=A\cdot u^5v,\]
after projectivizing. We may thus assume that $\Gcd(a,c)=\Gcd(b,d)=1$ and $\phi=A$ in the following arguments.

By \eqref{companion2} and that $F$ is nondegenerate in ${\mathbb C}P^6$, none of $a,b,c,d$ are identically zero, from which there induces
two non-constant holomorphic maps 
\begin{equation*}
\phi_1:\mathbb{C}P^1\rightarrow \mathbb{C}P^1,~z\mapsto [a:c];\quad\quad\quad\quad\phi_2:\mathbb{C}P^1\rightarrow \mathbb{C}P^1,~z\mapsto [b:d].
\end{equation*}
Moreover, the coordinates $b_0, b_1, \cdots, b_6$ of $F$ given in \eqref{companion2} cannot vanish simultaneously at any point of $\mathbb{C}$. Therefore, 
{\small
\begin{equation} \label{eq-deg}
\deg(F)=\max\limits_{0\leq i\leq 6}\{\deg(b_i)\}\leq \max\{\deg a,\deg c\}+5\max\{\deg b,\deg d\}=\deg(\phi_1)+5\deg (\phi_2),
\end{equation}}

\noindent where we have used the fact that $b_i$ are homogeneous of bidegree $(1,5)$ in $(a,c)$ and $(b,d)$, respectively. We assert that the reverse inequality of \eqref{eq-deg} also holds. To this end,  
multiplying a matrix from the left by interchanging the rows, we may assume that $\deg(a)\geq \deg(c)$. If $\deg(b)\geq \deg(d)$, then 
\[\deg(F)\geq \deg(b_6)=\deg(a)+5\deg(b)=\deg(\phi_1)+5\deg (\phi_2);\]
otherwise,
$\deg(F)\geq \deg(b_1)=\deg(a)+5\deg(d)=\deg(\phi_1)+5\deg (\phi_2)$.
Hence, we have the equality in \eqref{eq-deg}. Lastly, since both $\phi_1$ and $\phi_2$ are non-constant, it follows from $\deg(F)=6$ that $\deg(\phi_1)=\deg(\phi_2)=1$. Therefore $\phi$ is a projective line in $\mathbb{C}P^3$. 
                                                                                          
\iffalse

{\bf Claim}. There exist a matrix $B\in SL_2$ and nonzero polynomials $\alpha,\beta$ in $z$, such that
\[
\begin{pmatrix}a&b\\c&d\end{pmatrix}=BA(z)\begin{pmatrix}\alpha&0\\0&\beta\end{pmatrix},\]
where $A:\mathbb{C}P^1\rightarrow \mathbb{C}P^3$ is of degree $1$. 
Note that if this claim holds, then 
$\psi(z)\triangleq [BA(z)]$ defines a line in ${\mathbb C}P^3$ satisfying $F=f_2\circ \psi=f_2\circ \phi$, which completes the proof of Lemma~\ref{lem-case2}.
\fi

\end{proof} 
\begin{remark}\label{why we find the lift}
It follows from the $PSL_2$-invariant theory that the curve $F$ lying in the closed $2$-dimensional orbit is equivalent to $F$ and $\frac{\partial F}{\partial u}$ having a greatest common divisor $G$ of positive degree in $u$. Indeed, their resultant with respective to $u$ is $\text{Res}_{u}(F,\frac{\partial F}{\partial u})=62208a_0v^{30}Q_5^5=0$. Moreover, 
$G$ can be found by Eculid's algorithm through $F=(\gamma u+\mu v) \frac{\partial F}{\partial u}+G$, 
where
\[ G=\frac{(2\sqrt{15}a_0a_2-5a_1^2)v^2}{6a_0}u^4-\frac{(a_1a_2\sqrt{2}-3a_0a_3)\sqrt{5}v^3}{3a_0}u^3+\cdots\triangleq c_4u^4v^2+c_3u^3v^3+\cdots.\]
So, $G$ is of degree $4$ in $u$ by \eqref{quadric but not in 1dim}. 
The proof that ${\partial F}/{\partial u}$ is divided by $G$, and $G$ has a root $b$ of multiplicity $4$ is similar to \eqref{why lift in quadric case is right}. Thus, by the relations between roots and coefficients for $G$, we derive $b=\frac{-c_3}{4c_4}$ given in \eqref{lift when curve in quadric}.
%%\[b=\frac{(\frac{(a_1a_2\sqrt{2}-3a_0a_3)\sqrt{5}}{3a_0})}{4\cdot \frac{(2\sqrt{15}a_0a_2-5a_1^2)}{6a_0}}.\]
Moreover, $b$ is also the root of $F$ with multiplicity $5$, and the simple root $-a$ of $F$ can also be found through $-(-a)-5b=-\frac{\sqrt{6}a_1}{a_0}$.
\end{remark}

\subsection{The case when the curve does not lie in the closed $2$-dimensional orbit}\label{sec5.2}~

We identify the projectivization of the space of $2\times 2$ nonzero (complex) matrices with ${\mathbb C}P^3$ by 
$$\iota: \begin{pmatrix}a&b\\c&d\end{pmatrix} \mapsto [a:b:c:d].$$%} %\end{tiny}.
Via $\iota$, the subset of $2\times 2$ matrices of zero determinant is the following $PSL_2$-invariant hyperquadric $Q_2$ of dimension $2$,
\begin{equation}\label{2q}
Q_2\triangleq \{[a:b:c:d]\in{\mathbb C}P^3~|~ ad-bc=0\}.
\end{equation}
Note that we can identify $PSL_2$ with $\mathbb{C}P^3 \setminus Q_2$. 
\begin{theorem}\label{lift of curve not in 2dim orbit}
Let $F:\mathbb{C}P^1\rightarrow \mathcal{H}_0^3\subset G(2,5)$ be a sextic curve. If $F$ does not lie in the closed $2$-dimensional orbit $\overline{PSL_2\cdot u^5v}$, then there exists a compact Riemann surface $g:M\rightarrow \mathbb{C}P^3$ covering $F$ as in the following commutative diagram
\begin{equation}\label{diagram2}
\begin{tikzcd}
M\arrow{r}{g} \arrow{d}{\varphi}&\mathbb{C}P^3 \arrow[dashed]{d}{f_1}\\
\mathbb{C}P^1  \arrow{r}{F} & \mathbb{C}P^6
\end{tikzcd}
\end{equation}
Moreover, $\varphi:M\rightarrow \mathbb{C}P^1$ is a {\rm (}branched{\rm )} Galois covering, and the group of covering transformations $G\triangleq\{\sigma\in \Aut(M)~|~\varphi\circ\sigma=\varphi\}$ is a subgroup of $S_4$ isomorphic to the isotropy group at $uv(u^4-v^4)$ given in item {\emph{(1)}} of Remark \eqref{isotropy group}.
\end{theorem}

\begin{proof}
Recall the invariant quadric $Q_5$ defined in \eqref{eQ}, which cuts the sextic curve $F$ in a divisor of degree $12$ with support points $q_1,\cdots,q_l$ by Bezout's theorem. 

{In the following, we abuse the notation to denote by $q$ either a point of the curve $F(\mathbb{C}P^1)$ or its preimage on $\mathbb{C}P^1$,}  whenever there is no possibility of confusion.

Consider the complementary set 
$V\triangleq {\mathbb C}P^1\setminus\{q_1,\ldots,q_l\}$; $F(V)$ lies in the open $3$-dimensional orbit $Y\triangleq PSL_2\cdot uv(u^4-v^4)$. Let 
$U$ be a connected component of the fibered product 
\begin{equation}\label{U}
U\subset V\times_Y PSL_2\triangleq \{(p,B)\in V\times PSL_2 : F(p)=f_1(B)\},
\end{equation}
with the two standard projections $\pi_1$ and $\pi_2$ onto $V$ and $PSL_2\subset {\mathbb C}P^3$, respectively. 
Then $U$ is an unramified covering space of $V$, by item (1) of Remark \ref{isotropy group}. We extend $\pi_1:U\rightarrow V$ to a ramified covering $\varphi:M\rightarrow \mathbb{C}P^1$ by the monodromy representation \cite[Theorem 8.4, p.~51]{Forster1999}, where $M$ is a compact Riemann surface. Hence, we obtain the commutative diagram \eqref{diagram2}, where $g$  extends $\pi_2$, $\varphi$ extends $\pi_1$, and $M$ is the desingularization of the closure of $\pi_2(U)$ in ${\mathbb C}P^3$.

Furthermore, the group of covering transformations $G=\{\sigma\in \Aut(M)~|~\varphi\circ\sigma=\varphi\}$ is isomorphic to the group
\[\widetilde{G}=\{C\in S_4~|~\forall ~(q,B)\in U\subset V\times PSL_2,~\text{s.t.}~(q,B C)\in U\}.\]
It is easy to see that the elements of the group $\widetilde{G}$ permutes the points on a regular fiber of $\varphi$; thus, we obtain the isomorphism
%%\begin{align*}
$$
\widetilde{G}\rightarrow G,\quad C\mapsto \sigma_C\triangleq [(q,B)\in U \mapsto (q, B C^{-1})],
%%\end{align*}
$$
%where the dot is the multiplication in $PSL_2$. 
whose inverse is given by 
%%\begin{align*}
$$
G\rightarrow \widetilde{G}\subset S_4,\quad \sigma\mapsto C_\sigma\triangleq g(\sigma(q))^{-1} g(q),~~\forall q\in U,
%%\end{align*}
$$
where $C_\sigma$ is well-defined due to that $U$ is connected and the isotropy group $S_4$ is finite.

Furthermore, the order of $\widetilde{G}$ equals $d\triangleq\deg \varphi$, the number of points on a regular fiber. Indeed, given a point $(q_0,B_0) \in U$, the fiber over $q_0$ is
\[\{(q_0, B_0 C_i),~|~ C_i\in S_4,~1\leq i\leq d\}.\]
By definition, we have $\widetilde{G}\subset \{C_1,\ldots,C_d\}$.
On the other hand, for a given $1\leq j\leq d$, since $U$ and $U\cdot C_j\triangleq \{(q, B C_j)~|~\forall~(q,B)\in U\}$ are two connected components of the fiber product $V\times_Y PSL_2 $ through the same point $(q_0, B_0 C_j)$ and so are identical, we conclude that $C_j\in \widetilde{G}$.

To show the Galoisness of $\varphi$, given the data in \eqref{f} and \eqref{diagram2}, consider the polynomial equation
\begin{equation}\label{splitting}
p(z,x) \triangleq \sum\limits_{i=0}^6\sqrt{\tbinom{6}{i}}a_i(z)x^i=0,
\end{equation}
where the entries of $F(z)=[a_0(z):\cdots:a_6(z)]$ belong to the  polynomial ring ${\mathbb C}[z]$ such that the coefficients are relatively prime with the maximum degree $6$. Then the splitting field of $p(\varphi,x)\triangleq \varphi^*(p(z,x))$ over $\mathbb{C}(\varphi)$, where $\varphi$ is given in \eqref{diagram2}, is exactly the function field ${\mathbb C}(M)$ of the covering $M$. 

To see this, the splitting field belongs to ${\mathbb C}(M)$ due to that the six roots of $p(\varphi,x)$ are linear fractions of the coordinate functions $a, b, c, d$ of $g(M)$. In fact, the map $g$ over $M$ in \eqref{diagram2} splits $f_1\circ g=u^6\cdot p(\varphi,\frac{v}{u})$ (see \eqref{splitting}) into
\begin{equation*}\label{Galois}
\aligned
&f_1\circ g=\begin{pmatrix}a&b\\c&d\end{pmatrix}\cdot uv(u^4-v^4)=(du-bv) (av-cu)\big((du-bv)-(av-cu)\big)\times \\
&\big((du-bv)+(av-cu)\big)\big((du-bv)+\sqrt{-1}(av-cu)\big)\big((du-bv)-\sqrt{-1}(av-cu)\big),
\endaligned
\end{equation*}
where the six distinct roots are
\begin{equation}\label{six}
\sigma_1\triangleq\frac{d}{b},\; \sigma_2\triangleq\frac{c}{a},\;\sigma_3\triangleq\frac{c+d}{a+b},\;\frac{c-d}{a-b},\;\frac{c+\sqrt{-1}\,d}{a+\sqrt{-1}\,b},\;\frac{c-\sqrt{-1}\,d}{a-\sqrt{-1}\,b}.
\end{equation}
Note that the denominators of the six roots can never be identically zero, since either of them being identically zero would imply $a_6=ab(a^4-b^4)=0$ {\rm{(}}the last coordinate in \eqref{f}{\rm{)}}, contradicting that $F$ is linearly full.

Conversely, let ${\bf F}\supset {\mathbb C}(\varphi)$ be any intermediate field of the function field of $M$. If  ${\bf F}$ contains the splitting field of \eqref{splitting}, then we can use the first three roots in \eqref{six} to solve for $\frac{b}{a}=\frac{\sigma_2-\sigma_3}{\sigma_3-\sigma_1}$ so that the curve $g:M\rightarrow \mathbb{C}P^3$ is given by
\begin{equation}\label{coordinatize}
[a:b:c:d]= [1:\frac{b}{a}:\frac{c}{a}: \frac{d}{b}\cdot \frac{b}{a}]=[1:\frac{\sigma_2-\sigma_3}{\sigma_3-\sigma_1}:\sigma_2:\sigma_1\frac{\sigma_2-\sigma_3}{\sigma_3-\sigma_1}].
\end{equation}
Thus ${\bf F}$ contains $\mathbb{C}(\varphi)(a,b,c,d)={\mathbb C}(M)$ since $g:M\rightarrow  \mathbb{C}P^3$ is generically injective. 

We conclude that the covering $\varphi:M\rightarrow \mathbb{C}P^1$ is Galois of order $d=[{\mathbb C}(M):{\mathbb C}(\varphi)]$, and the group of covering transformations of $\varphi$ is the Galois group $\Aut(\mathbb{C}(M)/\mathbb{C}(\varphi))$.

\iffalse

Combine the above discussions, we see that $G$ acts on the fibers over $V$ transitively, thus the cover $\varphi$ is Galois (see in \cite[8.6. Def, p.~52; 8.12. Thm, p.~57]{Forster1999}).

\fi

\end{proof}}

For the lift $g=\begin{pmatrix}
a & b \\
c & d
\end{pmatrix}:M\rightarrow \mathbb{C}P^3$ as in Theorem \ref{lift of curve not in 2dim orbit}, to be referred to as a Galois lift of $F$, we associate it with two meromorphic functions
\[x\triangleq c/a,~~~w\triangleq b/a.\]
We employ the geometry of the octahedron to study the Galois covering $\varphi$.

The quadric $Q_2$ defined in \eqref{2q} is a saddle surface in ${\mathbb C}P^3$ isomorphic to $\mathbb{C}P^1\times \mathbb{C}P^1$ by the parametrization $\begin{pmatrix}
1 & w\\
x & wx
\end{pmatrix}$, where each pair $(w,x)$ determines uniquely a point $p\in Q_2$, through which there passes a unique $w$-ruling $L_{w}\triangleq\{(w,x)~|~x\in\mathbb{C}P^1\}$. Let $S^2$ be the unit $2$-sphere centered at the origin in ${\mathbb R}^3$, and let ${\mathbb C}$ be the complex plane projected onto by the stereographic projection $\eta: S^2\setminus\{(0,0,1)\}\longrightarrow {\mathbb C}$, with $(0,0,1)$ mapped to $\infty$. We identify the regular octahedron in $S^2$ by  sending its top and bottom vertices to $(0,0,1)$ and $(0,0,-1)$, respectively, and identifying the four horizontal vertices with $(\pm 1, 0,0), (0,\pm 1,0)$. 

It is well-known that the symmetric group $S_4$ is isomorphic to the projective binary octahedral group, which acts on the regular octahedron as the rotational group of symmetry, consisting of the identity, $6$ quarter turns and $3$ half turns around the axes passing through two opposite vertices (see the $6$ blue points in Figure \ref{Octahedron}), $6$ half turns around the axes passing through two opposite edge centers (see the $12$ red points in Figure \ref{Octahedron}), and $8$ one-third turns around the axes passing through the centers of two opposite faces (see the $8$ green points in Figure \ref{Octahedron}). The above $26$ points (to be called centers in the following) enumerate all points on the regular octahedron whose stabilizers are nontrivial under the above group action. 

\begin{figure}[htbp]
\centering
\includegraphics[width=0.25\textwidth]{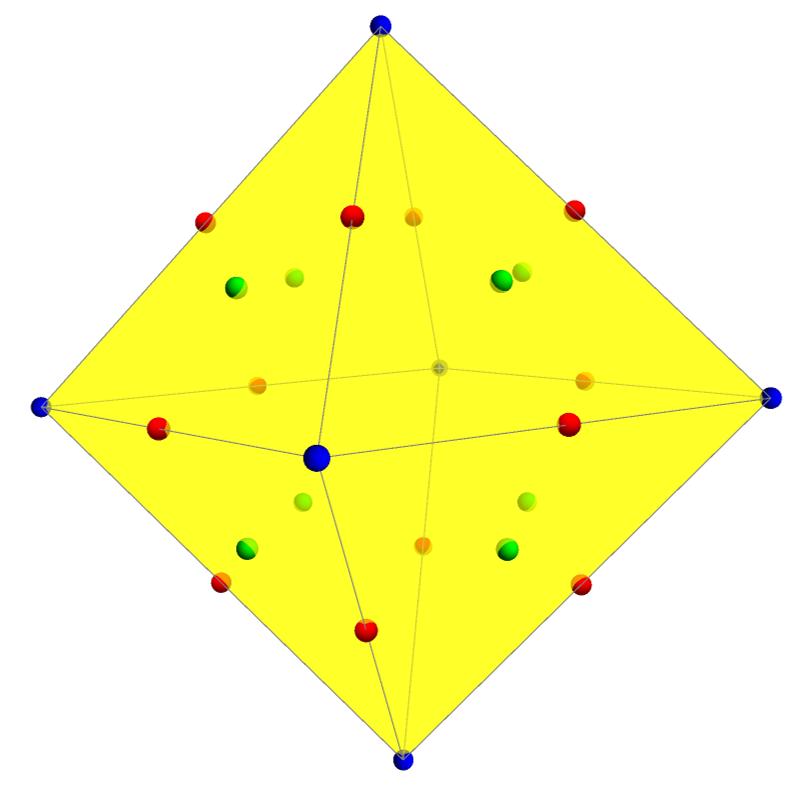}
\caption{Octahedron}\label{Octahedron}
\end{figure}

Composing the central projection of the regular octahedron on $S^2$ with the stereographic projection, we can build a one-to-one correspondence between the points on the octahedron and points on the extended plane $\mathbb{C}\cup \{\infty\}$. Under this correspondence, the above $26$ centers turn out to be the roots of the polynomial equation 
\begin{equation}\label{26}
(w^5-w)(w^8+14w^4+1)(w^{12}-33w^8-33w^4+1)=0,
\end{equation}
where we also count $w=\infty$ as a root. We point out that the above three polynomial factors take exactly the vertices, edge centers, and face centers as their roots, respectively. Furthermore, via this correspondence, the symmetric group $S_4$ (isomorphic to the projective binary octahedral group) acts on the $w$-rulings of $Q_2$, and is exactly the action of the isotropic group in item (1) of Remark \eqref{isotropy group} on  $\begin{pmatrix}
1 & w\\
x & wx
\end{pmatrix}$
by matrix multiplication on the right, where the $26$ centers are related to the $26$ distinct eigenvectors of these isotropy matrices. %%in item (1) of Remark \eqref{isotropy group}. 

Now, we introduce two important divisors to study the Galois covering~\eqref{diagram2}. In the following, we denote the degree of the covering $\varphi:M\rightarrow \mathbb{C}P^1$ by $d$, and the degree of the curve $g:M\rightarrow \mathbb{C}P^3$ by $k$.

Let ${\mathcal Q}$ be the intersection divisor defined by $g$ and the quadric $Q_2$. %%defined by $ad-bc=0$ in $\mathbb{C}P^3$, 
By  Bezout's theorem, we have $\deg({\mathcal Q})=2k$. %We also use $\supp \mathcal{Q}$ (resp. $\supp \mathcal F$) to denote the support of divisor $\mathcal Q$ (resp. $\mathcal F}$). 

In the following, we say that a hypersurface $G=0$ of degree $t$ in ${\mathbb C}P^6$ is {\emph{generic}} if it does not contain the curve $F$ and it cuts out a divisor on $F$ whose support lives in $V=\mathbb{C}P^1\setminus F^{-1}(Q_5)$. Projective normality  \cite[pp.230-231]{Miranda1995}) of the rational normal curve $F$ warrants the existence of generic hypersurfaces. 

 A generic hyperplane $H=\sum_{i=0}^6 c_i\, a_i=0$ in ${\mathbb C}P^6$ with coordinates $[a_0: \cdots:a_6]$ cuts $\gamma=F(M)$ in a divisor $D_H$ of degree $6$ whose support lies in $V$, while $f_1$ pulls the hyperplane $H=0$ back to a hypersurface of degree $6$ in ${\mathbb C}P^3$ that cuts $g$ in a divisor ${\mathcal D}$ of degree $6k$ by Bezout's theorem. Since $\varphi|_U$ is a covering map of degree $d$ over $V$, the divisor ${\mathcal D}$ contains the pullback divisor ${\mathcal D}_0\triangleq \varphi^\ast(D_H)$ of degree $6d$. Define their difference by ${\mathcal F}$, 
\begin{equation}\label{computationsInterDiv}
{\mathcal F}\triangleq {\mathcal D}-{\mathcal D}_0.
\end{equation}
In the following, we denote the support of a divisor $\bf D$ by $\supp {\bf D}$. 

\begin{remark}\label{rk-diviF}
$\mathcal{F}$ is the fixed part of the intersection divisors of $g$ with the hypersurfaces of degree $6$ obtained by the coordinates of $f_1$ given in \eqref{f}, namely, 
\begin{equation}
\mathcal{F}=\min\limits_{0\leq i\leq 6}\{g^\ast(a_i\circ f_1)\}.
\end{equation}
Moreover, let $G=0$ be a generic hypersurface of degree $t$ in ${\mathbb C}P^6$ not containing the curve $F$. Then 
\begin{equation}\label{General difference between degt divisors}
g^{\ast}(G\circ f_1)=\varphi^{\ast}(F^\ast G)+t\mathcal{F}.
\end{equation}
\end{remark}

 \begin{proposition}\label{charc on 2dim orbit}
 Let $F:\mathbb{C}P^1\rightarrow \mathcal{H}_0^3\subset \mathbb{C}P^6$ be a sextic curve not lying in the closed $2$-dim orbit $\overline{PSL_2\cdot u^5v}$, and $g:M\rightarrow \mathbb{C}P^3$ be the Galois lift of $F$ in the commutative diagram \eqref{diagram2}. Then
 \begin{equation}\label{FlessthenOrEqualQ}
\mathcal{F}\leq \mathcal{Q}.
\end{equation}
%For any $p$ lies in the support of $\mathcal{Q}$, we have $(f_1\circ g) (p)$ belongs to the closed $2$-dim orbit $\overline{PSL_2\cdot u^5v}$. 
{
Moreover, for any given $p\in \supp \mathcal Q$, 
\begin{enumerate}
\item[(1)] if $w(p)$ is not associated with any of the $6$ vertices, then $\ord_p(\mathcal{F})=0$ and $f_1\circ g (p)$ lies in the $1$-dim orbit $PSL_2\cdot u^6$, and
\item[(2)] if $w(p)$ is associated with one of the $6$ vertices, then $\ord_p(\mathcal{F})>0$ and $f_1\circ g (p)$ lies in the $1$-dimensional {\rm(}respectively, $2$-dimensional{\rm)} orbit if and only if $\ord_p(\mathcal{F})<\ord_p(\mathcal{Q})$ {\rm(}respectively, $\ord_p(\mathcal{F})=\ord_p(\mathcal{Q})${\rm)}. 
\end{enumerate}
}
%\begin{enumerate}
%\item[(2.1)] $1$-dim orbit if and only if $\ord_p(\mathcal{F})<\ord_p(\mathcal{Q})${\emph{;}}
%\item[(2.2)] open $2$-dim orbit $PSL_2\cdot u^5v $ if and only if $\ord_p(\mathcal{F})=\ord_p(\mathcal{Q})$.
%\end{enumerate}
 \end{proposition}
 \begin{proof}
 
Recall the quadratic $Q_5$ in \eqref{eQ}. Via $f_1$ in \eqref{f} we have the remarkable $SL_2$-invariant identity 
\begin{equation}\label{remarkable}
Q_5=2a_0a_6 - 2a_1a_5 + 2a_2a_4 - a_3^2 = (ad-bc)^6. 
\end{equation}
Therefore, we derive that the support of $\mathcal{F}$ is contained in that of $\mathcal{Q}$, since the former one can be further determined (see Remark~\ref{rk-diviF}) by setting the coordinate functions zero, i.e., 
\begin{equation}\label{eeQ}
a_i\circ f_1\circ g =0,\quad 0\leq i\leq 6. 
\end{equation}
%where $[a_0:\cdots: a_6]$ are given in \eqref{f}. Indeed, assume $p$ is a point in the support of $\mathcal{F}$ such that one of the functions in 
%\eqref{eeQ} is not zero at $p$. 
%It would follow that $f_1\circ g(p)$, as a well-defined point in ${\mathbb C}P^6$, is one of $F\cap Q_5$ on the one hand, so that it never lives in a generic hyperplane $H=0$ (we choose generic $H$ to avoid $F\cap Q_5$), while on the other it lives in the generic hyperplane $H=0$ since $(f_1\circ g)^{*}(H)$ nullifies $p$, a contradiction.

Suppose that $p\in \supp \mathcal Q$, i.e., $g(p) \in Q_2$. 
By the action of $PSL_2$ on $\mathcal{H}_0^3$, we may assume 
\begin{equation}\label{simple form not free lines}
g(p)=\begin{pmatrix}
1 & w \\
0 & 0
\end{pmatrix}.
\end{equation}

If $w(p)$ is not associated with any of the $6$ vertices, i.e., $w^4\neq 0,1,\infty$, then let $U_p$ be a chart around $p$ with local coordinate $s$ and $s(p)=0$. From \eqref{f} and $\ord_p(c),\ord_p(d)\geq 1,$ %(see \eqref{simple form not free lines})
we obtain $0=\ord_p(a_6)<\min\limits_{0\leq i\leq 5}\{\ord_p(a_i)\}$; thus $f_1\circ g(p)=[0:0:0:0:0:0:1]$ lies in the $1$-dimensional orbit $PSL_2\cdot u^6$, whence
\[\ord_p(\mathcal{F})=\min\limits_{0\leq i\leq 6}\{\ord_p(a_i)\}=\ord_p(a_6)=0<\ord_p(ad-bc)=\ord_p({\mathcal Q}).\]

Next, we assume $w^4=0,1,$ or $\infty$. By the action of $PSL_2$ and the isotropy group of $u^6$ (see Remark \eqref{isotropy group}) on $\mathcal{H}_0^3$, we may assume that 
\begin{equation}\label{simple form of g}
g(p)=\begin{pmatrix}
1 & 0 \\
0 & 0
\end{pmatrix}.
\end{equation}
Let $U_p$ be a chart around $p$ with local coordinate $s$ and $s(p)=0$. By \eqref{f} and the property $\ord_p(b),\ord_p(c),\ord_p(d)\geq 1$ %(see \eqref{simple form of g})
 , we obtain that if $\ord_p(d)\leq \ord_p(b)$, then $\ord_p(a_5)<\ord_p(a_j)$ for any $j\neq 5$, so that $f_1\circ g(p)=[0:0:0:0:0:1:0]$ lies in the open $2$-dimensional orbit, whence
 \[\ord_p(\mathcal{F})=\min\limits_{0\leq i\leq 6}\{\ord_p(a_i)\}=\ord_p(a_5)=\ord_p(d)=\ord_p(ad-bc)=\ord_p({\mathcal Q})>0.\] 
On the other hand, if $\ord_p(d)>\ord_p(b)$, then $\ord_p(a_6)<\ord_p(a_k)$ for $0\leq k\leq 5$, so that $f_1\circ g(p)=[0:0:0:0:0:0:1]$ lies in the $1$-dimensional orbit to yield
 \[\ord_p(\mathcal{F})=\ord_p(a_6)=\ord_p(b)<\ord_p(ad-bc)=\ord_p({\mathcal Q}).\] 
In conclusion, $f_1\circ g(p)$ lies in the open $2$-dimensional orbit if and only if $\ord_p(\mathcal{F})=\ord_p(\mathcal{Q})$.
 \end{proof}

\begin{corollary}\label{Cor4.1}
Assume the same setting as in Proposition {\rm \ref{charc on 2dim orbit}}. We have 
$\deg \varphi\leq \deg g\leq \frac{3}{2}\deg \varphi.$
Moreover, 
\begin{enumerate}
\item[(1)] $\deg g= \deg \varphi$ if and only if $\mathcal{F}=0$.
\item[(2)] $\deg g= \frac{3}{2}\deg \varphi$ if and only if $\mathcal{F}=\mathcal{Q}$.
\item[(3)] If $\deg \varphi=1$, then $\deg g=1$ so that $g(M)$ is a line in $\mathbb{C}P^3$. 
\end{enumerate}
\end{corollary}

\begin{proof} Counting the degree of both sides of \eqref{computationsInterDiv}, by \eqref{FlessthenOrEqualQ} we obtain 
\[ 6\deg g-6\deg \varphi=\deg({\mathcal F})\leq \deg({\mathcal Q})=2\deg g,\] 
which implies $\deg \varphi\leq \deg g\leq \frac{3}{2}\deg\varphi$ with the equality conditions as asserted in (1) and (2). In particular,  if $\deg \varphi=1$, then $\deg g=1$ and $g(M)$ is a line. 
\end{proof}

\vspace{1mm}

\subsection{The Generally Ramified Family}~

%%As  in Theorem~\ref{thm-tangential}

By Theorem \ref{thm-ramichar}, a sextic curve $\gamma$ in $\mathcal{H}^3_0$ is ramified in the sense of harmornic sequences at a point $q$  if and only if the tangent line of $\gamma$ at $q$ lies in $\mathcal{H}^3_0$. An important class of lines in $\mathcal{H}^3_0$ is given by the rulings of the tangent developable surface $\bf{S}$ (i.e., the closed $2$-dimensional $PSL_2$-orbit), which are exactly the tangent lines of the $1$-dimensional orbit $PSL_2\cdot u^6$; in particular, that there is a unique line though $q$ in the $1$-dimensional orbit implies that $\gamma$ is ramified at $q$ if and only if $\gamma$ is tangent to the $1$-dimensional orbit at $q$.  Our investigation of various examples and Galois analysis have prompted the following definition. 

\begin{definition} \label{def-generally}
We say that a sextic curve $\gamma$ in ${\mathcal H}_0^3$ is in the {\bf generally ramified family} if $\gamma$ is ramified {\rm (}as always, in the sense of harmonic sequences{\rm )} at the $1$-dimensional orbit $PSL_2\cdot u^6$ somewhere.
\end{definition}

Now, we give a characterization of tangency of $\gamma$ at the $1$-dimensional orbit in terms of intersection divisors. In the following, we denote the intersection multiplicity at a point $q\in \gamma \cap Q_5$ by $\ord_q(Q_5)$, and we stipulate that $\ord_q(Q_5)=+\infty$ if $\gamma$ lies in $Q_5$.  
 \begin{proposition}\label{mult4 implies tangent at 1 dim orbit}
 Let $F:\mathbb{C}P^1\rightarrow \mathcal{H}_0^3$ %\subset \mathbb{C}P^6$ 
 be a sextic curve. %For any point $q\in F\cap PSL_2\cdot u^6$,  % of the intersection of $F$ with the $1$-dimensional orbit, %$q\in F\cap Q_5$, 
 %the intersection multiplicity, denoted by $\ord_q(Q_5)$, of $F$ and $Q_5$ is greater than or equal to $2$.
%Moreover, 
$F$ is ramified at the $1$-dimensional orbit at $q$ %%, or, equivalently, $F$ is tangent to the $1$-dimensioanl orbit at $q$,
 if and only if $\ord_q(Q_5)\geq 4$.
 \end{proposition}
We defer the proof of this proposition to that of  Proposition \ref{LowTriRam} for the sake of not interrupting the smoothness of exposition. Immediately we obtain the following. 

\begin{corollary}
Let $F:\mathbb{C}P^1\rightarrow \mathcal{H}_0^3$ %\subset \mathbb{C}P^6$ 
be a sextic curve. If either the degree of the covering $\varphi: M\rightarrow {\mathbb C}P^1$ in Theorem~{\rm \ref{lift of curve not in 2dim orbit}} equals $1$, or $F$ lives in the closed $2$-dimensional $PSL_2$-orbit, then $F$ belongs to the generally ramified family.  %, then $F$ belongs to the generally ramified family. 
\end{corollary}
\begin{proof}
When $F$ lives in the closed $2$-dimensional orbit,
the conclusion holds because the curve intersects the $1$-dimensional orbit at a point $q$ by Theorem \ref{lem-case2}, while the fact that $Q_5\equiv 0$ on this curve implies $\ord_q(Q_5)=+\infty$. 

For the other case, it follows from Corollary~\ref{Cor4.1} that the line $g(M)$ cuts $Q_2$ in two points $p_1$ and $p_2$. Since $\mathcal{F}=0$, $p_1$ and $p_2$ are mapped to points on the $1$-dimensional orbit by $f_1\circ g$, at which there must hold $\ord_{p_i}(Q_5)\geq 6$ for $i=1$ or $2$. Then the conclusion follows from Proposition~\ref{mult4 implies tangent at 1 dim orbit}.     
\end{proof}

Henceforth, we assume that $F$ does not lie in the closed $2$-dimensional orbit. 

\begin{lemma}\label{mult by w}
Consider the Galois covering $\varphi:M\rightarrow \mathbb{C}P^1$ with the Galois group $G\subset S_4$ in the same setting as in Theorem {\rm \ref{lift of curve not in 2dim orbit}}. Given $p\in \supp \mathcal{Q}$, denote by $\mult_p(\varphi)$ the multiplicity of $\varphi$ at $p$, i.e., $\varphi: s\mapsto s^{\mult_p(\varphi)}$ for a local uniformizing parameter $s$ with $s(p)=0.$   
\begin{enumerate}
\item[(1)] \!\!If $w(p)$ is not associated with any of %does not belong to 
the $26$ centers of the octahedron,  \!\!\! %$26$ points of \eqref{26}, 
then $\mult_p(\varphi)\!\!=\!\!1$.
\item[(2)]  \!\!If $w(p)$ is associated with one of the $12$ edge centers, %is one of the midpoints of the $12$ edges of the octahedron, 
then $\mult_p(\varphi)=1$ or $2$.
\item[(3)]  \!\!If $w(p)$ is associated with one of the $8$ face centers,  %is one of the centers of the $8$ faces of the octahedron, 
then $\mult_p(\varphi)=1$ or $3$.
\item[(4)]  \!\!If $w(p)$ is associated with one of the $6$ vertices, then $\mult_p(\varphi)=1$ or $2$ or $4$.
\end{enumerate}

\end{lemma}
\begin{proof}

Let $G_p\triangleq\{\sigma\in G~|~\sigma(p)=p\}$ be the stabilizer of $p$. Then by \cite[Proposition 3.1, p.~76; Theorem 3.4, p.~78]{Miranda1995}, we have 
$\mult_p(\varphi)=|G_p|,$
 and $G_p$ is a finite cyclic subgroup of $G$, and hence of $S_4$. Since the non-trivial finite cyclic subgroups of $S_4$ are $C_2,~C_3$, and $C_4$, we infer $1\leq \mult_p(\varphi)=|G_p|\leq 4$. As said before, $G_p$ is trivial when $w(p)$ does not correspond to any of the $26$ centers of the regular octahedron, from which the conclusion in (1) follows. Moreover, with respect to the action of $S^4$ on the octahedron, the stabilizer of the vertex (respectively, edge center, face center) is isomorphic to $C_4$ (respectively, $C_2$, $C_3$), from which the conclusions in (2) $\sim$ (4) follow.  %using the stabilizer of these $26$ centers under the action of $S^4$, one can derive the conclusions in (2) $\sim$ (4). 
%If $\mult_p(\varphi)>1$, then $w(p)$ is fixed by a non-trivial rotation, hence $w(p)$ must lie in one of the axes of rotations, and so it belongs to the $26$ special points. Thus, $G_p=C_2$ (resp. $C_3$) if $w(g(p))$ is an edge midpoint (resp. center of a face). And $G_p=C_2$ or $C_4$, if $w(g(p))$ is a vertex.

\end{proof}
\iffalse
Observe that in the proof of Proposition \ref{charc on 2dim orbit}, the support of $\mathcal{F}$ lies in the six $w$-rulings in $Q_2$ for which $w^4=0,1, \infty$ that correspond to the six vertexes of the regular octahedron. Hence we may rewrite the Proposition \ref{charc on 2dim orbit} as follows.

\begin{corollary}\label{whether in 2dim orbit by w}
Assume the same setting as in Theorem {\rm \ref{lift of curve not in 2dim orbit}}. For any $p$ lying in the support of $\mathcal{Q}$, we have $(f_1\circ g) (p)$ belongs to the closed $2$-dimensional orbit $\overline{PSL_2\cdot u^5v}$. Moreover, 
\begin{enumerate}
\item[(1)] If $w(g(p))$ does not belong to the $6$ vertexes, then $\ord_p(\mathcal{F})=0$ and $(f_1\circ g) (p)$ lies in the $1$-dim orbit $PSL_2\cdot u^6$.
\item[(2)] If $w(g(p))$ is one of the $6$ vertexes, then $\ord_p(\mathcal{F})>0$, and
\begin{enumerate}
\item[(2.1)] $(f_1\circ g) (p)$ lies in the $1$-dimensional orbit if and only if $\ord_p(\mathcal{F})<\ord_p(\mathcal{Q})${\emph{;}}
\item[(2.2)] $(f_1\circ g) (p)$ lies in the open $2$-dimensional orbit $PSL_2\cdot u^5v $ if and only if $\ord_p(\mathcal{F})=\ord_p(\mathcal{Q})$.
\end{enumerate}
\end{enumerate}
\end{corollary}
\fi
Now, we provide sufficient conditions for the curve $F$ to belong to the generally ramified family.

\begin{theorem}\label{when tangent using divisor}
Let $F:\mathbb{C}P^1\rightarrow \mathcal{H}_0^3$ %\subset \mathbb{C}P^6$ 
be a sextic curve which is not contained in the closed $2$-dimensional $PSL_2$-orbit. If one of the following holds, then $F$ belongs to the generally ramified family. 
\begin{enumerate}

\item[(1)] There exists a point $p\in \supp \mathcal Q$ %\setminus \supp \mathcal F$ 
such that $w(p)$ is not associated with any of the $26$ centers of the octahedron, i.e., $w(p)$ does not satisfy \eqref{26}.

\item[(2)] There exists a point $p\in \supp \mathcal Q \setminus \supp \mathcal F$ such that either $\mult_{p}(\varphi)=1$, or $\ord_p(\mathcal{Q})\geq 2$ and $w(p)$ is associated with one of the $12$ edge centers and $8$ face centers. 

\item[(3)] There exists a point $p\in M$ such that $0<\ord_p(\mathcal{F})<\ord_p(\mathcal{Q})$. 
\end{enumerate}
\iffalse
\begin{enumerate}

\item[(1)] If there exists point $p$ lies in the support of $\mathcal{Q}$, and $w(g(p))$ is not one of the $26$ special points of the octahedron, or equivalently not satisfy the equation \eqref{26}.

\item[(2)] If there a point $p$ lies in the support of $\mathcal{Q}$ with $\ord_p(\mathcal{Q})\geq 2$, and $w(g(p))$ belongs to the $12$ edge midpoints and $8$ centers of faces.

\item[(3)] If there exists a point $p$ lies in the support of $\mathcal{F}$, and $\ord_p(\mathcal{F})<\ord_p(\mathcal{Q})$.% or equivalently, $(f_1\circ g)(p)$ belongs to the $1$-dim orbit $PSL_2\cdot u^6$.
\end{enumerate}
\fi
\end{theorem}

\begin{proof}
From \eqref{General difference between degt divisors} and \eqref{remarkable}, we obtain $\varphi^{\ast}(F^\ast Q_5)+2\mathcal{F}=6\mathcal{Q}$; hence, for any point $p\in M$, we have 
\begin{equation}\label{difference from Q5andQ2}
\mult_p(\varphi) \ord_{\varphi(p)}(Q_5)=6\ord_p(\mathcal{Q})-2\ord_p(\mathcal{F})=6\big(\ord_p(\mathcal{Q})-\ord_p(\mathcal{F})\big)+4\ord_p(\mathcal{F}).
\end{equation}
We will use Proposition \ref{charc on 2dim orbit} and Proposition \ref{mult4 implies tangent at 1 dim orbit} to prove this theorem. 

The conclusion for condition (1) follows from $\ord_p(\mathcal Q)>0=\ord_p(\mathcal F)$, and item (1) of Lemma~\ref{mult by w}. % as well as Proposition \ref{charc on 2dim orbit} and Proposition \ref{mult4 implies tangent at 1 dim orbit}. 

Under condition (2), we have $\ord_p(\mathcal F)=0$ while
$$\text{either}\;\ord_p(\mathcal Q)\geq 2\;\text{and}\;\mult_p(\varphi)\leq 3,\quad \text{or}\;\ord_p(\mathcal Q)\geq 1\;\text{and}\;\mult_p(\varphi)=1,$$
where items (2) and (3) of Lemma~\ref{mult by w} are used. Substituting these into \eqref{difference from Q5andQ2}, we obtain $\ord_{\varphi(p)}(Q_5)\geq 4$. Therefore, $F$ is tangent to the $1$-dimensional orbit at $F\circ \varphi(p)$.

Under condition (3), we have 
$$\ord_p(\mathcal F)\geq 1,~~~~~~\ord_p(\mathcal Q)-\ord_p(\mathcal F)\geq 1,$$
which implies $\mult_p(\varphi) \ord_{\varphi(p)}(Q_5)\geq 10$. Note that in this case, $\mult_p(\varphi)=1,2$, or $4$. The conclusion follows from that the minimal integer of the form  in \eqref{difference from Q5andQ2} is $16$ when it is a multiple of $4$ and is $\geq 10$.  
\end{proof}

In contrast to Definition~\ref{def-generally}, we introduce the following definition.  

\begin{definition}
If a sextic curve in ${\mathcal H}_0^3$ is not tangent to the $1$-dimensional orbit $PSL_2\cdot u^6$, then we say that it lies in the {\emph{exceptional transversal family}}.
\end{definition}

By Lemma \ref{mult by w} and Theorem \ref{when tangent using divisor}, %and Lemma 3.6 in \cite[p.~80]{Miranda1995}, 
we obtain the following necessary conditions for the exceptional transversal family.

\begin{proposition}\label{ChacExcepTransFamily}
Let $F:\mathbb{C}P^1\rightarrow \mathcal{H}_0^3$ %\subset \mathbb{C}P^6$ 
be a sextic curve that belongs to the exceptional transversal family. Then for any point $p \in \supp \mathcal{Q}$, %lies in the support of $\mathcal{Q}$, 
there holds that $w(p)$ corresponds to one of the $26$ centers of the octahedron. Moreover, 
\begin{enumerate}
\item[(1)] $\supp \mathcal F= \supp \mathcal Q$ if and only if $\mathcal F= \mathcal Q$, and

\item[(2)] for any given $p\in \supp \mathcal Q \setminus \supp \mathcal F$,  we have $\ord_p(\mathcal Q)=1$, and $w(p)$ is either one of the $12$ edge centers, for which 
$$\mult_p(\varphi)=2, ~~~~~~\ord_{\varphi(p)}(Q_5)=3,$$
or one of the $8$ face centers, for which 
$$\mult_p(\varphi)=3, ~~~~~~\ord_{\varphi(p)}(Q_5)=2.$$
\end{enumerate}
\iffalse
\begin{enumerate}
\item[(1)] If $w(p)$ is one of the $12$ edge midpoints, then $F$ intersect with the $1$-dim orbit $PSL_2\cdot u^6$ at $(f_1\circ g)(p)$ transversally.   Moreover $\ord_p(\mathcal{Q})=1,~\mult_p(\varphi)=2$, and $\ord_{\varphi(p)}(Q_5)=3$, and the fiber $\varphi^{-1}(\varphi(p))$ consists of $\frac{\deg \varphi}{2}$ distinct points.

\item[(2)] If $w(g(p))$ is one of the $8$ centers of faces, then  then $F$ intersect with the $1$-dim orbit $PSL_2\cdot u^6$ at $(f_1\circ g)(p)$ transversally.   Moreover $\ord_p(\mathcal{Q})=1,~\mult_p(\varphi)=3$ and $\ord_{\varphi(p)}(Q_5)=2$, and the fiber $\varphi^{-1}(\varphi(p))$ consists of $\frac{\deg \varphi}{3}$ distinct points.

\item[(3)] If $w(g(p))$ is one of the $6$ vertexes, then $\ord_p(\mathcal{F})=\ord_p(\mathcal{Q})$, and so $(f_1\circ g)(p)$ belongs to the open $2$-dim orbit $PSL_2\cdot u^5v$.  Meanwhile $\ord_{\varphi(p)}(Q_5)=\frac{4\ord_p(\mathcal{Q})}{\mult_p(\varphi)}$, where $\mult_p(\varphi)=1,2$ or $4$, and the fiber $\varphi^{-1}(\varphi(p))$ consists of $\frac{\deg \varphi}{\mult_p(\varphi)}$ distinct points.

\end{enumerate}
\fi
\end{proposition}
Note that for points in item (2) of the preceding proposition, $F$ intersects the $1$-dimensional orbit $PSL_2\cdot u^6$ at $f_1\circ g(p)$ transversally.

\subsection{The Exceptional Transversal Family}\label{sec5.4}~
\label{exceptional}

We now look at the exceptional transversal family in a unified fashion.

Let $G$ be the group of covering transformations of $\varphi:M\rightarrow \mathbb{C}P^1$ in the same setting as in Theorem \ref{lift of curve not in 2dim orbit}. As a subgroup of $S_4$, the Galois group $G$ can only be one of the following subgroups: the trivial group, the cyclic groups $C_i,~2\leq i\leq 4$, the dihedral groups $D_j,~2\leq j\leq 4$, the alternating group $A_4$, and $S_4$ itself. 

When $M=\mathbb{C}P^1$, the above Galois coverings were classified by Klein \cite{Klein1956} as given in Table \ref{Rational Galois Coverings} below. %(see also \cite[p.~53]{JeroenSijsling}).

\begin{table}[htpb]
\begin{tabular}{|c|c|c|c|c|}
\hline
$G$          & $C_i$ & $D_j$ & $A_4$ & $S_4$ \\ \hline
$\varphi(s)$ &   $s^i$    &  $s^j-2+s^{-j}$      & $\frac{(s^3-1)^3}{s^3(s^3+8)^3}$      &  $\frac{(s^8+14s^4+1)^3}{(s(s^4-1))^4}$      \\ \hline
\end{tabular}
\vspace{0.15cm}
\caption{Rational Galois Coverings}\label{Rational Galois Coverings}
\end{table}
\vskip -0.5cm

Given a sextic curve belonging to the exceptional transversal family, we label the points of intersection of this curve and the 1-dimensional orbit as points of type I, and the points of intersection of this curve and the open 2-dimensional orbit as points of type II.

Let $q$ be a point of type I. For each ramified point $p$ over $q=\varphi(p)$, it follows from Proposition~\ref{ChacExcepTransFamily} that $\ord_p(\mathcal Q)=1$, and 
$$\text{ either }~~\mult_p(\varphi)=2\; \text{and}\;\ord_{q}(Q_5)=3,\quad \text{or}~~ \mult_p(\varphi)=3\; \text{and}\;\ord_{q}(Q_5)=2.$$
%By Theorem~\ref{when tangent using divisor} and Proposition~\ref{ChacExcepTransFamily},
%} each singular point $q$ lying in the $1$-dimensional $PSL_2$ orbit, to be labeled points of group 1, has $\#_q=2$ or $3$ and $\ord_p({\mathcal Q})=1$, so that $m_p = 3$ or $2$ for each ramified point $p$ over $q$. 
Let $\Sigma_1$ and $\Sigma_2$ be the number of points assuming $\ord_q(Q_5)=2$ and $3$, respectively. Then 
\begin{equation}\label{Sig}
l\triangleq 2\Sigma_1+3\Sigma_2
\end{equation}
satisfies $0\leq l\leq 12$. It is easy to verify that the total ${\mathcal Q}$-degree for type I is 
\begin{equation}\label{eq-gropu1}
(\deg \varphi/3)\Sigma_1+(\deg \varphi/2)\Sigma_2=l\,\deg \varphi/6, 
\end{equation}
where $\deg \varphi /3$ (respectively, $\deg \varphi/2$) is the number of ramified points $p$ over $q$ with $\mult_p(\varphi)=3$ (respectively, $\mult_p(\varphi)=2$)  \cite[Lemma 3.6, p.~80]{Miranda1995}.

%Each singular point on the $2$-dimensional orbit $PSL_2\cdot u^5v$, to be labeled points of group 2, satisfies $4\ord_p({\mathcal Q})=\mult_p(\varphi)\ord_q(Q_5)$ for each ramified point $p$ over $q$, so that the total ${\mathcal Q}$-order  over $q$ is
Let $q$ be a point of type II. For each ramified point $p$ over $q$, it follows from Proposition~\ref{ChacExcepTransFamily} and \eqref{difference from Q5andQ2} that $4\ord_p({\mathcal Q})=\mult_p(\varphi)\ord_q(Q_5)$, which implies that the total ${\mathcal Q}$-degree over $q$ is    
\begin{equation}\label{QqQQQ}
\sum_{p\in \varphi^{-1}(q)} \ord_p(\mathcal{Q})=\sum_{p\in \varphi^{-1}(q)}\mult_p(\varphi)\,\ord_q(Q_5)/4=\deg \varphi \,\ord_q(Q_5) /4. %\quad \text{for group 2}. 
\end{equation}
Therefore the total ${\mathcal Q}$-degree for type II is
\begin{equation}\label{Qq}
\sum_{q\,\text{of type II}} \deg \varphi \,\ord_q(Q_5) /4= \left(\sum_{q\,\text{of type II}}\ord_q(Q_5)\right)\deg \varphi/4=(12-l)\deg \varphi/4.
\end{equation}
Hence 
\[2\deg(g) =\deg(\mathcal{Q})=l\deg \varphi/6+(12-l)\deg \varphi/4=(36-l)\deg \varphi/12,\]
which implies $\deg g=(36-l)\deg \varphi/24$.

To illustrate, consider $\deg \varphi=2$, for which $\deg g=(36-l)\deg \varphi/24$ gives $l=0$ or $12$. 

If $l=0$ then $\deg g=3$; all points $q$ of $\supp{\mathcal{Q}}$ live in the $2$-dimensional orbit $PSL_2\cdot u^5v$. We seek to find examples where the genus of $M$ is zero. The Riemann-Hurwitz formula dictates that there be exactly two points $q_1$ and $q_2$ of type II over each of which there sits a single ramified point $p_1$ and $p_2$, respectively, with ramification index $1$ ($\mult_{p_i}(\varphi)=2$), so that the formula $4\ord_p({\mathcal Q})=\mult_p(\varphi)\ord_q(Q_5)$ gives that $\mult_{q_i}(Q_5)$ are multiples of $2$. There may exist other points $q_3,\cdots, q_m$ of type II over each of which there sit two ramified points $p_{j1}$ and $p_{j2}, 3\leq j\leq m,$ each with ramification index 0 ($\mult_{p_{jk}}(\varphi)=1$) so that $\ord_{q_j}(Q_5)$ is a multiple of $4$ for $3\leq j\leq m$. In the most generic situation, $\ord_{q_i}(Q_5)=2$ for $i=1,2$ and $\ord_{q_j}(Q_5)=4$ for $3\leq j\leq m$, for which we have the constraint
$$
12=2+2+4(m-2), \quad \text{so}\; m=4.
$$
In other words, there are four points $q_1, \cdots,q_4$ of type II, where the Galois covering is unramified over $q_3$ and $q_4$.

Indeed, up to a $PSL_2$-transformation on the left and an isotropy group action on the right of the Galois lift, a detailed Galois analysis, which we will report elsewhere, proves that this is the only possibility with the Galois lift $g(s)$ given by
$${\small
\aligned
&g=[1:w:x:yw],\\
&x=-\sqrt{-1}(t^2 - 1)/((-t + \sqrt{-1})s^2 + t^3 - \sqrt{-1}),\quad y=(-\sqrt{-1} t^2+ \sqrt{-1}s^2 - t s^2+ t)/(t^3 - t),\\
&w=(t^3 - t)/(s((-t + \sqrt{-1})s^2 + t^3 - \sqrt{-1})),
\endaligned}
$$
whenever $t$ is not a zero of a certain polynomial of a large degree which we do not record here. The Galois covering $\varphi$ is $z=s^2$ corresponding to the group $C_2$, $q_1$ and $q_2$ are $z=0$ and $z=\infty$, and $q_3$ and $q_4$ are $z=1$ and $z=t^4$.

If $l=12$ then $\deg g=\deg \varphi=2$. Since $\gamma$ is in the exceptional transversal family, each point $q$ of type I has $\ord_q(Q_5)=3$ (since $\mult_p(\varphi)=2$ as $\deg \varphi=2$). As a result, there are four points $q_1,\cdots,q_4$ of type I over each of which there sits a single ramified point $p_1,\cdots,p_4$, respectively, each with ramification index $1$. The Riemann-Hurwitz formula implies that there do not exist any such Galois lifts $g$ with genus zero.

Suffices it to say that a detailed Galois analysis proves that when $\deg g=\deg\varphi=2$, there are two $1$-parameter classes of Galois lifts in the generally ramified family.

As another example,  let us find a procedure to determine the structure of $M$ with genus zero, for which the Riemann-Hurwitz formula gives
$$
-2\geq -2\deg\varphi+ 2(\deg\varphi/3)\Sigma_1 + (\deg\varphi/2)\Sigma_2,
$$
so that $(2\Sigma_1/3 +\Sigma_2 /2 -2)\deg\varphi\leq -2$ from which we determine, since $2\Sigma_1/3 +\Sigma_2 /2-2< 0$, an even $l$ to make sure $\deg g=(36-l)\deg\varphi/24$ is an integer, which comes down to 
$$
(l,\Sigma_1,\Sigma_2)=(8,1,2), (6,0,2), (4,2,0),(2,1,0).
$$
If we set $\deg\varphi\geq 4$, then $(2\Sigma_1/3 +\Sigma_2 /2 -2)\deg\varphi\leq -2$ gives $2\Sigma_1/3 +\Sigma_2 /2\leq 3/2$, from which we narrow it down to 
$$
(l,\Sigma_1,\Sigma_2)=(8,1,2), (6,0,2), (2,1,0).
$$
We choose $(8,1,2)$ to find an example. Since $l=8$ for group 1, whose structural constants $\ord_q(Q_5)$ and $\ord_p(\varphi)$ are known to leave the relatively small number $4$ for group 2, we calculate
$$
(l,\deg \varphi,\deg g)=(8,4,6),\quad (8,6, 7), \quad (8,12,14), \quad (8,24,28).
$$
If we seek to find an example with an irreducible $p(z,x)$ so that $\deg\varphi\geq 6$, we should start with $(8,6,7)$, where the genus of $M$ is zero.

In more details, since $\Sigma_1=1$ and $\Sigma_2=2$, we have three  points $q_1, q_2, q_3$ of type I such that $\ord_{q_1}(Q_5)=2$ and $\ord_{q_2}(Q_5)=\ord_{q_3}(Q_5)=3$, where each ramified point $p_{1j}$ sitting over $q_1$ has $\mult_{p_{1j}}(\varphi)=3$ while each ramified point $p_{2l}, p_{3s}$ over $q_2$ and $q_3$ has $\mult_{p_{2l}}(\varphi)=\mult_{p_{3s}}(\varphi)=2$. Therefore, there are two ramified points over $q_1$ and three ramified points over each of $q_2$ and $q_3$.

Switching to the points of type II, since 
$$
-2\deg \varphi+(2 \deg \varphi/3)\Sigma_1+(\deg \varphi/2)\Sigma_2=-2
$$
already verifies the Riemann-Hurwitz formula, we see that all points of type II are unramified. We have at most 4 such kind of points. To make \eqref{QqQQQ} an integer, for such a point $\tilde{q}$ in the formula, there must hold $\ord_{\tilde{q}}(Q_5)\geq 4$ since $\mult_{\tilde{p}}(\varphi)=1$ for each ramified point $\tilde{p}$ over $\tilde{q}$. But then this means that $\tilde{q}$ is the only such kind of point since the total $Q_5$-degree for type II is 4. Therefore, there are six ramified points sitting over $\tilde{q}$ each with ramification index 0.

Indeed, a detailed Galois analysis in the case of genus zero proves that this is the only possibility (up to left $PSL_2$ and right isotropy actions):
$${\small
\aligned
&g=[a:b:c:d],\\
&a=((\sqrt{3}\sqrt{-1} - \sqrt{-1} - (1 + \sqrt{-1})b_3s^4)\sqrt{2})/2 - (((-1 + \sqrt{-1})b_3\sqrt{3} - 2s^4 + (1 - \sqrt{-1})b_3)s^3)/2,\\
&b=((((-1 + \sqrt{-1})\sqrt{3} + 2\sqrt{-1}b_3s^4 - 1 + \sqrt{-1})\sqrt{2} - 2(-b_3\sqrt{3} + (1 +\sqrt{-1})s^4 - b_3)s^3)/(2 + 2\sqrt{3}),\\
&c=-(((1 - \sqrt{-1}) + (1 + \sqrt{-1})(s^2 + d_3)s^3\sqrt{2}+ (1 - \sqrt{-1})(-d_3s^2 + 1)\sqrt{3}+ (1 - \sqrt{-1})d_3s^2)s)/2,\\
&d=-((((s^2 - d_3)\sqrt{3}+ s^2 + d_3)s^3\sqrt{2}\sqrt{-1} - 2d_3s^2 - 2)s)/2,\quad \text{where},
\endaligned
}
$$
%%where
$$
\aligned
&b_3=(-1/6 -\sqrt{-1}/6)(\sqrt{-6} t^6 + \sqrt{3} d_3 t^3 - 3d_3 t^3 + \sqrt{-6}) \sqrt{3}/t^3,\\
&d_3=-(\sqrt{-6}t^5 -\sqrt{-2}t^5 +2)/[t^2(\sqrt{-6}t-\sqrt{-2}t+2\sqrt{3}-4)].
\endaligned
$$
Moreover, $q_1$ corresponds to $z=\infty$, $q_2$ corresponds to $z=0$, $q_3$ corresponds to $z=-4$, and $\tilde{q}$ corresponds to $z=t^3-2+1/t^3$. No Galois lifts with $k=6,8,9$ exist. Here, the Galois covering $\varphi$ is $z=s^3-2+1/s^3$ with the Dihedral group $D_3$.

To end this subsection, due to its length, we only summarize our classification of Galois covering when $M=\mathbb{C}P^1$ in Table~\ref{Classification of Rational Galois Coverings.} below. %After a detailed Galois analysis, we classify the Galois covering when $M=\mathbb{C}P^1$, see Table \ref{Classification of Rational Galois Coverings.}.

{\tiny\begin{center}
\begin{table}[htpb]
\begin{tabular}{|c|c|c|c|c|}
\hline
\multirow{2}{*}{$\deg\varphi$} & \multirow{2}{*}{$G$} & \multirow{2}{*}{$\deg g$} & \multicolumn{2}{c|}{Dimension of the Moduli Spaces}             \\ \cline{4-5} 
                               &                      &                           & \multicolumn{1}{c|}{Generally Ramified Family} & Exceptional Transversal Family \\ \hline
\multirow{2}{*}{$2$}  & \multirow{2}{*}{$C_2$} & $2$                 & $1$                       & $\varnothing$                           \\ \cline{3-5} 
                      &                        & $3$                 & $\varnothing$                    & $1$                            \\ \hline
\multirow{2}{*}{$3$}  & \multirow{2}{*}{$C_3$} & $3$                 & $1$                       &$\varnothing$ \\ \cline{3-5} 
                      &                        & $4$                 & $\varnothing$                      & $1$                            \\ \hline
\multirow{4}{*}{$4$}  & \multirow{2}{*}{$C_4$} & $4,6$               & $\varnothing$                  & $\varnothing$                           \\ \cline{3-5} 
                      &                        & $5$                 & $2$                       & $\varnothing$                         \\ \cline{2-5} 
                      & \multirow{2}{*}{$D_2$} & $4,6$               & $\varnothing$                      & $\varnothing$ \\ \cline{3-5} 
                      &                        & $5$                 & $2$                       & $1$                            \\ \hline
\multirow{2}{*}{$6$}  & \multirow{2}{*}{$D_3$} & $6,8,9$             & $\varnothing$                      & $\varnothing$                           \\ \cline{3-5} 
                      &                        & $7$                 & $\varnothing$                     & $1$                            \\ \hline
\multirow{3}{*}{$8$}  & \multirow{3}{*}{$D_4$} & $8,10,12$           & $\varnothing$                      & $\varnothing$                          \\ \cline{3-5} 
                      &                        & $9$                 & $1$                       & $\varnothing$                          \\ \cline{3-5} 
                      &                        & $11$                & $\varnothing$                     & $1$                            \\ \hline
\multirow{3}{*}{$12$} & \multirow{3}{*}{$A_4$} & $12,14,16 \sim 18$               & $\varnothing$                       & $\varnothing$                           \\ \cline{3-5} 
                      &                        &    $13$    & $1$                     &$\varnothing$                           \\ \cline{3-5} 
                      &                        & $15$                & $0$                  & $\varnothing$                          \\ \hline
\multirow{4}{*}{$24$} & \multirow{4}{*}{$S_4$} & $24,26 \sim 28,30,32 \sim 36$ & $\varnothing$                     & $\varnothing$                          \\ \cline{3-5} 
                      &                        & $25$                & $2$                       &$\varnothing$                          \\ \cline{3-5} 
                      &                        & $29$                & $1$                       & $\varnothing$                           \\ \cline{3-5} 
                      &                        & $31$                & $\varnothing$                     & $1$                            \\ \hline
\end{tabular}
\vspace{0.15cm}
\caption{Classification of Rational Galois Coverings.}\label{Classification of Rational Galois Coverings.}
\end{table}
\end{center}}
\vskip -0.5cm 

In particular, there are at most finitely many constantly curved sextic curves $\subset {\mathcal H}_0^3$ on the list that belong to the exceptional transversal family, by checking total unramification encountered in Section \ref{Sec4}.
In view of the above examples and classification, it is tempting to suggest that a constantly curved  holomorphic $2$-sphere in $G(2,5)$, which differs from a sextic curve in the exceptional transversal family by a $GL(5,{\mathbb C})$-automorphism, be nongeneric among all constantly curved holomorphic $2$-spheres of degree $6$.

On the other hand, the situation in the generally ramified family is clear-cut. We will show in the next section that
a constantly curved holomorphic $2$-sphere of degree $6$ in $G(2,5)$, which differs from a sextic curve $\gamma$ in the generally ramified family by a $GL(5,{\mathbb C})$-transformation, is such that the $6$-plane $\bf L$ it spans in ${\mathbb C}P^9$ differs from that spanned by the standard Veronese curve \eqref{eq-standard} only by a diagonal matrix in $GL(5,{\mathbb C})$.

 \section{Generally ramified holomorphic $2$-spheres of degree $6$ in $G(2,5)$}\label{sextic}
Thanks to the discussion in Section~\ref{sec-para}, we say that a holomorphic $2$-sphere of degree $6$ in $G(2,5)$ is {\em generally ramified} if it is projectively equivalent to a sextic curve in $\mathcal{H}^3_0$ belonging to the generally ramified family.  %%ramified at the $1$-dimensional orbit at some point. 
In this section, we will first give a useful parameterization to such kind of $2$-spheres, then employ it to investigate such $2$-spheres of constant curvature. We will %%focus on those generic codimension 3 linear sections ${\mathcal H}^3$, thanks to their rich structures, to 
show that a generally ramified constantly curved holomorphic $2$-sphere of degree $6$ can only live in the Fano $3$-folds ${\mathcal H}^3$ that differ from the standard $\mathcal{H}^3_0$ by a diagonal transformation in $GL(5,{\mathbb C}^5)$, up to $U(5)$-equivalence. 

\begin{definition}\label{defn2} By the {\bf diagonal family} we mean constantly curved holomorphic $2$-spheres of degree $6$ in $G(2,5)$ parameterized as follows{\rm :} %in Case {\rm (1)} of Proposition {\rm \ref{parametrization prop}}, %{\em (}see also \eqref{final parametrization}{\em )}, 
\begin{equation}\label{GenericG}
%G\triangleq 
\diag(a_{00},\cdots,a_{44})\cdot (E_0,E_1,\ldots,E_6)\, \diag\{\omega_0,\omega_1,\ldots,\omega_6\} \, Z_6(z), 
\end{equation}
where $\{E_0,\ldots,E_6\}$ is the orthonormal basis of $V_6$ defined in \eqref{basisOfV6}, and $Z_6(z)$ is the Veronese $2$-sphere in \eqref{Veronse sphere}.  %$A\triangleq \diag(a_{00},\cdots,a_{44})$ is a diagonal matrix and the columns of are mutually orthogonal and all of unit length.
\end{definition}
\iffalse
\begin{remark}
The holomorphic $2$-sphere parameterized as in \eqref{GenericG} lives in $G(2,5)$, if and only if, 
{\small\begin{equation}\label{perturbed}
\aligned
&\omega_0\omega_4-4\omega_1\omega_3+3\omega_2^2=0,\quad \omega_0\omega_5-3\omega_1\omega_4+2\omega_2\omega_3=0,\quad\omega_0\omega_6-9\omega_2\omega_4+\\
&8\omega_3^2=0,\quad 
\omega_2\omega_6-4\omega_3\omega_5+3\omega_4^2=0,\quad
\omega_1\omega_6-3\omega_2\omega_5+2\omega_3\omega_4=0.
\endaligned
\end{equation}}
%and $\omega_i\neq 0,~i=0,\ldots,6$.
\vskip -0.2cm
\end{remark}
\fi
The following is the main result of this section. 
\begin{theorem}\label{thm-tangential}
    Let $\gamma:\mathbb{C}P^1\rightarrow G(2,5) $ be a generally ramified holomorphic $2$-sphere of degree $6$. If $\gamma$ is of constant curvature, then $\gamma$ belongs the the diagonal family. %, and is ramified at two distinct points with multiplicities at least $2$. 
\end{theorem}
\iffalse
\begin{definition}
We call a rational normal curve $\gamma:\mathbb{C}P^1\rightarrow G(2,5)$ of degree $6$ belongs to the \textbf{Lower-Triangular Family}, if under some suitable parameter $z$ of $\mathbb{C}P^1$, the curve is written as 
\begin{equation}\label{lower-triangular family}
\gamma(z)=A\cdot (E_0,\ldots,E_6)\cdot L \cdot \begin{pmatrix}
1 & \sqrt{6}z & \cdots & z^6
\end{pmatrix}^T,
\end{equation}
where $L=(L_{ij})_{0\leq i,j\leq 6}$ is a non-degenerate lower-triangular matrix, and $A\in GL(5,\mathbb{C})$. Moreover, if $\gamma$ lies in the standard fano 3-fold $\mathcal{H}_0^3$, then $A\in SL_2$.
\end{definition}
\fi
\subsection{Sextic curves in $\mathcal{H}^3_0$ ramified at the $1$-dimensional orbit}
\begin{proposition}\label{LowTriRam}
Let $\gamma:\mathbb{C}P^1\rightarrow \mathcal{H}_0^3$ be a sextic curve, and let $p$ be a point of the $1$-dimensional orbit.\\
{\rm (1)} $\gamma$ is ramified at $p$, if and only if, up to a transformation in $SL(2,\mathbb{C})$, $\gamma $ can be parameterized as 
\begin{equation}\label{eq-parameter}
\gamma(z)=L  \begin{pmatrix}
1 & z & z^2 &  \cdots & z^6,
\end{pmatrix}^t,\quad\text{where}
\end{equation}
\begin{equation}\label{L hard}
L=\begin{pmatrix}
L_{00} & L_{01} & L_{02} & 0 &0 & 0 & 0\\
L_{10} & L_{11} & L_{12} & 0 & 0 & 0 & 0\\
L_{20} & L_{21} & L_{22} & L_{23} & 0 &0 & 0\\
L_{30} & L_{31} & L_{32} & L_{33} & L_{34} & 0 & 0\\
L_{40} & L_{41} & L_{42} & L_{43} & L_{44} & 0 & 0\\
L_{50} & L_{51} & L_{52} & L_{53} & L_{54} & L_{55} & 0\\
L_{60} & L_{61} & L_{62} & L_{63} & L_{64} & L_{65} & 1\\
\end{pmatrix},  
\end{equation}
 if and only if, the vanishing order of $Q_5$ restricted on ${\gamma}$ at $p$ is no less than $4$. \\
{\rm (2)} $\gamma$ is ramified at $p$ with multiplicity no less than $2$, if and only if,  one of $L_{02},L_{23},L_{34}$ in \eqref{L hard} vanishes, if and only if, $L$ is lower-triangular, if and only if, the vanishing order of $Q_5$ restricted on $\gamma$ at $p$ is no less than $6$. 
\end{proposition}
\begin{proof}
If $\gamma$ can be parameterized as \eqref{eq-parameter} with $L$ taking the form of \eqref{L hard}, then it is easily checked that $\gamma$ is ramified at $p$, which has multiplicity $2$ if $L$ is lower-triangular. 

Next, we use the transvectant characterization of $\mathcal{H}^3_0$ to prove the reverse part. %, assume $\gamma$ is ramified at $p$. 
Choose a coordinate $z$ on $\mathbb{C}P^1$ such that $\gamma(\infty)=p$. Note that by applying a transformation in $SL(2,C)$, we can assume $p=v^6$. 
Let $\{l_0, l_1, \cdots, l_6\}$ be the columns of $L$, and set $L_{ij}=\frac{\partial^6 l_j}{\partial v^i u^{6-i}}$. Then we have $$\gamma=\sum_{j=0}^6z^jl_j,\quad l_6=v^6.$$
%Moreover, the ramification condition is equivalent to say that $l_1\in \mathcal{H}^3_0$. 
Assume $\gamma$ is ramified at $p$. By Theorem~\ref{thm-ramichar}, we know the line spanned by $l_6$ and $l_5$ lies in $\mathcal{H}^3_0$. It is well-known that the only line passing through $v^6$ is given by $v^6+t u v^5$. Therefore, we have 
$l_5=\alpha v^6 +\beta uv^5,$
with $\beta\neq 0$.

In terms of the transvectant characterization (see Proposition \ref{orbits defined by transvectant}), $\gamma$ lying in $\mathcal{H}^3_0$ is equivalent to saying $(\gamma,\gamma)_4=0$, which implies, for any $0\leq j \leq 12$, that we have 
\begin{equation}\label{eq-4thtrans}
\sum_{r+s=j} (l_r,l_s)_4=0.
\end{equation}

 In the following, we use the symbol "$*$" to denote some unimportant nonzero constants.

 Take $j=10$ in \eqref{eq-4thtrans}. Since $(l_5, l_5)_4=0$, we have 
 %$$0=(l_6,l_4)_4\approx \frac{\partial^4 l_6}{\partial v^4}\frac{\partial^4 l_4}{\partial u^4}\approx v^2\frac{\partial^4 l_4}{\partial u^4},$$
 $$0=(l_6,l_4)_4=*\,\frac{\partial^4 l_6}{\partial v^4}\frac{\partial^4 l_4}{\partial u^4}=*\, v^2\frac{\partial^4 l_4}{\partial u^4}.$$
 %where the symbol "$\approx$" means equalling up to a constant coefficient. 
%where the symbol "$*$" is used to denote some unimportant nonzero constants. 
It follows that $\frac{\partial^4 l_4}{\partial u^4}=0$. 

 Taking $j=9$ in \eqref{eq-4thtrans}, we have 
 $$0=(l_5,l_4)_4+(l_6,l_3)_4= \beta(uv^5,l_4)_4+(l_6,l_3)_4 %\approx %\frac{\partial^4 l_6}{\partial v^4}\frac{\partial^4 l_4}{\partial u^4}
 = *\,v^2\frac{\partial^4 l_4}{\partial u^3\partial v}+*\,v^2\frac{\partial^4 l_3}{\partial u^4},$$
% where the symbol "$\approx$" means equalling up to a nonzero constant coefficient. 
where we have used $\frac{\partial^4 l_4}{\partial u^4}=0$. 
It follows that $\frac{\partial^5 l_3}{\partial u^5}=*\,
\frac{\partial^5 l_4}{\partial u^4\partial v}=0$.
%Therefore, we have proved that 

From 
\begin{equation}\label{eq-l6543}
l_6=v^6,~~~ l_5=\alpha v^6 +\beta uv^5,~~~ \frac{\partial^4 l_4}{\partial u^4}=0,~~~ \frac{\partial^5 l_3}{\partial u^5}=0,
\end{equation}
we can derive that $L$ takes the form as in \eqref{L hard}. 

To calculate the vanishing order of $Q_5|_{\gamma}$ at $p$, we use the the transvectant characterization 
\begin{equation}\label{eq-Q5}
Q_5|_{\gamma}=(\gamma,\gamma)_6=\sum_{j=0}^{12}z^j\sum_{r+s=j} (l_j,l_k)_6.
\end{equation}
It follows from \eqref{eq-l6543} that 
$$(l_6,l_5)_6=0,~~~(l_6,l_4)_6=*\, \frac{\partial^6 l_4}{\partial u^6}=0,~~~(l_6,l_3)_6=*\, \frac{\partial^6 l_3}{\partial u^6}=0,~~~(l_5,l_5)_6=0,~~~(l_5,l_4)_6=*\, \frac{\partial^6 l_4}{\partial u^5 \partial v}=0.$$
\iffalse
It is easy to verify that 
$$(l_6,l_5)_6=*\, \frac{\partial^6 l_5}{\partial u^6},~~~(l_6,l_4)_6=*\, \frac{\partial^6 l_4}{\partial u^6},~~~(l_6,l_3)_6=*\, \frac{\partial^6 l_3}{\partial u^6},~~~ $$
\fi
Therefore, we have $\deg(Q_5|_{\gamma})\leq 8$, which implies the vanishing order of $Q_5|_{\gamma}$ at $\gamma(\infty)=p$ is no less than $4$. 

Conversely, assume the vanishing order of $Q_5|_{\gamma}$ at $\gamma(\infty)=p$ is no less than $4$. Then we have 
\begin{align}
    &2(l_6, l_4)_6+(l_5,l_5)_6=0,\label{eq-j=10-6}\\
    &2(l_6, l_4)_4+(l_5,l_5)_4=0,\label{eq-j=10-4}\\
    &(l_6, l_3)_4+(l_5,l_4)_4=0,\label{eq-j=9-6}\\
    &(l_6, l_3)_6+(l_5,l_4)_6=0.\label{eq-j=9-4}
\end{align}
It follows from $(l_6,l_5)_4=0$ that $\frac{\partial^4 l_5}{\partial u^4}=0$. 
By comparing \eqref{eq-j=10-6} with the second derivative with respect to $u$ on \eqref{eq-j=10-4}, it is easy to derive  
%$\frac{\partial^6 l_4}{\partial u^6}=0$, and then 
$\frac{\partial^3 l_5}{\partial u^3}=0$. Then substituting this into the first derivative with respect to $u$ of \eqref{eq-j=10-4}, we obtain $\frac{\partial^5 l_4}{\partial u^5}=0$. By comparing \eqref{eq-j=9-6} with the second derivative with respect to $u$ of \eqref{eq-j=9-4}, it is easy to derive 
%$\frac{\partial^6 l_4}{\partial u^6}=0$, and then 
$\frac{\partial^6 l_3}{\partial u^6}=0$. Substituting this into \eqref{eq-j=9-6} and combining \eqref{eq-j=10-4}, we have $\frac{\partial^2 l_5}{\partial u^2}=0$ and $\frac{\partial^4 l_4}{\partial u^4}=0$. Finally, by taking the second derivative with respect to $u$ on both sides of 
$0=2(l_5, l_3)_4+(l_4,l_4)_4,$
we derive $\frac{\partial^5 l_3}{\partial u^5}=0$. Therefore $L$ has  the form of \eqref{L hard}, and is ramified at $p$.  

%Now we assume the multiplicity of $\gamma$ at the ramified point $p$ is no less than $2$, which 
In fact, that the multiplicity of $\gamma$ at the ramification point $p$ is no less than $2$ can be characterized by one more equation that  
$(l_5, l_4)_4=0.$
It follows from $\frac{\partial^4 l_4}{\partial u^4}=0$ that 
$$(l_5, l_4)_4=(uv^5, l_4)_4=*\,\frac{\partial^4 l_4}{\partial u^3\partial v}.$$ 
Therefore, that $\gamma$ is ramified at $p$ with multiplicity no less than $2$ is equivalent to saying that $L$ takes the form as in \eqref{L hard} and $\frac{\partial^3 l_4}{\partial u^3}=0$, i.e., $L_{34}=0$. 
\iffalse
It follows from $\frac{\partial^4 l_4}{\partial u^4}=0$ that 
$$0=(l_5, l_4)_4=(uv^5, l_4)_4=*\,\frac{\partial^4 l_4}{\partial u^3\partial v},$$ 
which implies $\frac{\partial^3 l_4}{\partial u^3}=0$.  
\fi

Taking $j=9$ in \eqref{eq-4thtrans}, we have 
 $$-(l_5, l_4)_4=(l_6,l_3)_4
 =*\,v^2\frac{\partial^4 l_3}{\partial u^4}.$$
 Therefore, $\frac{\partial^3 l_4}{\partial u^3}=0$ is equivalent to $\frac{\partial^4 l_3}{\partial u^4}=0$, i.e., $L_{23}=0$ in \eqref{L hard}.   

Choosing $j=8$ in \eqref{eq-4thtrans}, it follows from \eqref{eq-l6543} that  
\begin{equation}\label{eq-j=8}
\begin{split}
    0=&2(l_6,l_2)_4+2(l_5,l_3)_4+(l_4,l_4)_4\\
 =&*\,v^2\frac{\partial^4 l_2}{\partial u^4}+*\,v^2\frac{\partial^4 l_3}{\partial u^3\partial v}+*\,uv\frac{\partial^4 l_3}{\partial u^4}+*\,\frac{\partial^4 l_4}{\partial u^3 \partial v}\frac{\partial^4 l_4}{\partial u \partial v^3}+*\,\frac{\partial^4 l_4}{\partial u^2 \partial v^2}\frac{\partial^4 l_4}{\partial u^2 \partial v^2}.
 \end{split}
\end{equation}
Taking the second partial derivative with respect to $u$ on both sides, we obtain 
$\frac{\partial^6 l_2}{\partial u^6}=*\, L_{23}^2,$
which implies that $L_{23}=0$ is equivalent to $L_{02}=0$.  

Next, we prove that $L_{34}=0$ if and only if $L$ is lower-triangular, i.e., that the following equations hold,  
\begin{equation}\label{eq-l654321}
l_6=v^6,~~~ l_5=\alpha v^6 +\beta uv^5,~~~ \frac{\partial^3 l_4}{\partial u^3}=0,~~~ \frac{\partial^4 l_3}{\partial u^4}=0, ~~~ \frac{\partial^5 l_2}{\partial u^5}=0, ~~~ \frac{\partial^6 l_1}{\partial u^6}=0. 
\end{equation} 
Note that only the last two equations need to be verified. The second to last follows from taking the partial derivative with respect to $u$ on both sides of \eqref{eq-j=8}.  
\iffalse
 $$0=2(l_6,l_2)_4+2(l_5,l_3)_4+(l_4,l_4)_4
 =*\,v^2\frac{\partial^4 l_2}{\partial u^4}+*\,v^2\frac{\partial^4 l_3}{\partial u^3\partial v}+*\,\frac{\partial^4 l_4}{\partial u^2 \partial v^2}\frac{\partial^4 l_4}{\partial u^2 \partial v^2}.$$
The partial derivative with respect to $u$ on both hand sides implies that $\frac{\partial^5 l_2}{\partial u^5}=0$. 
\fi
Taking $j=7$ in \eqref{eq-4thtrans}, we have 
 \begin{equation*}
 \begin{split}
 0&=(l_6,l_1)_4+(l_5,l_2)_4+(l_4,l_3)_4\\
 &=*\,v^2\frac{\partial^4 l_1}{\partial u^4}+*\,v^2\frac{\partial^4 l_2}{\partial u^4}+*\,v^2\frac{\partial^4 l_2}{\partial u^3\partial v}+*\,\frac{\partial^4 l_4}{\partial u^2 \partial v^2}\frac{\partial^4 l_3}{\partial u^2 \partial v^2}+*\,\frac{\partial^4 l_3}{\partial u^3 \partial v}\frac{\partial^4 l_4}{\partial u \partial v^3}.
 \end{split}
 \end{equation*}
The second partial derivative with respect to $u$ on both sides implies that $\frac{\partial^6 l_1}{\partial u^6}=0$. 

%Thus the claim \eqref{eq-l654321} has been proved. 
Similar to the discussion in the first part, we can derive that the vanishing order of $Q_5|_{\gamma}$ at $\gamma(\infty)=p$ is no less than $6$, and the reverse part is also true. 
\end{proof}

The following technical lemma entailing ramification will be used in the proof of Theorem~\ref{thm-tangential}. %Before going into the investigation of the condition of constant curvature, we introduce a technical lemma characterizing
It characterizes when the lower-triangular matrix $L$ is diagonal. 

\begin{lemma}\label{lower triangular technique lemma}
Let $\gamma(z)= L\, Z_6(z)$ be a rational normal curve of degree $6$ in $\mathcal{H}_{0}^3$, with $L$ being lower-triangular and $L_{21}=0$. If $\gamma(z)$ is also ramified at $z=0$ with multiplicity no less than $2$, then $\gamma(0)$ lies in the closed $2$-dimensional orbit. Moreover, the following are equivalent.
$$
(1)\; L\;\text{ is diagonal{\emph.}} \quad (2)\; \gamma(0)\;\text{lies in the 1-dimensional orbit{\emph.}}\quad (3)\; L_{10}=0{\emph.} \quad (4)\; L_{65}=0{\emph .}
$$

\iffalse

\begin{enumerate}
\item[(1)] $L$ is diagonal{\emph,}
\item[(2)] $\gamma(0)$ lies in the $1$-dim orbit{\emph,}
\item[(3)] $L_{10}=0${\emph,}
\item[(4)] $L_{65}=0${\emph .}
\end{enumerate}

\fi

\end{lemma}
\begin{proof}
We continue to use the notation given in the proof of the preceding proposition. Write $a_i=\sum\limits_{j=0}^6\sqrt{\tbinom{6}{j}}L_{ij}z^j,~0\leq i\leq 6$. Then  $L$ is lower-triangular if and only if $\deg a_i=i,~0\leq i\leq 6$. Note that $[a_0: a_1: \cdots: a_6]$ is exactly the coordinates of $\gamma$, and $a_0$ is a constant. This implies that $a_4,a_5$ and $a_6$ can be solved as polynomials of $a_1,a_2$ and $a_3$ as in \eqref{a4,a5,a6graph1}.   

\iffalse

\begin{align}\label{a4,a5,a6graph}
\begin{split}
a_4&=\frac{\sqrt{2}a_1a_3-\sqrt{\frac{3}{5}}a_2^2}{a_0},\\
a_5&=\frac{\sqrt{\frac{6}{5}}a_1^2a_3-\frac{3}{5}a_1a_2^2}{a_0^2}-\frac{\sqrt{2}a_2a_3}{5a_0},\\
a_6&=\frac{3\sqrt{2}a_1a_2a_3-3\sqrt{\frac{3}{5}}a_2^3}{5a_0^2}-\frac{2a_3^2}{5a_0}.
\end{split}
\end{align}

\fi

By the first equation of \eqref{a4,a5,a6graph1} and $L_{21}=0$, we obtain 
%%\begin{align}
\begin{equation}\label{L41L42}
%%\begin{split}
L_{41}=\frac{\sqrt{2}(L_{10}L_{31}+L_{11}L_{30})}{L_{00}},\quad
L_{42}=\frac{1}{L_{00}}(\sqrt{2}L_{10}L_{32}+\frac{6\sqrt{2}}{\sqrt{15}}L_{11}L_{31}-\frac{6}{\sqrt{15}}L_{20}L_{22}).
\end{equation}
%%\end{split}
%%\end{align}

Assume $\gamma$ is also ramified at $z=0$ with multiplicity no less than $2$. Then we have 
{\small\begin{align}
    0&=(l_1,l_1)_4=*\,\frac{\partial^4 l_1}{\partial u^4}\frac{\partial^4 l_1}{\partial v^4}+*\,\frac{\partial^4 l_1}{\partial u^3 \partial v}\frac{\partial^4 l_1}{\partial u \partial v^3}+*\,\frac{\partial^4 l_1}{\partial u^2 \partial v^2}\frac{\partial^4 l_1}{\partial u^2 \partial v^2},\label{eq-l1l1-4}\\
    0&=(l_1, l_2)_4=*\,\frac{\partial^4 l_1}{\partial u^4}\frac{\partial^4 l_2}{\partial v^4}+*\,\frac{\partial^4 l_1}{\partial u^3 \partial v}\frac{\partial^4 l_2}{\partial u \partial v^3}+*\,\frac{\partial^4 l_1}{\partial u^2 \partial v^2}\frac{\partial^4 l_2}{\partial u^2 \partial v^2}+*\,\frac{\partial^4 l_1}{\partial u \partial v^3}\frac{\partial^4 l_2}{\partial u^3 \partial v}+*\,\frac{\partial^4 l_1}{ \partial v^4}\frac{\partial^4 l_2}{\partial u^4 }. \label{eq-l1l2-4}
\end{align}
}

Taking the fourth partial derivative with respect to $u$ on \eqref{eq-l1l1-4}, it follows from $\frac{\partial^6 l_1}{ \partial u^5  \partial v}=L_{11}\neq 0$ that $L_{31}=\frac{\partial^6 l_1}{ \partial u^3  \partial v^3}= 0$. Then considering the fourth partial derivative with respect to $v$ on \eqref{eq-l1l1-4}, we obtain 
$L_{41}=\frac{\partial^6 l_1}{ \partial u^2  \partial v^4}= 0$. Substituting these into the first equation of \eqref{L41L42}, we have %$\frac{\partial l_0}{ \partial u^3  \partial v^3}=$
$L_{30}=0$. 

Taking the fourth partial derivative with respect to $u$ on \eqref{eq-l1l2-4}, it follows from $\frac{\partial^6 l_1}{ \partial u^4  \partial v^2}=L_{21}= 0$ and $\frac{\partial^6 l_1}{ \partial u^3  \partial v^3}=L_{31}= 0$ that $L_{32}=\frac{\partial^6 l_2}{ \partial u^3  \partial v^3}= 0$. Then considering the first partial derivative with respect to $v$ followed by the third partial derivative with respect to $u$ on \eqref{eq-l1l2-4}, we obtain that 
$L_{42}=\frac{\partial^6 l_2}{ \partial u^2  \partial v^4}= 0$. Substituting these into the second equation of \eqref{L41L42}, we have %$\frac{\partial l_0}{ \partial u^3  \partial v^3}=$
$L_{20}=0$. 
Thus, we have proved that 
$$L_{20}=0,~~~L_{21}=0,~~~L_{30}=0,~~~L_{31}=0,~~~L_{32}=0,$$
which implies that the vanishing order of $a_2$ and $a_3$ at $z=0$ satisfy
\begin{equation}\label{eq-orda12}
    \ord(a_2)\geq 2,~~~~~~\ord(a_3)\geq 3. 
\end{equation}
%is no less than $2$ and $3$ respectively. 
It follows from \eqref{a4,a5,a6graph1} that 
$$\ord(a_4)\geq 3,~~~\ord(a_5)\geq 3,~~~\ord(a_6)\geq 5.$$
\iffalse
which is equivalent to saying that $L$ takes the following form, 
\begin{equation}\label{the form of L Lower triangular}
L=\begin{pmatrix}
L_{00} & 0 & 0 & 0 & 0 & 0 & 0\\
L_{10} & L_{11} & 0 & 0 & 0 &0 &0\\
0 & 0 & L_{22} & 0 & 0 &0 &0 \\
0 & 0 & 0 & L_{33} & 0 & 0 &0\\
0 & 0 &0 & L_{43} & L_{44} &0 &0\\
0 & 0 &0 &L_{53} & L_{54} & L_{55} &0\\
0 & 0 & 0 & 0 & 0 &  L_{65} & L_{66}
\end{pmatrix}.
\end{equation}
\fi
It also follows that the ramified point
$\gamma(0)=u^5(L_{00}u+\sqrt{6}L_{10}v)$
lies in the closed $2$-dimensional orbit. 

Note that $\gamma(0)$ lies in the $1$-dim orbit if and only if $L_{10}=0$, i.e., 
\begin{equation}\label{eq-orda1}
\ord(a_1)\geq 1, 
\end{equation}
which is equivalent to one of the following inequalities 
\begin{equation}\label{eq-orda456}
\ord(a_4)\geq 4,~~~\ord(a_5)\geq 5,~~~\ord(a_6)\geq 6,
\end{equation}
where \eqref{a4,a5,a6graph1} and \eqref{eq-orda12} are used. 
Note also that one of the seven inequalities in \eqref{eq-orda12} $\sim$ \eqref{eq-orda456} becomes an equality if and only if all of them do, if and only if  $L$ is diagonal. This finishes the proof.   
\end{proof}

\subsection{Necessary conditions for generally ramified holomorphic $2$-spheres to be of constant curvature}~

Let $\gamma:\mathbb{C}P^1\rightarrow G(2,5) $ be a generally ramified holomorphic $2$-sphere of degree $6$. By definition and Proposition~\ref{LowTriRam},  $\gamma$ can be parametrized as 
 \begin{equation}\label{eq-para}\gamma=A\cdot \left(E_0, E_1, E_2, E_3, E_4, E_5, E_6\right)\,L\,  Z_6(z),
 \end{equation}
%We may assume that $\gamma$ is parametrized as 
%\[\gamma(z)=A\cdot (E_0,\ldots,E_6) \,L \,Z_6(z),\]
where $A\in GL(5,\mathbb{C})$, $L$ is in the form of \eqref{L hard}, % a lower-triangular matrix 
and 
$Z_6(z)=\begin{pmatrix}
1 & \sqrt{6}z & \cdots & z^6
\end{pmatrix}^T,$
with $z$ being the standard parameter for the condition of constant curvature. %(Move the ramified point with multiplicity at least $2$ to $z=\infty$ by $SU(2)$, the isometric group of $\mathbb{C}P^1$, then $L$ is automatically lower-triangular following the proof of Proposition \ref{LowTriRam}).  %Let $a_i\triangleq \sum\limits_{j=0}^6\sqrt{\tbinom{6}{j}}L_{ij}z^j,~0\leq i\leq 6$. Since $L$ is lower-triangular, we have $\deg a_i=i,~0\leq i\leq 6$. 
Note that up to an isometry of $G(2,5)$, i.e., a $U(5)$-transformation, we may assume that $A$ is lower-triangular. 
\iffalse
Note that by using the automorphism of $\mathcal{H}^3_0$, we can re-choose $A\in GL(5,\mathbb{C})$ such that $L_{21}=0$. In fact, set $A_1\triangleq \begin{pmatrix}
1 & b \\
0 & 1 
\end{pmatrix}$ with $b=\frac{L_{21}}{\sqrt{10}L_{11}}$, then by \eqref{commutediag}, 
\[A\cdot (E_0,\ldots,E_6) \,L=(A\,  \rho^4(A_1^{-1}))\cdot  (E_0,\ldots,E_6) \,(\rho^6(A_1)\, L).\]
Since $\rho^6(A_1)$ is lower-triangular, $\rho^6(A_1) L$ is also lower-triangular, and 
\[(\rho^6(A_1) L)_{21}=L_{21}-b\sqrt{10}L_{11}=0.\]
\fi
%Given a lower-triangular matrix $A\in GL(5;{\mathbb C})$, 
Then, by the definition of  $\wedge^2$-action, %is identified with its counterpart $\tilde{A}\in GL(10;{\mathbb C})$. Then 
it follows from \eqref{basisOfV6} that $C\triangleq A\cdot \left(E_0, E_1, E_2, E_3, E_4, E_5, E_6\right)$ is of the form
\begin{equation}\label{GenrealC}
C= \begin{psmallmatrix} C_{00}&0&0&0&0&0&0
\\C_{10}&C_{11}&0&0&0&0&0
\\C_{20}&C_{21}&C_{22}&0&0&0&0
\\ C_{30}&C_{31}&C_{32}&C_{33}&0&0&0
\\ C_{40}&C_{41}&C_{42}&0&0&0&0
\\ C_{50}&C_{51}&C_{52}&C_{53}&0&0&0
\\C_{60}&C_{61}&C_{62}&C_{63}&C_{64}
&0&0
\\C_{70}&C_{71}&C_{72}&C_{73}&
C_{74}&0&0
\\C_{80}&C_{81}&C_{82}&C_{83}&C_{84}&C_{85}&0
\\C_{90}&C_{91}&C_{92}&C_{93}&C_{94}&C_{95}&C_{96}\end{psmallmatrix},
\end{equation}
%%where the inequality follows  from the product of diagonal entries of $A$ by $L_{02},L_{23},L_{34}$, 
which is a $10\times 7$ matrix %(denoted by $C$ in the sequel) 
obtained by column vectors $A\cdot E_k$ written relative to the standard basis $e_i\wedge e_j, ~0\leq i<j\leq 4$, in the lexicographic order. %For example, the first column of \eqref{GenrealC} is $Ae_0\wedge Ae_1$, etc.
We point out that $C_{ij}$ are quadratic in terms of the entries of $A$.  The following lemma is important.

\begin{lemma}\label{usefulAlgeLemma}
Let $G$ be a $10\times 7$ matrix of rank $7$ in the same form as on the right-hand side of \eqref{GenrealC} with $G_{33}G_{53}G_{64}G_{74}\neq 0$, and let the column vectors of $G$ be mutually orthogonal. If the holomorphic $2$-sphere $\gamma(z)\triangleq G\, Z_6(z)$ lies in a generic linear section of $G(2,5)$, % and {\color{blue}is generic}, 
then $G$ is in the form   
\begin{equation}\label{simple form}
\begin{psmallmatrix}G_{00}&0&0&0&0&0&0
\\ 0&G_{11}&0&0&0&0&0
\\ 0&0&G_{22}&0&0&0&0
\\0&0&0&G_{33}&0&0&0
\\0&0&G_{42}&0&0&0&0
\\0&0&0&G_{53}&0&0&0
\\0&0&0&0&G_{64}&0&0
\\0&0&0&0&G_{74}&0&0
\\0&0&0&0&0&G_{85}&0
\\0&0&0&0&0&0&G_{96}
\end{psmallmatrix},
 \end{equation}
%%\begin{equation}\label{simple form}\left( \begin {array}{ccccccc} {\it G_{00}}&0&0&0&0&0&0
%%\\ \noalign{\medskip}0&{\it G_{11}}&0&0&0&0&0\\ \noalign{\medskip}0&0&{
%%\it G_{22}}&0&0&0&0\\ \noalign{\medskip}0&0&0&{\it G_{33}}&0&0&0
%%\\ \noalign{\medskip}0&0&{\it G_{42}}&0&0&0&0\\ \noalign{\medskip}0&0&0&{
%%\it G_{53}}&0&0&0\\ \noalign{\medskip}0&0&0&0&{\it G_{64}}&0&0
%%\\ \noalign{\medskip}0&0&0&0&{\it G_{74}}&0&0\\ \noalign{\medskip}0&0&0&0
%%&0&{\it G_{85}}&0\\ \noalign{\medskip}0&0&0&0&0&0&{\it G_{96}}\end {array}
%% \right),\end{equation}
 where $\gamma$ is ramified at $z=0$ and $z=\infty$ with multiplicities at least $2$.
\end{lemma}

\begin{proof}
If \eqref{simple form} holds, then the last statement follows from 
\begin{align*}
\gamma^{\prime}(0)&=e_0\wedge e_2\in G(2,5),~~\gamma^{\prime\prime}(0)=G_{22}e_{0}\wedge e_3+G_{42}e_{1}\wedge e_2,~~\gamma^{\prime}(0)\wedge \gamma^{\prime\prime}(0)=0,\\
\gamma^{\prime}(\infty)&=e_2\wedge e_4 \in G(2,5),~~\gamma^{\prime\prime}(\infty)=G_{64}e_{1}\wedge e_4+G_{74}e_{2}\wedge e_3,~~\gamma^{\prime}(\infty)\wedge \gamma^{\prime\prime}(\infty)=0.
\end{align*}

Hence, we need only prove \eqref{simple form} in the following.
Since the first five columns of $G$ are perpendicular to the last two, we have  
\[\gamma(z)=\begin{psmallmatrix}
G_{00}&0&0&0&0&0&0
\\ G_{10}&G_{11}&0&0&0&0&0
\\ G_{20}&G_{21}&G_{22}&0&0&0&0
\\ G_{30}&G_{31}&G_{32}&G_{33}&0&0&0
\\ G_{40}&G_{41}&G_{42}&0&0&0&0
\\ G_{50}&G_{51}&G_{52}& G_{53}&0&0&0
\\ G_{60}&G_{61}&G_{62}&G_{63}&G_{64}&0&0
\\G_{70}&G_{71}&G_{72}&G_{73}&G_{74}&0&0
\\0&0&0&0&0&G_{85}&0
\\0&0&0&0&0&0&G_{96}
\end{psmallmatrix}
\begin{psmallmatrix}
1\\
\sqrt{6}z\\
\sqrt{15}z^2\\
2\sqrt{5}z^3\\
\sqrt{15}z^4\\
\sqrt{6}z^5\\
z^6\end{psmallmatrix}.\]
 %%\[\gamma(z)=\left( \begin {array}{ccccccc} {\it G_{00}}&0&0&0&0&0&0
%%\\ \noalign{\medskip}{\it G_{10}}&{\it G_{11}}&0&0&0&0&0
%%\\ \noalign{\medskip}{\it G_{20}}&{\it G_{21}}&{\it G_{22}}&0&0&0&0
%%\\ \noalign{\medskip}{\it G_{30}}&{\it G_{31}}&{\it G_{32}}&{\it G_{33}}&0&0&0
%%\\ \noalign{\medskip}{\it G_{40}}&{\it G_{41}}&{\it G_{42}}&0&0&0&0
%5\\ \noalign{\medskip}{\it G_{50}}&{\it G_{51}}&{\it G_{52}}&{\it G_{53}}&0&0&0
%%\\ \noalign{\medskip}{\it G_{60}}&{\it G_{61}}&{\it G_{62}}&{\it G_{63}}&{\it G_{64}}
%%&0&0\\ \noalign{\medskip}{\it G_{70}}&{\it G_{71}}&{\it G_{72}}&{\it G_{73}}&{\it 
%%G_{74}}&0&0\\ \noalign{\medskip}0&0&0&0&0&{\it G_{85}}&0
%%\\ \noalign{\medskip}0&0&0&0&0&0&{\it G_{96}}\end {array} \right)
%%\left(\begin {array}{c}1\\
%%\sqrt{6}z\\
%%\sqrt{15}z^2\\
%%2\sqrt{5}z^3\\
%%\sqrt{15}z^4\\
%%\sqrt{6}z^5\\
%%z^6\end {array}
%% \right).\]
 %where the columns (denoted by $\{G_j~|~j=0,\ldots,6\}$) of the above $10\times 7$ matrix are orthogonal to each other. 
We denote by $\{\gamma_j~|~j=0,\ldots,9\}$ the coordinates of $\gamma$. Then it is easy to see 
\[{\small\aligned&\deg(\gamma_0)=0, ~~~\deg(\gamma_1)\leq 1,~~~\deg(\gamma_2)\leq 2,~~~\deg(\gamma_3)\leq 3,~~~\deg(\gamma_4)\leq 2,\\
&\deg(\gamma_5)\leq 3, ~~~\deg(\gamma_6)\leq 4,~~~\deg(\gamma_7)\leq 4,~~~\deg(\gamma_8)=5,~~~\deg(\gamma_9)=6.\endaligned}\]
It follows from $\gamma\subset G(2,5)$ that 
\begin{align}{\small
\gamma_2 \gamma_4 - \gamma_1 \gamma_5 + \gamma_0 \gamma_7=0,\label{eq-g1}\\
\gamma_3 \gamma_4 - \gamma_1 \gamma_6 + \gamma_0 \gamma_8=0,\label{eq-g2}\\
\gamma_3 \gamma_5-\gamma_2 \gamma_6+\gamma_0 \gamma_9=0,\label{eq-g3}\\
\gamma_3 \gamma_7 - \gamma_2 \gamma_8 + \gamma_1 \gamma_9=0,\label{eq-g4}\\
\gamma_6 \gamma_7 - \gamma_5 \gamma_8 + \gamma_4 \gamma_9=0 \label{eq-g5}.
}\end{align}

Moreover, $\gamma_i\neq 0,~i=0\ldots,9$, since $\gamma$ lies in a generic linear section. 
Meanwhile, by the orthogonality of $\{G_j~|~j=0,\ldots,6\}$, we obtain
{\small $|G_j|^2\sqrt{\tbinom{6}{j}}z^j=\langle \gamma, G_j\rangle=\sum_{k=0}^9\overline{G_{kj}}\gamma_k,$} so that
{\small\begin{align}
\overline{G_{64}}\gamma_6+\overline{G_{74}}\gamma_7&=|G_4|^2\sqrt{15}z^4,\label{eq-g6}\\
\overline{G_{33}}\gamma_3+\overline{G_{53}}\gamma_5+\overline{G_{63}}\gamma_6+\overline{G_{73}}\gamma_7&=|G_3|^2\sqrt{20}z^3 \label{eq-g7}.
\end{align}}
\vskip -0.35cm

\vspace{2mm}

In the following, we will use the assumption $G_{33}G_{53}G_{64}G_{74}\neq 0$. Observe that $\gamma_8=G_{85}z^5$ and $\gamma_9=G_{96}z^6$. As a polynomial of $z$, we denote by $m(\gamma_j)$ the order of $\gamma_j$ at $z=0$. 

Combining \eqref{eq-g6} and $G_{64}G_{74}\neq 0$, and using $\deg(\gamma_6)=\deg(\gamma_7)=4$, it yields 
$0\leq  m(\gamma_6)=m(\gamma_7)\leq 4.$
Meanwhile \eqref{eq-g5} gives $z^5|\gamma_6\gamma_7$, which implies $5\leq m(\gamma_6)+m(\gamma_7)$. It follows that $m(\gamma_6)=m(\gamma_7)\geq 3$. Moreover,  we obtain $z^5 \mid \gamma_3\gamma_7$ in accord with \eqref{eq-g4}.

{\bf Claim 1.} $\gamma_6=G_{64} z^4$ and $\gamma_7=G_{74} z^4.$

Otherwise, we assume $m(\gamma_7)=3$. Then $2\leq m(\gamma_3)\leq 3$ and $m(\gamma_6)=3$. Using \eqref{eq-g7}, we have 
$m(\gamma_5)\geq 2$, which implies 
$z^4 \mid (\gamma_3 \gamma_5+\gamma_0 \gamma_9).$
It follows from \eqref{eq-g3} that $z^4 \mid \gamma_2\gamma_6$. As a result, $m(\gamma_2)\geq1$, 
and 
$z^6 \mid (\gamma_2 \gamma_8-\gamma_1 \gamma_9).$ Next, \eqref{eq-g4} yields $z^6 \mid \gamma_3\gamma_7$, and then $m(\gamma_3)=3$. Using \eqref{eq-g7} again, we obtain
$m(\gamma_5)\geq 3$. 
Coupled with \eqref{eq-g3}, $z^6 \mid \gamma_2\gamma_6$ can be deduced. Consequently, $m(\gamma_2)\geq3$, which contradicts $\deg(\gamma_2)\leq 2$. Hence the claim follows from the degrees of $\gamma_6$ and $\gamma_7$.

Now that we have $z^4 \mid (\gamma_1 \gamma_6-\gamma_0 \gamma_8)$, it follows from \eqref{eq-g2} that $z^4 \mid \gamma_3\gamma_4$. Since $\deg(\gamma_4)=2$, there follows $m(\gamma_3)\geq 2$.

{\bf Claim 2.} $\gamma_3=G_{33} z^3.$

Otherwise, we assume $m(\gamma_3)= 2$. Then $m(\gamma_4)=2$. Hence 
$z^8 \mid (\gamma_4 \gamma_9+\gamma_6 \gamma_7),$
 and $z^8\mid \gamma_5\gamma_8$, from which we can derive that $m(\gamma_5)\geq3$. Using \eqref{eq-g7} again, there yields that $m(\gamma_3)\geq3$ (by $G_{33}\neq 0$), which contradicts the assumption. Therefore $m(\gamma_3)=3$ and the Claim 2 follows from $\deg(\gamma_3)=3$. 

Now, $\gamma_5=G_{53} z^3$ follows from \eqref{eq-g7} and $\deg(\gamma_5)=3$. 

Using \eqref{eq-g5}, we obtain $z^8\mid \gamma_4\gamma_9$. Hence $\gamma_4=G_{42} z^2$ by $\deg(\gamma_4)=2$.

From \eqref{eq-g3}, we have $z^6\mid \gamma_2\gamma_6$. Therefore, $\gamma_2=G_{22} z^2$ due to that $\deg(\gamma_2)=2$.

Lastly, it follows from \eqref{eq-g4} that $z^7\mid \gamma_1\gamma_9$.  So $\gamma_1=G_{11} z$, as $\deg(\gamma_1)=1$.
\end{proof}
The method used in the proof of the above lemma can be generalized to prove the following important proposition. 

\begin{proposition}%\label{prop-lowertri}
Let $\gamma:\mathbb{C}P^1\rightarrow G(2,5) $ be a generally ramified holomorphic $2$-sphere of degree $6$ parametrized by \eqref{eq-para}. If $\gamma$ is of constant curvature, then $L$ is lower-triangular. %, then $\gamma$ belongs the the diagonal family.
%Moreover, $\gamma$ is ramified at two distinct points with multiplicities at least $2$. 
\end{proposition}
\begin{proof}
\iffalse
By definition and Proposition~\ref{LowTriRam}, $\gamma$ can be parametrized as 
\[\gamma(z)=B\cdot (E_0,\ldots,E_6)\cdot L \cdot Z_6(z),\]
where $B\in GL(5,\mathbb{C})$ and $L$ in the form \eqref{L hard}, under the standard parameter $z$ for the constant curvature condition (Move the tangent point to $z=\infty$ by $SU(2)$, the isometric group of $\mathbb{C}P^1$, then $L$ is automatically in the form \eqref{L hard} following the proof of Lemma \ref{tangential1case}). It follows from Theorem \ref{conj 1} and Proposition~\ref{LowTriRam}, we only need 
\fi
To show that $L$ is lower-triangular, it follows from Proposition~\ref{LowTriRam} that we need only prove that one of $L_{02},L_{23},L_{34}$ vanishes. %We proved it by contradictions. 

Suppose that in the following $L_{02}L_{23}L_{34}\neq 0.$
Similarly as before, we assume that %Let $B=U\cdot A$ be the QR decomposition of $B$, where $U$ is unitary, and 
$A$ is a lower-triangular matrix. Then
$G\triangleq A\cdot  (E_0,\ldots,E_6)\, L$
is a $10\times 7$ matrix with orthonormal columns and takes the following form %\eqref{GenrealC}, with
%As before, $A\cdot  (E_0,\ldots,E_6)$ is in the form \eqref{GenrealC}, and by \eqref{L hard}, $G=A\cdot  (E_0,\ldots,E_6)\cdot L$ is in the form 
\begin{equation}\label{inequality0}
G=\begin{psmallmatrix}
G_{00}&G_{01}&G_{02}&0&0&0&0
\\ G_{10}&G_{11}&G_{12}&0&0&0&0
\\ G_{20}&G_{21}&G_{22}&G_{23}&0&0&0
\\ G_{30}&G_{31}&G_{32}&G_{33}&G_{34}&0&0
\\ G_{40}&G_{41}&G_{42}&G_{43}&0&0&0
\\ G_{50}&G_{51}&G_{52}& G_{53}&G_{54}&0&0
\\ G_{60}&G_{61}&G_{62}&G_{63}&G_{64}&0&0
\\G_{70}&G_{71}&G_{72}&G_{73}&G_{74}&0&0
\\G_{80}&G_{81}&G_{82}&G_{83}&G_{84}&G_{85}&0
\\G_{90}&G_{91}&G_{92}&G_{93}&G_{94}&G_{95}&G_{96}
\end{psmallmatrix}, \quad G_{02}G_{23}G_{43}G_{34}G_{54}\neq 0,
\end{equation}
%%where
%%$$
%%begin{equation}\label{assume some nonvanish}
%%G_{02}G_{23}G_{43}G_{34}G_{54}\neq 0,
%%$$
%%\end{equation}
%where the inequality follows
where the inequality  comes from the product of diagonal entries of $A$ and  $L_{02}L_{23}L_{34}$. 

Since the first five columns of $G$ are perpendicular to the last two, we have  
\[\gamma(z)= G\, Z_6(z)=\begin{psmallmatrix}
G_{00}&G_{01}&G_{02}&0&0&0&0
\\ G_{10}&G_{11}&G_{12}&0&0&0&0
\\ G_{20}&G_{21}&G_{22}&G_{23}&0&0&0
\\ G_{30}&G_{31}&G_{32}&G_{33}&G_{34}&0&0
\\ G_{40}&G_{41}&G_{42}&G_{43}&0&0&0
\\ G_{50}&G_{51}&G_{52}& G_{53}&G_{54}&0&0
\\ G_{60}&G_{61}&G_{62}&G_{63}&G_{64}&0&0
\\G_{70}&G_{71}&G_{72}&G_{73}&G_{74}&0&0
\\0&0&0&0&0&G_{85}&0
\\0&0&0&0&0&0&G_{96}
\end{psmallmatrix}
\begin{psmallmatrix}
1\\
\sqrt{6}z\\
\sqrt{15}z^2\\
2\sqrt{5}z^3\\
\sqrt{15}z^4\\
\sqrt{6}z^5\\
z^6\end{psmallmatrix}. \]
\iffalse
\begin{align*}
G_{02}&=a_{00}a_{11}L_{02}\neq 0,\\
G_{23}&=\frac{\sqrt{15}}{5}a_{00}a_{33}L_{23}\neq 0 ,\\
G_{43}&=\frac{\sqrt{10}}{5}a_{11}a_{22}L_{23}\neq 0,\\
G_{34}&=\frac{\sqrt{5}}{5}a_{00}a_{44}L_{34}\neq 0,\\
G_{54}&=\frac{2\sqrt{5}}{5}a_{11}a_{33}L_{34}\neq 0,
\end{align*}
so we assume 

in the following.
\fi
We denote by $\{\gamma_j~|~j=0,\ldots,9\}$ the coordinates of $\gamma$. Then it is easy to see 
\[{\small\aligned&\deg(\gamma_0)\leq 2, ~~~\deg(\gamma_1)\leq 2,~~~\deg(\gamma_2)\leq 3,~~~\deg(\gamma_3)\leq 4,~~~\deg(\gamma_4)\leq 3,\\
&\deg(\gamma_5)\leq 4, ~~~\deg(\gamma_6)\leq 4,~~~\deg(\gamma_7)\leq 4,~~~\deg(\gamma_8)=5,~~~\deg(\gamma_9)=6,\endaligned}\]
satisfying \eqref{eq-g1} through \eqref{eq-g5}. The same constraint between \eqref{eq-g5} and \eqref{eq-g6} gives

\iffalse

It follows from that $\gamma$ lies in $G(2,5)$ that 
\begin{align}{\small
\gamma_2 \gamma_4 - \gamma_1 \gamma_5 + \gamma_0 \gamma_7=0,\label{eq-g1}\\
\gamma_3 \gamma_4 - \gamma_1 \gamma_6 + \gamma_0 \gamma_8=0,\label{eq-g2}\\
\gamma_3 \gamma_5-\gamma_2 \gamma_6+\gamma_0 \gamma_9=0,\label{eq-g3}\\
\gamma_3 \gamma_7 - \gamma_2 \gamma_8 + \gamma_1 \gamma_9=0,\label{eq-g4}\\
\gamma_6 \gamma_7 - \gamma_5 \gamma_8 + \gamma_4 \gamma_9=0 \label{eq-g5}.
}\end{align}

Moreover, $\gamma_i\neq 0,~i=0,\ldots,9$, since $\gamma$ lies in a generic linear section. From the orthogonality of $\{G_j~|~j=0,\ldots,6\}$, we obtain 
{\small $|G_j|^2\sqrt{\tbinom{6}{j}}z^j=\langle \gamma, G_j\rangle=\sum_{k=0}^9\overline{G_{kj}}\gamma_k,$} so that

\fi

{\small\begin{align}
\overline{G_{00}}\gamma_0+\overline{G_{10}}\gamma_1+\overline{G_{20}}\gamma_2+\overline{G_{30}}\gamma_3+\overline{G_{40}}\gamma_4+\overline{G_{50}}\gamma_5+\overline{G_{60}}\gamma_6+\overline{G_{70}}\gamma_7&=|G_0|^2\label{eq-g6},\\
%\overline{G_{02}}\gamma_0+\overline{G_{12}}\gamma_1+\overline{G_{22}}\gamma_2+\overline{G_{32}}\gamma_3+\overline{G_{42}}\gamma_4+\overline{G_{52}}\gamma_5+\overline{G_{62}}\gamma_6+\overline{G_{72}}\gamma_7&=|G_2|^2\sqrt{15}z^2\label{eq-g7},\\
\overline{G_{23}}\gamma_2+\overline{G_{33}}\gamma_3+\overline{G_{43}}\gamma_4+\overline{G_{53}}\gamma_5+\overline{G_{63}}\gamma_6+\overline{G_{73}}\gamma_7&=|G_3|^2\sqrt{20}z^3 \label{eq-g8},\\
\overline{G_{34}}\gamma_3+\overline{G_{54}}\gamma_5+\overline{G_{64}}\gamma_6+\overline{G_{74}}\gamma_7&=|G_4|^2\sqrt{15}z^4.\label{eq-g9}
\end{align}}
As a polynomial of $z$, we denote by $m(\gamma_j)$ the order of $\gamma_j$ at $z=0$. 

Since $\gamma_3=p_{04}(F),~\gamma_6=p_{14}(F)$, we have that $\gamma_3$ and $\gamma_6$ are linearly independent (since $F$ lies in a generic linear section). Hence combining this with $\deg(\gamma_3),\deg(\gamma_6)\leq 4$, we deduce
\begin{equation}\label{importan k-1}
k\triangleq \min\{m(\gamma_3),m(\gamma_6)\}\leq 3.
\end{equation}
It follows from $m(\gamma_8)=5$ and $m(\gamma_9)=6$, \eqref{eq-g4} and \eqref{eq-g5}, that $5\leq m(\gamma_3\gamma_7),m(\gamma_6\gamma_7).$
Since $\deg \gamma_7\leq 4$, by \eqref{importan k-1}, we obtain $1\leq k\leq 3,~~2\leq m(\gamma_7),$
while \eqref{eq-g9} and $G_{54}\neq 0$ yields
\begin{equation}\label{mg5-eq1}
m(\gamma_5)\geq \min\{k,m(\gamma_7)\}\geq 1.
\end{equation} 
Using \eqref{eq-g1} and \eqref{eq-g8}, we arrive at
\begin{align}
m(\overline{G_{23}}\gamma_2+\overline{G_{43}}\gamma_4)&\geq \min\{k,m(\gamma_7)\}\geq 1,\label{prestar2}\\
m(\gamma_2\gamma_4)&\geq \min\{k,m(\gamma_7)\}\geq 1.\label{prestar3}
\end{align}
We claim that 
%%\begin{align}
\begin{equation}\label{bootstrap-1}
%%\begin{split}
m(\gamma_2)\geq [\frac{\min\{k,m(\gamma_7)\}+1}{2}]\geq 1,\quad
m(\gamma_4)\geq [\frac{\min\{k,m(\gamma_7)\}+1}{2}]\geq 1.
%%\end{split}
%%\end{align}
\end{equation}
Indeed, if $m(\gamma_2)=m(\gamma_4)$, then the claim follows from \eqref{prestar3}. %%and parity of $\frac{\min\{k,m(\gamma_7)\}}{2}$ ({\color{red} ?$\min\{k,m(\gamma_7)\}$ }). 
If $m(\gamma_2)\neq m(\gamma_4)$, then by $G_{23}G_{43}\neq 0$ and \eqref{prestar2}, we obtain that 
\[\min\{m(\gamma_2), m(\gamma_4)\}=m(\overline{G_{23}}\gamma_2+\overline{G_{43}}\gamma_4)\geq \min\{k,m(\gamma_7)\} \geq [\frac{\min\{k,m(\gamma_7)\}+1}{2}].\]
This proves our claim. Next, from \eqref{eq-g4}, \eqref{eq-g5} and \eqref{mg5-eq1} we derive (because $\min\{k,m(\gamma_7)\}\geq 1$) that
\begin{align}\label{bootstrap}
\begin{split}
m(\gamma_3\gamma_7)&\geq \min\{5+[\frac{\min\{k,m(\gamma_7)\}+1}{2}],6\}\geq 6,\\
m(\gamma_6\gamma_7)&\geq \min\{5+\min\{k,m(\gamma_7)\},[\frac{\min\{k,m(\gamma_7)\}+1}{2}]+6\}\geq 6.
\end{split}
\end{align}
Since $1\leq k=\min\{\deg \gamma_3,\deg\gamma_6\}\leq 3,~\deg{\gamma_7}\leq 4$, we must have $2\leq m(\gamma_3), m(\gamma_6)$ and $3\leq m(\gamma_7)$; hence 
\begin{equation}\label{bootstrap3}
2\leq k\leq 3,~3\leq m(\gamma_7) \leq 4. 
\end{equation}
Now, we divide the discussion according to $m(\gamma_7)$.
 
\textbf{Case 1}: Assume that $m(\gamma_7)=3$. Then $\min\{k,m(\gamma_7)\}=k\geq 2$, so that \eqref{bootstrap} implies 
%%\begin{align*}
$$
m(\gamma_3\gamma_7)\geq 6,\quad
7\geq \deg \gamma_6+m(\gamma_7)\geq m(\gamma_6\gamma_7)\geq 7;
$$
%%\end{align*}
hence, $k=2,~m(\gamma_6)=4$, and $m(\gamma_3)\geq 3$. But then 
\[2=k=\min\{m(\gamma_3),m(\gamma_6)\}\geq \min\{3,4\}=3,\]
a contradiction.

\textbf{Case 2}: Assume that $m(\gamma_7)=4$. Then by \eqref{mg5-eq1}, \eqref{bootstrap-1}, \eqref{bootstrap} and \eqref{bootstrap3}, we obtain  
 \[1\leq m(\gamma_2),m(\gamma_4),~2\leq m(\gamma_3),m(\gamma_5),~3\leq m(\gamma_6).\]
We conclude that $G$ is in the form 
 \begin{equation}\label{some hard case}
G=A\cdot (E_0,\ldots,E_6)\, L=\begin{psmallmatrix}
G_{00}&G_{01}&G_{02}&0&0&0&0
\\ G_{10}&G_{11}&G_{12}&0&0&0&0
\\ 0&G_{21}&G_{22}&G_{23}&0&0&0
\\0&0&G_{32}&G_{33}&G_{34}&0&0
\\ 0&G_{41}&G_{42}&G_{43}&0&0&0
\\ 0&0&G_{52}& G_{53}&G_{54}&0&0
\\0&0&0&G_{63}&G_{64}&0&0
\\0&0&0&0&G_{74}&0&0
\\0&0&0&0&0&G_{85}&0
\\0&0&0&0&0&0&G_{96}
\end{psmallmatrix}.
\end{equation}
Consider the QR decomposition of $A\cdot (E_0,\ldots,E_6)=N\cdot L_{1}$, where $N$ is a $10\times 7$ matrix with orthonormal columns, and $L_{1}=(J_{ij})_{0\leq i,,j\leq 6}$ is a $7\times 7$ lower-triangular matrix. Since $A\cdot (E_0,\ldots,E_6)$ is in the form \eqref{GenrealC}, necessarily $N$ is given by 
\[N= \begin{psmallmatrix} 
N_{00}&0&0&0&0&0&0\\N_{10}&N_{11}&0&0&0&0&0
\\N_{20}&N_{21}&N_{22}&0&0&0&0
\\ N_{30}&N_{31}&N_{32}&N_{33}&0&0&0
\\ N_{40}&N_{41}&N_{42}&0&0&0&0
\\ N_{50}&N_{51}&N_{52}&N_{53}&0&0&0
\\N_{60}&N_{61}&N_{62}&N_{63}&N_{64}
&0&0
\\N_{70}&N_{71}&N_{72}&N_{73}&
N_{74}&0&0
\\ 0 & 0 & 0 & 0 & 0 &1& 0
\\ 0 & 0 & 0 & 0 & 0&0&1\end{psmallmatrix}= \begin{psmallmatrix} 
N_0 & 0_{1\times 2}\\
N_1 & 0_{1\times 2}\\
\vdots \\
N_7 & 0_{1\times 2}\\
0_{2\times 5} & Id_2\end{psmallmatrix},\]
where $N_j,~0\leq j\leq 7$ are row vectors in $\mathbb{C}^5$. Moreover, 
\begin{equation}\label{N64N74not0}
N_{64}N_{74}\neq 0,
\end{equation}
since $(N_{64},N_{74})$ is parallel to $(\frac{\sqrt{15}}{5}a_{11}a_{44},\frac{\sqrt{10}}{5}a_{22}a_{33})$ and the diagonal entries of $A$ are not zero. Now, from $G=N\cdot L_2\cdot L$ and the orthogonality of columns of $G$ and $N$, respectively, we must have that $L_2\cdot L=(H_{ij})_{0\leq i,j\leq 6}\in  U(7)$ is in the same form as \eqref{L hard} with
\begin{equation}\label{eq-H1}
H_{23}=J_{22}L_{23}\neq 0,~~H_{34}=J_{33}L_{34}\neq 0.
\end{equation}
Since $L_2\cdot L\in U(7)$, it is necessary that
\[L_2\cdot L=\begin{pmatrix}
H_0 & H_1 &\cdots & H_4 & 0_{5\times 2}\\
& 0_{2\times 5} & & &  \begin{pmatrix}
H_{55} & 0\\
0& H_{66}
\end{pmatrix} 
\end{pmatrix}\]
where $H_i,~0\leq i\leq 4$, are column vectors in $\mathbb{C}^5$ that form an orthonormal basis of $\mathbb{C}^5$, and $H_3$ and $H_4$ are in the form
%%\begin{align}
\begin{equation}\label{form of H3H4}
%%\begin{split}
H_3=(0,0,H_{23},H_{33},H_{43})^t,\quad H_4=(0,0,0,H_{34},H_{44})^t.
\end{equation}
%%\end{split}
%%\end{align}
Since $G=N\cdot L_2\cdot L$, by $G_{6j}=0,~0\leq j\leq 2,$ and $G_{7i}=0,~0\leq i\leq 3$ (see \eqref{some hard case}), we obtain 
\[N_6\cdot H_j=0,~~0\leq j\leq 2,\quad N_7\cdot H_{i}=0,~0\leq i\leq 3;\]
hence, $N_6\in \Span\{\overline{H_3^t},\overline{H_4^t}\}$ and $N_7\in \Span\{\overline{H_4^t}\}$, so that we conclude by \eqref{form of H3H4} that $N$ is in the following form 
\[N= \begin{psmallmatrix} 
N_{00}&0&0&0&0&0&0\\N_{10}&N_{11}&0&0&0&0&0
\\N_{20}&N_{21}&N_{22}&0&0&0&0
\\ N_{30}&N_{31}&N_{32}&N_{33}&0&0&0
\\ N_{40}&N_{41}&N_{42}&0&0&0&0
\\ N_{50}&N_{51}&N_{52}&N_{53}&0&0&0
\\0&0&N_{62}&N_{63}&N_{64}
&0&0
\\0&0&0&N_{73}&
N_{74}&0&0
\\ 0 & 0 & 0 & 0 & 0 &1& 0
\\ 0 & 0 & 0 & 0 & 0&0&1\end{psmallmatrix}.\]
Then the inner product of the third column with fifth column gives 
$\overline{N_{62}}N_{64}=0,$
and by $N_{64}\neq 0$ (see \eqref{N64N74not0}) we obtain $N_{62}=0$. Meanwhile, from $N_6\in \Span\{\overline{H_3^t},\overline{H_4^t}\}$ we deduce 
\[N_6=(0,0,0,N_{63},N_{64})=a\cdot \overline{H_3^t}+b\cdot \overline{H_4^t},\]
for some constant $a,b$. Then from $H_{23}\neq 0$ (see \eqref{eq-H1} and \eqref{form of H3H4}), we infer $a=0$; hence $N_6$ is parallel to $\overline{H_4^t}$. Thus,
$G_{63}=N_6\cdot H_3=0,$
which implies that $m(\gamma_6)=4$. Then 
\[2\leq k=\min\{m(\gamma_3),m(\gamma_6)\}=m(\gamma_3)\leq 3.\]
It follows from \eqref{eq-g3} and \eqref{eq-g9} that 
%%\begin{align}
\begin{equation}\label{finalbootstrp}
%%\begin{split}
m(\gamma_3\gamma_5)\geq 5,\quad m(\overline{G_{34}}\gamma_3+\overline{G_{54}}\gamma_5)\geq 4.
\end{equation}
%%\end{split}
%%\end{align}
From \eqref{mg5-eq1}, we have $m(\gamma_5)\geq k=\min\{m(\gamma_3),m(\gamma_6)\}=m(\gamma_3)$.
Combining \eqref{finalbootstrp}, $2\leq m(\gamma_3)\leq 3$ with $G_{34}G_{54}\neq 0$, we arrive at $m(\gamma_3)=m(\gamma_5)$. Then $k=m(\gamma_3)=m(\gamma_5)\geq \frac{5}{2}$ implies $k=3$. Next, from \eqref{bootstrap-1}, we see $2\leq m(\gamma_2),m(\gamma_4)$. Lastly by \eqref{eq-g2}, we arrive at
\[m(\gamma_1)+4=m(\gamma_1\gamma_6)\geq 5;\]
hence $m(\gamma_1)\geq 1$, so $G_{10}=0$. Then \eqref{eq-g6} gives 
$\overline{G_{00}}\gamma_0=|G_0|^2,$
so that $G_{02}=0$, contradictory to the inequality in \eqref{inequality0}. 

In short, one of $L_{02},L_{23},L_{34}$ vanishes so that $L$ is lower-triangular.
\end{proof}

Now we can finish the proof of Theorem~{\rm \ref{thm-tangential}}.%% \\~\vskip -2cm\\

\newenvironment{demo}[1][Proof of  Theorem~{\rm 6.1.}]{\proof[#1]}{\endproof}
\begin{demo}~

\iffalse
We may assume that $\gamma$ is parametrized as 
\[\gamma(z)=A\cdot (E_0,\ldots,E_6) \,L \,Z_6(z),\]
where $A\in GL(5,\mathbb{C})$, $L$ is a lower-triangular matrix and 
$$Z_6(z)=\begin{pmatrix}
1 & \sqrt{6}z & \cdots & z^6
\end{pmatrix}^T,$$ 
with $z$ being the standard parameter for the constant curvature condition (Move the ramified point with multiplicity at least $2$ to $z=\infty$ by $SU(2)$, the isometric group of $\mathbb{C}P^1$, then $L$ is automatically lower-triangular following the proof of Proposition \ref{LowTriRam}).  %Let $a_i\triangleq \sum\limits_{j=0}^6\sqrt{\tbinom{6}{j}}L_{ij}z^j,~0\leq i\leq 6$. Since $L$ is lower-triangular, we have $\deg a_i=i,~0\leq i\leq 6$. 
\fi
We continue to use the parameterization given in \eqref{eq-para}. 
Note that by using the automorphism of $\mathcal{H}^3_0$, we can re-choose $A\in GL(5,\mathbb{C})$ such that $L_{21}=0$. In fact, set $A_1\triangleq \begin{pmatrix}
1 & b \\
0 & 1 
\end{pmatrix}$ with $b=\frac{L_{21}}{\sqrt{10}L_{11}}$, then by \eqref{commutediag}, 
\[A\cdot (E_0,\ldots,E_6) \,L=(A\,  \rho^4(A_1^{-1}))\cdot  (E_0,\ldots,E_6) \,(\rho^6(A_1)\, L).\]
Since $\rho^6(A_1)$ is lower-triangular, $\rho^6(A_1) L$ is also lower-triangular, and so we derive
\[(\rho^6(A_1) L)_{21}=L_{21}-b\sqrt{10}L_{11}=0.\]
%The choice $b\triangleq \frac{L_{21}}{\sqrt{10}L_{11}}$ will realize our claim. %Then to show that $L$ is diagonal, it suffices to verify that $a_i=\sqrt{\tbinom{6}{j}}L_{ii}z^i,~i=0,\ldots,6$ are monomials.

The constant curvature condition of $\gamma$ implies that  
\begin{equation}\label{constant curvature in lower-triangular}
G\triangleq A\cdot (E_0,\ldots,E_6) \,L,
\end{equation}
is a $10\times 7$ matrix with orthonormal columns. Similarly as before, up to a $U(5)$-transformation, %Let $A\triangleq U\, L_1$ be the QR decomposition of $A$, where $U$ is unitary and $L_1$ is lower-triangular. 
we may assume $A$ is lower-triangular.  Since $L$ is lower-triangular, we see that $G$ %By \eqref{constant curvature in lower-triangular}, we obtain that 
has the form as \eqref{GenrealC}.  % with orthonormal columns. 
It is easy to verify that %and $G\, Z_6(z)$ is a rational normal curve in $G(2,5)$ with 
$G_{33}G_{53}G_{64}G_{74}\neq 0$. %(they are products of diagonals of $L_1$ and $L$). 

It follows from Lemma \ref{usefulAlgeLemma} that $G$ must be in the form of \eqref{simple form}. Moreover, $G\,Z_6(z)$ and 
\[\mu(z)\triangleq A^{-1}\cdot G\, Z_6(z)= (E_0,\ldots,E_6)\, L\,  Z_6(z)\] 
are ramified at $z=0$ and $z=\infty$ with multiplicities at least $2$. % (the action of $GL(5;\mathbb{C})$ commutes with wedge). 
Thus we can apply Lemma \ref{lower triangular technique lemma} to the curve $\mu(z)$. It follows from the proof of Lemma \ref{lower triangular technique lemma} that now 
$$L_{01}=L_{21}=L_{31}=L_{41}=L_{51}=L_{61}=0, \quad L_{05}=L_{15}=L_{25}=L_{35}=L_{45}.$$
%$L$ is in the form of \eqref{the form of L Lower triangular}. 
%the Lemma \ref{lower triangular technique lemma} works, and $L$ is in the form \eqref{the form of L Lower triangular}. 
To prove that $L$ is diagonal, we need only show $L_{65}=0$.
%In the following, we show that $L$ is diagonal by showing $L_{65}=0$ under the constant curvature assumption.

%where $G$ is in the form of \eqref{simple form} as proved above. Write the lower triangular matrix $L_{1}^{-1}$ by $(b_{ij})_{0\leq i,j\leq 4}$. Compare the second columns of both sides of \eqref{compare column final lowTri}, 
Since $G$ is in the form of \eqref{simple form}, comparing the second column of both sides of  
\begin{equation}\label{compare column final lowTri}
A^{-1}\cdot G=(E_0,\ldots,E_6)\, L,
\end{equation}
we deduce
\[G_{11}\,A^{-1}\cdot e_0 \wedge  A^{-1}\cdot e_2 = L_{11}e_0\wedge e_2,\]
whence $ A^{-1}\cdot e_2 \in \Span\{e_0,e_2\}$. %=(b_{22}e_2+b_{32}e_3+b_{42}e_4)\in \Span\{e_0,e_2\}$, thus 
%\begin{equation}\label{key observation}
%b_{32}=b_{42}=0.
%\end{equation}
Then comparing the penultimate column of both sides of \eqref{compare column final lowTri}, %and using \eqref{key observation}, 
we have 
\[L_{55} e_{2}\wedge e_4+L_{65} e_3\wedge e_4=G_{85}\,A^{-1}\cdot e_2 \wedge  A^{-1}\cdot e_4\in \Span\{e_0\wedge e_4, e_2\wedge e_4\},\] %= G_{85}b_{22}b_{44}e_2\wedge e_4,\]
which implies $L_{65}=0$. Hence $L$ is diagonal. 

Furthermore, due to that $A$ is lower-triangular, we can also derive that 
$$A^{-1}\cdot e_2\equiv 0\mod e_2,\quad\quad\quad A^{-1}\cdot e_4\equiv 0\mod e_4,$$
and then $A^{-1}\cdot e_0\equiv 0\mod e_0$. By comparing the first and last columns of both sides of \eqref{compare column final lowTri}, we have $A^{-1}\cdot e_i\equiv 0\mod e_i,~i=1,3$.   In conclusion, we have arrived at that $A$ is diagonal. Therefore, the curve $\gamma$ belongs to the diagonal family.
%by Lemma \ref{lower triangular technique lemma}. 
\iffalse
It follows from the first two columns of $G$ that 
$$Ae_0\wedge Ae_1 \equiv 0 \mod (e_0\wedge e_1), \quad \quad \quad Ae_0\wedge Ae_2 \equiv 0 \mod (e_0\wedge e_2).$$
Hence, $Ae_0,Ae_1 \in \Span \{e_0,e_1\}$, and $Ae_0,Ae_2 \in \Span \{e_0,e_2\}$. As a result, $Ae_0\equiv 0\mod e_0$. %, i.e., $a_{j0}=0,~j=1,\ldots,4$. 
Due to that $A$ is lower-triangular, we obtain 
$Ae_1 \equiv 0\mod e_1$ and $Ae_2\equiv 0\mod e_2$. 
%hence, $a_{k1}=0,~k=2,\ldots,4$, and $a_{l2}=0,~l=3,4$. 
A similar observation on the $5$th column of $G$ gives that $Ae_2\wedge Ae_3 \equiv 0 \mod (e_2\wedge e_3)$, which implies $A e_3\equiv 0 \mod e_3$. In conclusion, we derive that $A$ is diagonal.
%Moreover, $L_1$ is also diagonal by Corollary \ref{lower-triangular lemma}, 
Therefore the curve $\gamma$ belongs to the diagonal family.
\fi
\end{demo}
\iffalse
Immediately we obtain the following corollary
\begin{corollary}
Let $\gamma:\mathbb{C}P^1\rightarrow G(2,5) $ be a rational normal curve of degree $6$ belongs to the Lower-Triangular Family. If one of the following statement holds,
\begin{enumerate}
\item if $\gamma$ is ramified at $n$ distinct points, and $n\neq 2${\emph ;}
\item or if $\gamma$ is ramified at two distinct points, but the multiplicity of one point is strictly less than $2${\emph ;}
\end{enumerate}
then $\gamma$ is not of constant curvature.
\end{corollary}
\fi

\section{Existence and uniqueness results for the diagonal family.}\label{sec7}
It follows from Theorem \ref{thm-tangential} that to classify generally ramified  constantly curved holomorphic $2$-spheres in $G(2,5)$, we need only consider those in the diagonal family, which are determined by diagonal matrices $A \in GL(5,\mathbb{C})$ and complex numbers $\{\omega_0, \omega_1, \ldots, \omega_6\}$ satisfying  %satisfying \eqref{perturbed}. %in the standard Fano $3$-fold $\mathcal{H}^3_0$.
{\small\begin{equation}\label{perturbed}
\aligned
&\omega_0\omega_4-4\omega_1\omega_3+3\omega_2^2=0,\quad \omega_0\omega_5-3\omega_1\omega_4+2\omega_2\omega_3=0,\quad\omega_0\omega_6-9\omega_2\omega_4+\\
&8\omega_3^2=0,\quad 
\omega_2\omega_6-4\omega_3\omega_5+3\omega_4^2=0,\quad
\omega_1\omega_6-3\omega_2\omega_5+2\omega_3\omega_4=0,
\endaligned
\end{equation}}
to guarantee that the holomorphic $2$-sphere parameterized as in \eqref{GenericG} lives in $G(2,5)$. 

In this section, we will pin down the class of diagonal matrices $A\in GL(5,\mathbb{C})$ that warrants the existence of constantly curved holomorphic $2$-spheres of degree $6$, and meanwhile find the number of such  $2$-spheres in each of these Fano $3$-folds $A(\mathcal{H}^3_0)$.
Assume $\varphi $ is a constantly curved holomorphic $2$-sphere in the diagonal family given by the data $A=\diag\{a_{00}, a_{11}, \cdots, a_{44}\}$ and  $\{\omega_0, \omega_1, \cdots, \omega_6\}$ satisfying \eqref{perturbed}. It follows from Definition~\ref{defn2} that 
{\small 
\begin{equation}\label{Ccurve}
 \aligned
& \varphi (z)=a_{00}a_{11}\omega_0\,e_0\wedge e_1+ \sqrt{6}a_{00}a_{22}\,\omega_1\, z\, e_0\wedge e_2 
 +3a_{00}a_{33}\,\omega_2\, z^2 \,e_0\wedge e_3\\
&+ \sqrt{6}a_{11}a_{22}\, \omega_2\, z^2\, e_1\wedge e_2
+2a_{00}a_{44}\,\omega_3\,z^3\, e_0\wedge e_4
+4a_{11}a_{33}\,\omega_3\,z^3 \,e_1\wedge e_3\\
&+3a_{11}a_{44}\,\omega_4\,z^4 \,e_1\wedge e_4
+\sqrt{6}a_{22}a_{33}\,\omega_4\,z^4 \,e_2\wedge e_3
+\sqrt{6}a_{22}a_{44}\,\omega_5\, z^5 \,e_2\wedge e_4\\
&+a_{33}a_{44}\,\omega_6\,z^6\, e_3\wedge e_4, 
\endaligned
\end{equation}
}
and %Moreover, since $\varphi $ is of constant curvature, there follows the length constraints
{\small 
\begin{equation}\label{krull}
\aligned
&\frac{(9a_{00}^2a_{33}^2+6a_{11}^2a_{22}^2)|\omega_2|^2}{15}
=\frac{(a_{00}^2a_{44}^2+4a_{11}^2a_{33}^2)|\omega_3|^2}{5}=a_{00}^2a_{11}^2|\omega_0|^2=\\&\frac{(9a_{11}^2a_{44}^2+6a_{22}^2a_{33}^2)|\omega_4|^2}{15}
=a_{00}^2a_{22}^2|\omega_1|^2=a_{22}^2a_{44}^2|\omega_5|^2=a_{33}^2a_{44}^2|\omega_6|^2=1.
\endaligned
\end{equation}
}
\begin{remark}\label{rk4}
We point out that $\varphi$ has the following standard parameterization in the sense of section {\rm \ref{JP}}. %give a new $1$-parameter family of examples in terms of Jiao-Peng's language (see section {\rm \ref{JP}}) to characterize the diagonal family as being equivalent to %%multiplying the $j$-th column of \eqref{JP1} by $a_{jj}$, where $j=0,\ldots,4$. Hence, \eqref{Ccurve}, in comparison to~\eqref{JP1}, is converted to
\begin{equation}\label{JP2}
\begin{pmatrix}\varphi_1(z)\\ \varphi_2(z)\end{pmatrix}=
\begin{pmatrix}
1 & 0 & -\sqrt{6}\frac{\omega_2 a_{22}}{\omega_0 a_{00}}\,z^2 & -4\frac{\omega_3 a_{33}}{\omega_0 a_{00}}\,z^3 & -3\frac{\omega_4 a_{44}}{\omega_0 a_{00}}\,z^4\\
0 &1 & \sqrt{6}\frac{\omega_1 a_{22}}{\omega_0 a_{11}}\,z& 3\frac{\omega_2 a_{33}}{\omega_0 a_{11}}\,z^2 & 2\frac{\omega_3 a_{44}}{\omega_0 a_{11}}\,z^3
\end{pmatrix}.
\end{equation}
In Jiao and Peng's approach, they considered collectively the undertermined variables
{\small\[\aligned
&\alpha_2\triangleq  -\sqrt{6}(\omega_2 a_{22})/(\omega_0 a_{00}),~\beta_3\triangleq -4(\omega_3 a_{33})/(\omega_0 a_{00}),~\varphi _4\triangleq  -3(\omega_4}{a_{44})/(\omega_0}{a_{00}),\\
&u_1\triangleq \sqrt{6}(\omega_1 a_{22})/(\omega_0 a_{11}), ~v_2\triangleq 3(\omega_2 a_{33})/(\omega_0 a_{11}),~z_3\triangleq 2(\omega_3 a_{44})/(\omega_0 a_{11}).
\endaligned\]}
Then the constant curvature condition \eqref{krull} is equivalent to 
 
{\small\begin{align}\label{constant curvature 
in JP}
\begin{split}
&|u_1|^2=6,~|v_2|^2+|\alpha_2|^2=15,~|z_3|^2+|\beta_3|^2=20\\
&|\varphi _4|^2+|\alpha_2v_2-\beta_3u_1|^2=15,~|\alpha_2 z_3-\varphi _4u_1|^2=6,~|\beta_3z_3-\varphi _4v_2|^2=1.
\end{split}
\end{align}}

%\iffalse

The standard Veronese curve in \eqref{eq-standard} %%in Remark {\rm \ref{onlykonwnexample}} 
corresponds to the solution
{\small\[(\alpha_2,\beta_3,\varphi _4,u_1,v_2,z_3)=(-\sqrt{6},-4,-3,\sqrt{6},3,2).\]
Branching out, observe that after fixing $(\alpha_2,\varphi _4,u_1,v_2)=(-\sqrt{6},-3,\sqrt{6},3)$, we have that the system of equations \eqref{constant curvature in JP} reduces to 
{\small\[|z_3|^2+|\beta_3|^2=20,\;\,|\beta_3+3|^2=1,\;\,|z_3-3|^2=1,\;\,|\beta_3z_3+9|^2=1.\]}}
Set 
\begin{equation}\label{fmly}
\beta_3\triangleq -3+e^{\sqrt{-1}\theta},~z_3\triangleq 3+e^{\sqrt{-1}\varphi}.
\end{equation} 
From the first equation we derive
$\cos\theta=\cos\varphi;$
and so $\varphi=\pm \theta$. If $\varphi=-\theta$, then the last equation above gives $\theta=0$ or $\pi$. Therefore without losing generality, we may set $\varphi=\theta$ in any event. Consequently, we obtain a $1$-parameter family of solutions
 \begin{equation}\label{family33}
%\begin{pmatrix}f(z)\\g(z)\end{pmatrix}=
\begin{pmatrix}
1 & 0 & -\sqrt{6}z^2 & (-3+e^{\sqrt{-1}\theta})z^3 & -3z^4\\
0 & 1 & \sqrt{6}z & 3z^2 &  (3+e^{\sqrt{-1}\theta})z^3\\
\end{pmatrix},
\end{equation}
hitherto unknown in the literature, to the authors' knowledge.

Though the simple perturbation \eqref{fmly} generates the explicit $1$-parameter family \eqref{family33}, in general, however, without further geometric clue it is a difficult task to completely classify the system~\eqref{constant curvature in JP}. As our analysis has revealed up to now, the nature of the classification lies in that one must perturb in certain Fano $3$-folds dictated by \eqref{Ccurve} to achieve the classification. In the following, we will present an algebro-geometric approach to describe all solutions to the diagonal system \eqref{Ccurve}. 

\end{remark}

\iffalse
\noindent from which we condense the preceding Theorem \ref{thm-tangential } in the following concrete statement.

\begin{proposition}
Any constantly curved rational normal curve $\varphi $ of degree $6$ of the diagonal family living in ${\mathcal H}^3=A({\mathcal H}^3_0)$, where diagonal matrix $A\triangleq \diag(a_{00},\cdots,a_{44})$, is the transformation by $A$ of a curve of the form
{\small $$
[\omega_0: \omega_1\sqrt{6} z:\cdots:\omega_5\sqrt{6} z^5:\omega_6\, z^6]
$$
satisfying the constraints in~\eqref{perturbed}, where%%in the perturbed family~\eqref{cCURVE} with
\begin{equation}\label{T}
\omega_i\triangleq e^{\sqrt{-1}\theta_i}/|T_i|,~~~0\leq i\leq 6,\quad\text{and}
\end{equation}}
%%where $|T_i|$ is the norm of the $i${th} column of $\bf B$ in~\eqref{DIAG} given by
$$
\aligned
&|T_0|^2=a_{00}^2a_{11}^2,~~~|T_1|^2=a_{00}^2a_{22}^2,~~~|T_5|^2=a_{22}^2a_{44}^2,~~~|T_6|^2=a_{33}^2a_{44}^2,\\
&\!\!\!\!|T_2|^2=\frac{9a_{00}^2a_{33}^2+6a_{11}^2a_{22}^2}{15},|T_3|^2=\frac{a_{00}^2a_{44}^2+4a_{11}^2a_{33}^2}{5},|T_4|^2=\frac{9a_{11}^2a_{44}^2+6a_{22}^2a_{33}^2}{15}.
\endaligned
$$
\end{proposition}
\fi

Set 
%{\small
$$\omega_i\triangleq \sqrt{t_i} e^{\sqrt{-1}\theta_i},\quad\quad i=0,\ldots,6.$$%} 
It follows from the condition of constant curvature \eqref{krull} that 
{\small\begin{equation}\label{eq-t01234}
\aligned
&t_0=1/a_{11}^2,~~~t_1=1/a_{22}^2,~~~~t_2={15}/{(9a_{00}^2a_{33}^2+6a_{11}^2a_{22}^2)},~~t_6=1/(a_{33}^2a_{44}^2),\\
&t_3={5}/{(a_{00}^2a_{44}^2+4a_{11}^2a_{33}^2)},~t_4={15}/{(9a_{11}^2a_{44}^2+6a_{22}^2a_{33}^2)},~~t_5=1/(a_{22}^2a_{44}^2).
\endaligned
\end{equation}}
\begin{remark}\label{rrk} 
For the detailed analysis to follow on the length constraints \eqref{krull}, without loss of generality through scaling, we may assume that $a_{00}=1$ and $a_{jj}\in \mathbb{R}^{+},~1\leq j\leq 4$ {\rm(}by a diagonal unitary transformation in $U(5)${\rm)}. Moreover, it follows from Lemma~{\em \ref{usefullemma}} that the transformation $\rho^4(\diag\{\lambda,1\})=\diag\{1, \lambda, \lambda^2, \lambda^3, \lambda^4\}$ 
preserves ${\mathcal H}_0^3$ for any $\lambda\in \mathbb{C}^*$. As a consequence, after multiplying by an appropriate real $\lambda$, we may furthermore assume  $a_{22}=a_{33}$. This process is equivalent to applying a M\"{o}bius reparametrization to the $2$-sphere $\varphi$ by $z\mapsto\lambda z$.

Similarly, we assume further that $\theta_0=\theta_6=0$, which follows from dehomogenizing to eliminate $\theta_0$ and introducing a rotational reparametrization of the $2$-sphere $\varphi $ to eliminate $\theta_6$.

\end{remark}

%$$t_0\triangleq{1}/{a_{11}^2}, \quad \quad \quad t_1\triangleq{1}/{a_{22}^2}, \quad \quad \quad t_6\triangleq{1}/({a_{33}^2a_{44}^2}).$$

%$t_i\triangleq {1}/{|T_i|^2},~i=0,\ldots,6$, then we arrive at that $|\omega_i|=\sqrt{t_i}$ and (see two lines below \eqref{T})
%{\small
Combining \eqref{eq-t01234} with the above normalization, we have  
{\small\begin{equation}\label{t234}%%{t016}
\aligned
%&t_0={1}/{a_{11}^2},\;t_1={1}/{a_{22}^2},\;t_6={1}/{a_{33}^2a_{44}^2},\;t_5=t_6,\\
%%\begin{tiny}\begin{equation}\label{t234}
&t_2=\frac{5 t_0 t_1}{(3 t_0+2)},\;t_3=\frac{5 t_0 t_1 t_6}{(t_0 t_1^2+4 t_6)},\;t_4=\frac{5 t_0 t_1^2 t_6}{(3 t_1^3+2 t_0 t_6)}, \;t_5=t_6. 
\endaligned
\end{equation}}
Moreover, it follows from \eqref{perturbed} that the angles $\theta_i$ of $\omega_i$ satisfy
{\small
 \begin{equation}\label{angle}
\aligned
&\sqrt{t_0t_4}=4\sqrt{t_1t_3}e^{\sqrt{-1}(\theta_1+\theta_3-\theta_4)}-3t_2e^{\sqrt{-1}(2\theta_2-\theta_4)},\\
&\sqrt{t_0t_5}=3\sqrt{t_1t_4}e^{\sqrt{-1}(\theta_1+\theta_4-\theta_5)}-2\sqrt{t_2t_3}e^{\sqrt{-1}(\theta_2+\theta_3-\theta_5)},\\
&\sqrt{t_0t_6}=9\sqrt{t_2t_4}e^{\sqrt{-1}(\theta_2+\theta_4)}-8t_3e^{\sqrt{-1}2\theta_3}.
\endaligned
\end{equation}}
\begin{remark}\label{3imply5}
Conversely, given a solution 
{\small$\{t_0,\,t_1\, \cdots, t_6\}$ $\subset$ $\mathbb{R}^+$} and {\small$\{\,\theta_1\, \cdots, \theta_5\}$ $\subset$ $\mathbb{R}$}
to \eqref{t234} and \eqref{angle}, by solving $a_{ii}$ from $t_i$ and defining $\omega_i=t_i e^{\sqrt{-1}\theta_i}$, we can obtain a constantly curved holomorphic $2$-sphere of degree $6$ in $G(2,5)$ parameterized as in \eqref{Ccurve}. %%Here, we have used the fact that the last two equations in \eqref{perturbed} can be deduced from the first three (since $\omega_0\neq 0$), which are equivalent to the ones in \eqref{angle}. 
\end{remark}
%\begin{remark}\label{3imply5}
%Due to $\omega_0\neq 0$, the above three equations actually imply the remaining two in \eqref{perturbed}.
%\end{remark}

We point out that the three equations in \eqref{angle} are not independent by the following Lemma \ref{solve by XYZ}. In fact, set
{\small 
\begin{equation}\label{defining0}
\aligned
&x_1\triangleq e^{\sqrt{-1}(\theta_1+\theta_3-\theta_4)},\;y_1\triangleq e^{\sqrt{-1}(2\theta_2-\theta_4)},\;
x_2\triangleq e^{\sqrt{-1}(\theta_1+\theta_4-\theta_5)},\;\\&y_2\triangleq e^{\sqrt{-1}(\theta_2+\theta_3-\theta_5)},\; x_3\triangleq e^{\sqrt{-1}(\theta_2+\theta_4)},\;\;\;y_3\triangleq e^{\sqrt{-1}(2\theta_3)}.
\endaligned
\end{equation}}

Taking norm squared on both sides of \eqref{angle}, we see from the realness of $t_0,\cdots,t_6$ that 
{\small 
\begin{equation}\label{DEF}
\aligned
&h_1\triangleq v-uw=0,\quad  h_2\triangleq u^2-X u + 1=0, \quad h_3\triangleq v^2-Y v + 1=0,\\
&h_4\triangleq w^2-Zw + 1=0,
\endaligned
\end{equation}}
where, 
{\small 
\begin{equation}\label{XYZt0t1t6}
\aligned
&u= x_1/y_1,\quad v= x_2/y_2,\quad w= x_3/y_3,\\
&X= (9 t_2^2+16 t_1 t_3-t_0 t_4)/(12 t_2 \sqrt{t_1 t_3}),\\
&Y= (4 t_2 t_3+9 t_1 t_4-t_0 t_5)/(6 \sqrt{t_2 t_3}\sqrt{t_1t_4}),\\
&Z= (64 t_3^2+81 t_2 t_4-t_0 t_6)/(72 t_3 \sqrt{t_2 t_4}).
\endaligned
\end{equation}
}

We first solve \eqref{DEF} by viewing $\{X,Y,Z\}$ as indeterminates.  Define 
{\small \begin{equation}\label{DH}
H\triangleq -XYZ+X^2+Y^2+Z^2-4.
\end{equation}}
\iffalse
Consider the polynomial ideal 
$I\triangleq \langle f_1,f_2,f_3,f_4\rangle\subseteq \mathbb{C}[v,u,w,X,Y,Z]$ generated by \eqref{DEF}.
Under the lexicographic order with respect to $\{v,u,w,X,Y,Z\}$, the Groebner Basis of $I$ is given by 
%{\small 
\begin{align*}
G:&=\{-XYZ+{X}^{2}+{Y}^{2}+{Z}^{2}-4,\quad-Zw+{w}^{2}+1,\\
&-X{Z}^{2}+XZw+{Z}^{2}u+Y
Z-2\,Yw+2\,X-4\,u,\\
&-XYZ+XYw-YZu+2\,Xu+{Y}^{2}+2\,{Z}^{2}-2\,Zw-4,\\
&XZ-Xw-Zu+2\,uw-Y,\quad XZ-Xw-Zu-Y+2\,v,\\
&-XYZ+XYw-YZu+{Y}^{2}+2\,{Z}^{2}-2\,Zw+2\,{u}^{2}-2\}.
\end{align*}
%}
Consequently, by the Elimination Theory (see \cite[Thm 2, p.~116]{CLD2007}), the $3$rd elimination ideal $%I_3\triangleq 
I\cap \mathbb{C}[X,Y,Z]=\langle H \rangle$ is principal and generated by 
%{\small 
\begin{equation}\label{DH}
H\triangleq -XYZ+X^2+Y^2+Z^2-4.
\end{equation}
%}
Since the leading coefficients of $f_1,\ldots,f_4$, under the lexicographic order with respect to $\{v,u,w,X,Y,Z\}$ are non-zero, a repeated application of the Extension Theorem (see \cite[Cor 4, p.~120]{CLD2007}) yields the following. 
\fi 
\begin{lemma}\label{solve by XYZ} If $\{v, u, w, X, Y, Z\}$ solves  the system \eqref{DEF}, then  $H=0$. Conversely, given any complex solution $(X_0,Y_0,Z_0)$ to $H=0$, there always exits $(v_0,u_0,w_0)\in \mathbb{C}^3$, such that $(v_0,u_0,w_0,X_0,Y_0,Z_0)$ solves this system.

Moreover, when the solution $X_0,Y,Z_0$ to $H=0$ are real, $|v_0|=|u_0|=|w_0|=1$ if and only if $X_0,Y_0,Z_0\in [-2,2]$, in which case there are at most two solutions, namely, $(v_0,u_0,w_0)$ and its complex conjugate $(\overline{v_0},\overline{u_0},\overline{w_0})$, which are distinct unless $X_0^2=Y_0^2=Z_0^2=4$ and $X_0Y_0Z_0=8$.
\end{lemma}
\begin{proof}
Assume $\{v, u, w\}$ solves the last three equations in  \eqref{DEF}, respectively. It follows %%from the Vieta theorem 
that $\{1/v, 1/u, 1/w\}$ also solves them, respectively, with 
$X=u+1/u,\,Y=v+1/v, \,Z=w+1/w.$
By a straightforward calculation, we have 
{\small$$H={(uvw-1)(u-vw)(v-uw)(w-uv)}/(u^2 v^2 w^2),$$}
from which the first statement follows by the first equation of \eqref{DEF}.

%%Using the Vieta theorem again and 
To prove the second statement, the realness of  $X_0,Y_0,Z_0$ dictates that $|v_0|=|u_0|=|w_0|=1$ if and only if the last three equations in  \eqref{DEF} all have a pair of conjugate solutions, which implies that their discriminants
$X_0^2-4,\;Y_0^2-4, \; Z_0^2-4$
are no more than 0. 
%It remains to show the second statement. On the one hand, for instance, assume that $u_0=e^{\sqrt{-1}\theta}$. Then it follows from the second equation $f_2$ in \eqref{DEF} that $X_0=2\cos\theta\in [-2,2]$. The same goes with $f_3$ and $f_4$. %%the other two equations. %%The other hand holds directly.

Furthermore, given $(X_0,Y_0,Z_0)\in [-2,2]^3$ that solves \eqref{DH}, assume $\{(v_i,u_i,w_i)|i=0,1\}$ are two pairs of solutions of the system \eqref{DEF}.  It follows that 
%{\small
 \[v_1=v_0~\text{or}~\overline{v_0},~\quad \quad \quad u_1=u_0~\text{or}~\overline{u_0},\quad \quad \quad~w_1=w_0~\text{or}~\overline{w_0}.\]
 %}
By the pigeonhole principle, we may assume  $u_1=\overline{u_0},~w_1=\overline{w_0}$  without loss of generality. Then it follows from the first equation $h_1$ in $\eqref{DEF}$ that $v_1=u_1w_1=\overline{v_0}$. Therefore, we deduce that these two solutions either coincide or differ by a complex conjugation, where the former case occurs when $u_0,v_0,w_0$ are all real to satisfy $X_0=Y_0=Z_0=\pm 2$ with $X_0Y_0Z_0=8$ to respect $H=0$.
\end{proof}

\begin{remark} The cubic surface $H=0$ with $|x|,|y|,|z|\leq 2$ is a semialgebraic sphere. %%One can see this by regarding $z$ as a parameter so that, for each fixed $|z|<2$, $H=0$ represents an ellipse $x^2 -zxy+ y^2 =4$ , which degenerates to two lines at $z=\pm 2$.

\begin{figure}[htbp]
\centering
\includegraphics[width=0.20\textwidth]{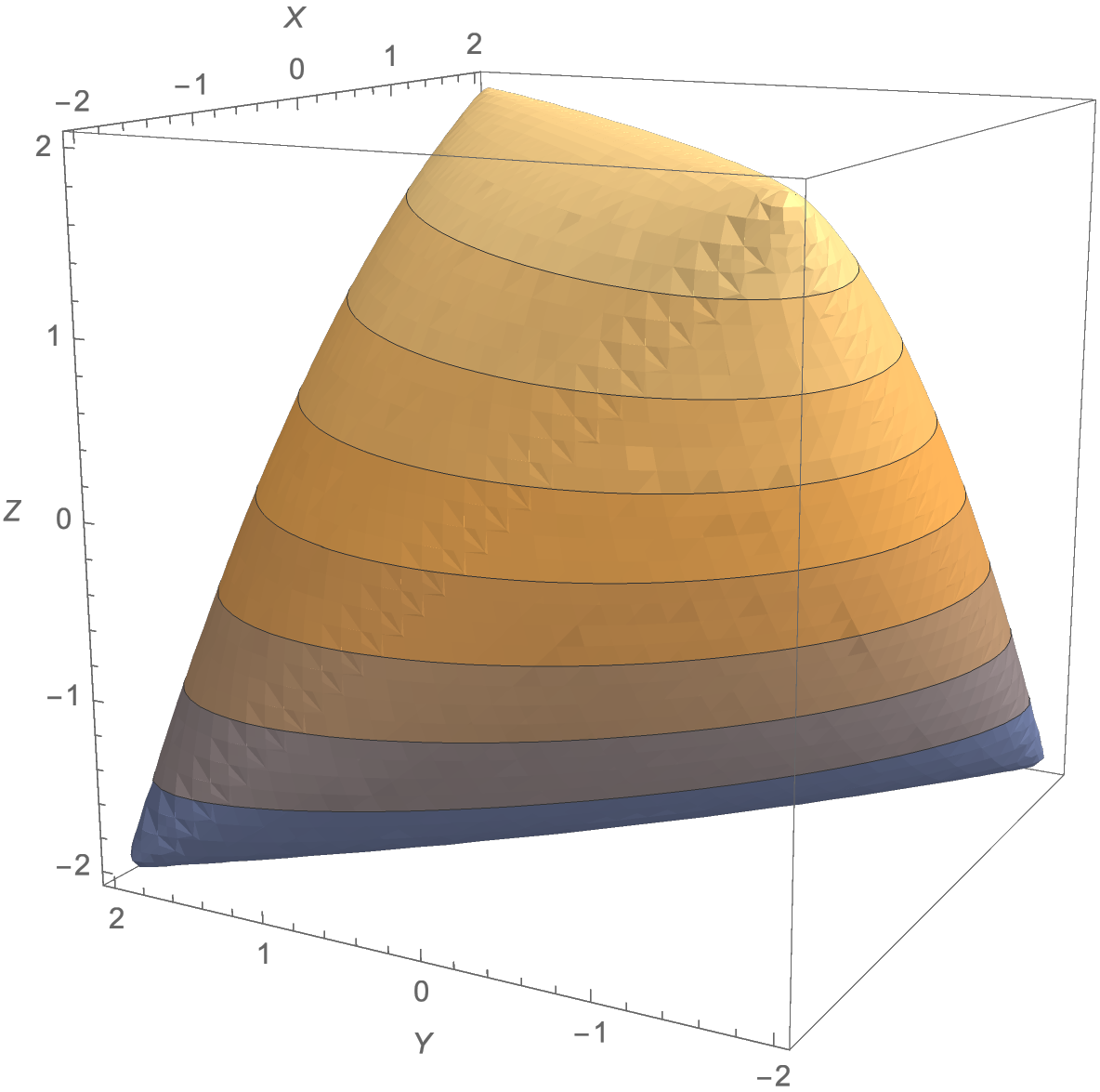}
\caption{Semialgebraic sphere $H=0$}
\end{figure}

\end{remark}

We now analyse the diagonal family in terms of $(t_0, t_1, t_6)\in(\mathbb{R}^{+})^{3}$. By substituting \eqref{t234} and \eqref{XYZt0t1t6} into the formula of $H$ in \eqref{DH} and ignoring the nonzero denominator of the fraction and the nonzero factors, we obtain a hypersurface in $(\mathbb{R}^{+})^{3}$ defined by $F(t_0,t_1,t_6)=0$, where
\begin{tiny}
\allowdisplaybreaks
\begin{align}
%\begin{split}
&F(t_0,t_1,t_6)\triangleq \, {168750000~H\, t_0^6t_1^{11}t_6^4}/(t_{{2}}t_{{3}}t_4^2)\nonumber\\
&=9\,{t_{{1}}}^{6}{t_{{6}}}^{3}{t_{{0}}}^{9}+ \Big( 6912\,{t_{{1}}}^{9}
{t_{{6}}}^{2}-366\,{t_{{1}}}^{6}{t_{{6}}}^{3}-10260\,{t_{{1}}}^{4}{t_{
{6}}}^{4} \Big) {t_{{0}}}^{8}\nonumber\\
&+\Big( 435888\,{t_{{1}}}^{2}{t_{{6}}}
^{5}+299592\,{t_{{1}}}^{4}{t_{{6}}}^{4}+ ( -397332\,{t_{{1}}}^{7}
+2560\,{t_{{1}}}^{6} ) {t_{{6}}}^{3}-58329\,{t_{{1}}}^{9}{t_{{6}
}}^{2}+63504\,{t_{{1}}}^{12}t_{{6}} \Big) {t_{{0}}}^{7}\nonumber\\
&+\Big(
65088\,{t_{{6}}}^{6}+225504\,{t_{{1}}}^{2}{t_{{6}}}^{5}+ ( 31968
\,{t_{{1}}}^{5}+533856\,{t_{{1}}}^{4} ) {t_{{6}}}^{4}+ ( -
451260\,{t_{{1}}}^{7}-128\,{t_{{1}}}^{6} ) {t_{{6}}}^{3}+\nonumber\\
&\;\;\;\;\;( -1296\,{t_{{1}}}^{10}-44868\,{t_{{1}}}^{9} ) {t_{{6}}}^{
2}+16416\,{t_{{1}}}^{12}t_{{6}} \Big) {t_{{0}}}^{6}\nonumber\\
&+ \Big( 78720\,{
t_{{6}}}^{6}+ ( -1366848\,{t_{{1}}}^{3}+154368\,{t_{{1}}}^{2}
 ) {t_{{6}}}^{5}+ ( -2480688\,{t_{{1}}}^{5}+203712\,{t_{{1}
}}^{4} ) {t_{{6}}}^{4}+ (2125440\,{t_{{1}}}^{8}+\nonumber\\
&\;\;\;\;\;541536\,{t
_{{1}}}^{7} ) {t_{{6}}}^{3}+ ( -501336\,{t_{{1}}}^{10}+2560
\,{t_{{1}}}^{9} ) {t_{{6}}}^{2}+ ( -190512\,{t_{{1}}}^{13}-
58329\,{t_{{1}}}^{12} ) t_{{6}}+63504\,{t_{{1}}}^{15} \Big) {t
_{{0}}}^{5} \label{eq-alge}\\
&+ \Big( 22016\,{t_{{6}}}^{6}+ ( 15552\,{t_{{1}}}^{3}+
99840\,{t_{{1}}}^{2}) {t_{{6}}}^{5}+ ( 145152\,{t_{{1}}}^{
6}-2192448\,{t_{{1}}}^{5}) {t_{{6}}}^{4}+ ( 1076544\,{t_{{
1}}}^{8}+\nonumber\\
&\;\;\;\;\;533856\,{t_{{1}}}^{7} ) {t_{{6}}}^{3}+ ( 31104\,{t
_{{1}}}^{11}-451260\,{t_{{1}}}^{10} ) {t_{{6}}}^{2}+ ( -
1296\,{t_{{1}}}^{13}-366\,{t_{{1}}}^{12} ) t_{{6}}+6912\,{t_{{1}
}}^{15} \Big) {t_{{0}}}^{4}\nonumber\\
&+ \Big( -1024\,{t_{{6}}}^{6}-645120\,{t_
{{1}}}^{3}{t_{{6}}}^{5}+ ( 5774976\,{t_{{1}}}^{6}+154368\,{t_{{1}
}}^{5} ) {t_{{6}}}^{4}+ ( -3048192\,{t_{{1}}}^{9}-2480688\,
{t_{{1}}}^{8} ) {t_{{6}}}^{3}+\nonumber\\
&\;\;\;\;\;( 2125440\,{t_{{1}}}^{11}+
299592\,{t_{{1}}}^{10} ) {t_{{6}}}^{2}-397332\,{t_{{1}}}^{13}t_{
{6}}+9\,{t_{{1}}}^{15} \Big) {t_{{0}}}^{3}\nonumber\\
&+ \left( 22016\,{t_{{1}}}^
{3}{t_{{6}}}^{5}+15552\,{t_{{1}}}^{6}{t_{{6}}}^{4}+ ( 145152\,{t_
{{1}}}^{9}+225504\,{t_{{1}}}^{8} ) {t_{{6}}}^{3}+31968\,{t_{{1}}
}^{11}{t_{{6}}}^{2}-10260\,{t_{{1}}}^{13}t_{{6}} \right) {t_{{0}}}^{2}\nonumber\\
&+ \left( 435888\,{t_{{1}}}^{11}{t_{{6}}}^{2}-1366848\,{t_{{1}}}^{9}{t_
{{6}}}^{3}+78720\,{t_{{1}}}^{6}{t_{{6}}}^{4} \right) t_{{0}}+65088\,{t
_{{1}}}^{9}{t_{{6}}}^{3}=0,\nonumber
%\end{split}
\end{align}
\end{tiny}
 with the three necessary discriminant constraints
{\small \begin{equation}\label{inequality}
\aligned
&(9 t_2^2+16 t_1 t_3-t_0 t_4)^2-576 t_1t_2^2t_3\leq 0,~~~
(4 t_2 t_3+9 t_1 t_4-t_0 t_5)^2-\\&144t_1t_2t_3t_4\leq 0,~~~
(64 t_3^2+81 t_2 t_4-t_0 t_6)^2-20736 t_2t_3^2t_4\leq 0,
\endaligned
\end{equation}}
 thanks to the assumptions made on $X, Y, Z\in [-2,2]$ in Lemma \ref{solve by XYZ}: 
 
%%Moreover, the leads to 
\begin{remark}\label{three inequality not indep}
%%Before proceeding further, let us 
The three constraints $|u|=|v|=|w|=1$ are not independent by the first equation in \eqref{DEF}. Any two of the three inequalities in \eqref{inequality} imply the third. Moreover, $Z\in(-2,2)$ implies $X,Y\in(-2,2)$ since for a fixed $Z\in (-2,2)$, $H=0$ in \eqref{DH}
defines an ellipse good for the conclusion. 
\end{remark}

In conclusion, we obtain the following existence and uniqueness theorem.
\begin{theorem}\label{prop-exist}
Given a diagonal matrix $A=\diag\{1, a_{11}, a_{22}, a_{22}, a_{44}\}$, normalized as in Remark {\rm\ref{rrk}}, there exists a sextic curve $\gamma$ belonging to the generally ramified family in $\mathcal{H}^3_0$ such that $A(\gamma)$ is of constant curvature,     %constantly curved holomorphic $2$-spheres of degree $6$ in $A({\mathcal H}_0^3) $ %%{\rm (}see also \eqref{GenericG}{\rm )}, 
if and only if $\{t_0, t_1, t_6\}$ given by \eqref{t234} satisfies the algebraic equation \eqref{eq-alge} and inequalities \eqref{inequality}. 

Moreover,  in $A({\mathcal H}_0^3) $, there exist at most two constantly curved holomorphic $2$-spheres of degree $6$ belonging to the generally ramified family; they are distinct except when $\{X, Y, Z\}$ defined in \eqref{XYZt0t1t6}  satisfies $X^2=Y^2=Z^2=4$ and $XYZ=8$. 
\end{theorem}
\begin{proof}
The necessary part has been verified in the preceding discussion. 

Conversely, assume that $\{t_0, t_1, t_6\}$ satisfy the algebraic equation \eqref{eq-alge} and inequalities~\eqref{inequality}. Then we obtain at least a triple $(v_0,u_0,w_0)$ of solution of system \eqref{DEF} according to Lemma \ref{solve by XYZ}. By substituting it into system \eqref{angle}, we obtain a unique solution $\{(x_i,y_i)| 1\leq i\leq 3\}$ by the following recipe: The first equation of \eqref{angle} gives that 
\begin{equation}\label{xdeterminedbyu}
y_1={\sqrt{t_0t_4}}/({4\sqrt{t_1t_3}\,u_0-3t_2}),~~\quad~x_1=y_1 {u_0}.
\end{equation}
It follows from $|u_0|=1$ that both $x_1$ and $y_1$ are of unit length. A similar discussion applies to $(x_2, y_2)$ and $(x_3, y_3)$. %The formulae for the other two pairs are similar. It is straightforward to verify that the norms of them are unital thanks to the system \eqref{DEF}.

Apply the logarithmic function on both sides of \eqref{defining0}. Since the ranks of the coefficient matrix of of $(\theta_1,\ldots,\theta_5)$ and its enlarged version with the augmented $(\log(x_1),\cdots,\log(y_3))$ 
 are both equal to $5$, we can solve $\theta_j$ from the arguments of the points $\{(x_i,y_i)| 1\leq i \leq 3\}$ on the plane. Substituting all the data into \eqref{Ccurve} gives a constantly curved holomorphic $2$-sphere $\varphi $ in $A({\mathcal H}_0^3)$ (see Remark \ref{3imply5}).

Lastly, we remark that $\varphi $ is uniquely determined by $(v_0,u_0,w_0)$, owing to that the only difference between any two pairs of solutions $\{\theta_j|1\leq j\leq 5\}$ and $\{\tilde{\theta_j}|1\leq j\leq 5\}$ of \eqref{defining0} is 
%{\small
 $\theta_j=\tilde{\theta_j}+{2kj\pi}/{6},~~~1\leq j\leq 5,$ %}
for some $0\leq k\leq 5$. 
It is straightforward to show that the corresponding two curves share the same image by introducing a rotational reparametrization $\tilde{z}=z e^{{\sqrt{-1}2k\pi}/{6}}$. %%The details are left to the reader. 

In conclusion, any solution $(v,u,w)$ of system \eqref{DEF} determines uniquely a constantly curved $2$-sphere. Then the second statement follows from Lemma \ref{solve by XYZ}.
\end{proof}

%We now look into the uniqueness of the problem.
\begin{corollary}\label{uniqueness}
The only constantly curved holomorphic $2$-sphere of degree $6$ in the standard Fano $3$-fold $\mathcal{H}_0^3$ tangent to the %%=V_6\cap G(2,5)$ 
the standard Veronese curve $PSL_2\cdot u^6$ is the Veronese curve itself.  
\end{corollary}
\begin{proof}
 %%(see also \eqref{xdeterminedbyu}).
For the standard Fano $3$-fold $\mathcal{H}_0^3$, the associated $\{t_0, t_1,t_6\}$ are all equal to $1$. Therefore the corresponding $X=Y=Z=2$ by \eqref{XYZt0t1t6}.
\end{proof}
\begin{remark}\label{RMK}
 In addition to the standard Fano $3$-fold $\mathcal{H}_0^3$, let us take the diagonal $A= \diag \{1,1,4,4,16\}$, there exists a unique constantly curved holomorphic $2$-sphere of degree $6$ belonging to the generally ramified family that lies in $A(\mathcal{H}_0^3)$ given by 
%{\small 
\[\begin{pmatrix}
1 & 0 & -\sqrt{6}z^2 &-2z^3 & -3z^4\\
0 & 1 & \sqrt{6}z & 3z^2 & 4z^3\\
\end{pmatrix},\]
%}
since the associated $X=Y=Z=2$. %; see also Example {\em \ref{ex-t01}} in Section {\em \ref{lastsection}} for details.
It turns out that among Fano $3$-folds ${\mathcal H}^3$ in $G(2,5)$, only three {\em (}up to unitary congruence{\em )} contain a unique constantly curved holomorphic $2$-sphere of degree $6${\emph ;} the last one will be given in Example~{\em \ref{eg-cusp}}. %owing to solving three algebraic equations $(X,Y,Z)=(\pm 2,\pm 2,\pm 2)$ {\emph (}with even minus symbols{\emph )} with respect to three positive parameters $\{t_0,t_1,t_6\}$.
\end{remark}

\section{The moduli space and new examples}\label{lastsection}

Before describing the moduli space of constantly curved holomorphic $2$-spheres belonging to the generally ramified family, we first consider the 
semialgebraic set $S\subseteq (\mathbb{R}^{+})^3$ determined by the algebraic equation \eqref{eq-alge} and the three inequalities \eqref{inequality}. %%which is the brown jagged portion of the upper right section in the following picture, in the $t_0t_1t_6$ coordinates.

%% \begin{figure}[htbp]
%%\centering
%%\includegraphics[width=0.5\textwidth]{moduli.png}
%%\end{figure}

\begin{proposition}\label{PROP7.1}
The semialgebraic set $S$ is $2$-dimensional and equipped with an involution 
%$\sigma: S  \rightarrow S$,
%%\begin{align}
\begin{equation}\label{sigmag}
%%\begin{split}
%%,\\
\sigma: S  \rightarrow S, \quad t=(t_0,t_1,t_6)  \mapsto T=(T_0,T_1,T_6)=(g\, t_0,g\, t_1,g^3\, t_6), 
%%\end{split}
%%\end{align}
\end{equation}
where $g(t_0,t_1,t_6)\triangleq t_1^3/(t_0^2t_6)$.
\end{proposition}

\begin{proof}
It is easy to show that $\sigma$ is an involution of $(\mathbb{R}^{+})^3$ restricted to $S$; consequently, we need only verify that $\sigma(S) \subseteq S$. 

Assume that $t=(t_0,t_1,t_6) \in S$, i.e., that  $t$ satisfies 
%{\small
 \[F(t)=0,~~\text{and}~~ X(t),Y(t),Z(t)\in [-2,2].\]%}
A direct computation yields that 
%{\small 
\[F(T)=g^{21} F(t)=0,~~~Z(T)=Z(t) \in [-2,2].\]%}

Note that the last equation of \eqref{perturbed} gives 
\[\sqrt{t_1t_6}=3\sqrt{t_2t_5}e^{\sqrt{-1}(\theta_2+\theta_5-\theta_1)}-2\sqrt{t_3t_4}e^{\sqrt{-1}(\theta_3+\theta_4-\theta_1)}.\]
Set $q=e^{\sqrt{-1}(\theta_2+\theta_5-\theta_3-\theta_4)}$. Then a similar argument to that deriving \eqref{DEF} leads to  
%{\small 
$
\aligned
&q^2-Qq+1=0, %\quad\text{where},\\
%&q=e^{\sqrt{-1}(\theta_2+\theta_5-\theta_3-\theta_4)},\quad 
\endaligned
$
where $$Q(t)\triangleq (-t_1t_6+9t_2t_5+4t_3t_4)/(6\sqrt{t_2t_3t_4t_5}).$$
%} 
Since $|q|=1$, it forces $Q(t)\in [-2,2]$. It is straightforward to show  that $Y(T)=Q(t)\in [-2,2]$. Therefore, combining Remark \ref{three inequality not indep}, we obtain that the norm of $X(T)$ is also less than or equal to $2$. This completes the proof that $T=\sigma(t)$ lies in $S$.

We are left with showing that the real dimension of %The next thing to do in the proof is to show  that 
the semialgebraic set $S$ is $2$. At the generic point $p_0=(1,\frac{1}{2},\frac{1}{8})\in S$ (for the choice of $p_0$, see Example \ref{example0} %%in Section \ref{lastsection} 
below for details). A calculation gives
{\small\[ \nabla F(p_0)=\big({\partial F}/{\partial t_0},{\partial F}/{\partial t_1},{\partial F}/{\partial t_6}\big)(p_0)=(0,-13125/256,4375/64)\neq 0.\]}
%%(and $X(p_0)\approx 1.947,~Y(p_0)\approx1.893,~Z(p_0)\approx1.990$). 
Owing to the implicit function theorem, near $p_0$, $S$ is locally a graph of $t_0$ and $t_1$; hence, its real dimension is $2$. 
\end{proof}
\begin{remark}\label{rk-involution}
%%It follows form the above proposition that there is a  $\mathbb{Z}_2$-action on $S$ induced by the involution $\sigma$. 
We point out that the involution $\sigma$ comes from the reciprocal transformation of\, $\mathbb{C}P^1$ {\em (}see the proof of the following Theorem{\em )}. %%Therefore, constantly curved $2$-spheres corresponding to a pair of involutive points in $S$ either coincide or differ by a complex conjugation, in perfect accord with Theorem~{\em \ref{prop-exist}}. 
%As a result, we obtain a $\mathbb{Z}_2$-action on $S$ induced by the involution $\sigma$. Combining with all the ingredients above, we are in a position to pose our main theorem. 
\end{remark}
Now, we are in a position to present our main theorem. Denote by $\mathcal{M}$ the moduli space  of %generic constantly curved holomorphic $2$-spheres of degree $6$ 
constantly curved holomorphic $2$-spheres belonging to the generally ramified family in $G(2,5)$, modulo the extrinsic ambient $U(5)$-equivalence %%and conjugation (with the correspondence $(v, u, w) \mapsto(\overline{v},\overline{u},\overline{w})$ in Lemma~\ref{solve by XYZ}), 
%and intrinsically, 
and the internal M\"{o}bius reparametrization.

\begin{theorem}\label{moduli space}
$\mathcal{M}=S/\mathbb{Z}_2$, so that it is a $2$-dimensional semialgebraic set. 
%The moduli space of the generic  family is $S/\mathbb{Z}_2$ and is $2$-dimensional, up to extrinsically, the ambient unitary $U(5)$-equivalence and complex conjugation, and intrinsically the internal M\"{o}bius reparametrization.
\end{theorem}

\begin{proof} 
Our first goal is to show that a holomorphic $2$-sphere of the diagonal family is also determined by its coefficients of $z^k,~k=2,3,4$ in \eqref{Ccurve}. %By the uniqueness result up to complex conjugation in Theorem~\ref{prop-exist},  %Proposition~\ref{conjugate}, 
Consider the quotients of them respectively to define a map
%%\begin{align}
%{\small
{\small\begin{equation}\label{tauf}
%%\begin{split}
\tau: S\rightarrow (\mathbb{R}^{+})^3,\quad
(t_0,t_1,t_6)\mapsto (A, B, C)\triangleq (\frac{a_{00}a_{33}}{a_{11}a_{22}},\frac{a_{00}a_{44}}{a_{11}a_{33}},\frac{a_{11}a_{44}}{a_{22}a_{33}}).
%%\end{split}
%%\end{align}
\end{equation}}
%}
%where $A_k$ is exactly the quotient of the norm of coefficient before $z_k$ for $k=2, 3, 4$. 

It follows from \eqref{eq-t01234} that $(A,B,C)=(\sqrt{t_0},\sqrt{\frac{t_0}{t_6}}\,t_1,\sqrt{\frac{t_1}{t_0t_6}}\,t_1)$. %by \eqref{t234}, and actually they are the quotients between norms of coefficients of $z^k,~k=2,3,4$ respectively, in \eqref{Ccurve}. 
It is straightforward to show that $t_0=A^2,~t_1={A^4C^2}/{B^2},~t_6={A^{10}C^4}/{B^6}$; therefore $\tau$ is injective.

The next step is to describe our moduli space. Let $\varphi _1(z)$ and $\varphi _2(\tilde{z})$ be two holomorphic $2$-spheres of the diagonal family corresponding to $t=(t_0,t_1,t_6)$ and $\tilde{t}=(\tilde{t_0},\tilde{t_1},\tilde{t_6})$, respectively. If there exists a $U\in U(5)$ such that the image of $U\cdot  \varphi _1$ agrees with that of $\varphi _2$, then $U$ induces a M\"{o}bius transformation $\tilde{z}=f(z)$ on $\mathbb{C}P^1$. Since the ramified points of $\varphi _1$ and $\varphi _2$ are both $\{0,\infty\}$ by Lemma \ref{usefulAlgeLemma}, this set is invariant under $\varphi$. Hence $\tilde{z}=\mu z$ or $\frac{\mu}{z}$, where $\mu\in \mathbb{C}^{\ast}$. Our aim is to establish that $\tilde{t}=t$ or $\tilde{t}=\sigma(t)$, which suffices to complete the proof. We divide the argument into two cases.

Case (1): Suppose that $\tilde{z}=\mu z$. Comparing the first two and last two terms of $\varphi _1$ and $\varphi _2$, we obtain that (see \eqref{Ccurve})
{\small \begin{align*}
U\cdot e_0\wedge U\cdot e_1 \equiv 0 \mod (e_0,e_1),~~~ U\cdot e_0\wedge U\cdot e_2 \equiv 0 \mod (e_0,e_2),\\
U\cdot e_2\wedge U\cdot e_4 \equiv 0 \mod (e_2,e_4),~~~ U\cdot e_3\wedge U\cdot e_4 \equiv 0 \mod (e_3,e_4).
\end{align*}}
Hence, $U=\diag\{u_{00},\ldots,u_{44}\}$ is diagonal as $U$ is unitary.  As a result, they share the same quotients in \eqref{tauf}, i.e., $\tau(t)=\tau(\tilde{t}),$ so that $t=\tilde{t}$ by the injectivity of $\tau$. 

Case (2): Suppose that $\tilde{z}=\frac{\mu}{z}$. Following a similar argument as in Case (1), we see that $U$ is anti-diagonal. Consequently, the quotients in \eqref{tauf} satisfy 
$A(\tilde{t})=C(t),~~B(\tilde{t})=B(t),~~C(\tilde{t})=A(t).$
By the exposition below \eqref{tauf}, it is easy to show that $\tilde{t}=\sigma(t)$.

%Hence, in each case, $U\cdot \varphi _1$ and $\varphi _2$ are the same up to complex conjugation. 
Now, the conclusion follows from Theorem~\ref{thm-tangential}. %% and Remark~\ref{rk-involution}. 
\end{proof}

%, where the only holomorphic homogeneous $2$-sphere of curvature $\frac{2}{3}$ in $G(2,5)$ is the Veronese curve $PSL_2\cdot u^6$.

%%We conclude this section by pointing out that, to the authors' knowledge, the holomorphic $2$-spheres of the diagonal family are novel and appear first time in the literature.

%{\color{blue}

%The moduli space of the diagonal family is shown in the following picture, which is the brown jagged portion of the upper right section. It appears that this moduli space is connected; it would be an interesting question whether this is true.

%\begin{figure}[htbp]
%\centering
%\includegraphics[width=0.5\textwidth]{moduli.png}
%\end{figure}

The end of this section is devoted to the construction of several interesting individual as well as 1-parameter families of examples. %, which give all possible associated number of distinct eigenvalues of the diagonal matrix $A=\diag\{1,a_{11},a_{22},a_{33},a_{44}\}$ in the diagonal family. Note that under normalization in Remark \ref{rrk}, we assume that $a_{22}=a_{33}$ in the sequel.

%New examples and the corresponding singular values

%%In this section, we illustrate our conjecture that the moduli space of the diagonal family $S/\mathbb{Z}_2$ is topologically a closed disk by numerical evidences. 

%%To study its topology, consider the function $g(t)\triangleq \frac{t_1^3}{t_0^2t_6}$ restricted on the semi-algebraic set $S$, see \eqref{sigmag}. By Mathematica, we conjecture that the range of $g$ is a closed interval $I\subseteq [\frac{1}{50},50]$. For ease of notations, we denote the level set of $g$ by $S_{a}$, where $a\in I$. The involution $\sigma$ actually interchanges the level set $S_{a}$ with $S_{\frac{1}{a}}$ since
%%\begin{equation}\label{specialsigma}
%%g(t)g(\sigma (t))=1.
%%\end{equation}
%%Denote the minimal and maximal of $g$ by $m$ and $M$, respectively. Then $M=\frac{1}{m}$ by  \eqref{specialsigma}. Moreover, the fixed-point set of $\sigma$ is exactly the level set $S_1$.

%%\begin{lemma}

Recall the involution $\sigma:S\rightarrow S$ and its invariant subset $S_1$ defined by setting $g=1$, so that $1=g={t_1^3}/{(t_0^2\,t_6)}$. It is a piecewise smooth simple closed curve. Indeed, %% homeomorphic to the unit circle.
%%\end{lemma}
%%\begin{proof}
substitute $t_6={t_1^3}/{t_0^2}$ into \eqref{eq-alge} and ignore the non-zero denominator and the non-zero factors. The level set $S_1$ is the semialgebraic set defined by the three inequalities in \eqref{inequality} and
{\small\begin{equation*}
\aligned
&  \left( 441\,t_0^{8}-42\,t_0^{7}+t_0^
{6}-72\,t_0^{5}t_{{1}}-5136\,t_0^{4}t_{{1}}-1592\,
t_0^{3}t_{{1}}+7056\,t_0^{2}t_1^{2}-672\,t_{{0}}t_1^{2
}+16\,t_1^{2} \right)\cdot\\
&\left( t_{{0}}-1 \right)  \left( 2\,t_0^{3}-3\,t_1t_0+t
_{{1}} \right)=0.
\endaligned
\end{equation*}}
 In the $t_0t_1$-coordinate plane, $S_1$ is plotted in Figure $3$. %It is constituted by three parts, the blue vertical line segment corresponds $ (t_{{0}}-1 )=0$, to 
The branch corresponding to $ (t_{{0}}-1 )=0$ is the blue vertical line segment. %; see also Example {\em \ref{ex-t01}}. 
The second branch described by $\left( 2\,t_0^{3}-3\,t_1t_0+t_{{1}} \right)=0$ is the end point $(1,1)$ of the blue line segment. The third branch corresponds to the union of the (upper) brown and (lower) green curves parametrized by
\begin{equation}\label{eq-curve}
\psi_1=\{(s,F_1(s))~|~s\in[1,{11}/{6}]\},~~~\psi_2=\{(s,F_2(s))~|~s\in[1,{11}/{6}]\},
\end{equation}
respectively, where {\small $F_1=(t_0^3(199+642t_0+9t_0^2+30\Delta))/(4(21t_0-1)^2)$, $F_2=(t_0^3(199+642t_0+9t_0^2-30\Delta)/(4(21t_0-1)^2)$}, and {\small $\Delta\triangleq \left(3 t_0+2\right) \sqrt{\left(4 t_0+1\right) \left(11-6 t_0\right)}$}. 

It follows from  Theorem \ref{moduli space} that the moduli space is $\mathcal{M}=S/\sigma$ with the simple closed curve $S_1$ on its boundary. By applying the coordinate transformation $(t_0, t_1, t_6) \mapsto (t_0, t_1, \lambda)$ with $\lambda=1/g$, we can plot $\mathcal{M}$ as in Figure $4$. It looks like a horn, with $S_1$ marked in red, and the level sets of $g=2$ and $g=3$ marked in green and blue, respectively. The figure seems to suggest that the moduli space $\mathcal{M}$ is a topological disk. It would be interesting to see whether this is indeed the case.
\begin{figure}[htbp]
\begin{minipage}{.49\linewidth}%\label{figa}
\centering
\includegraphics[width=0.25\textwidth]{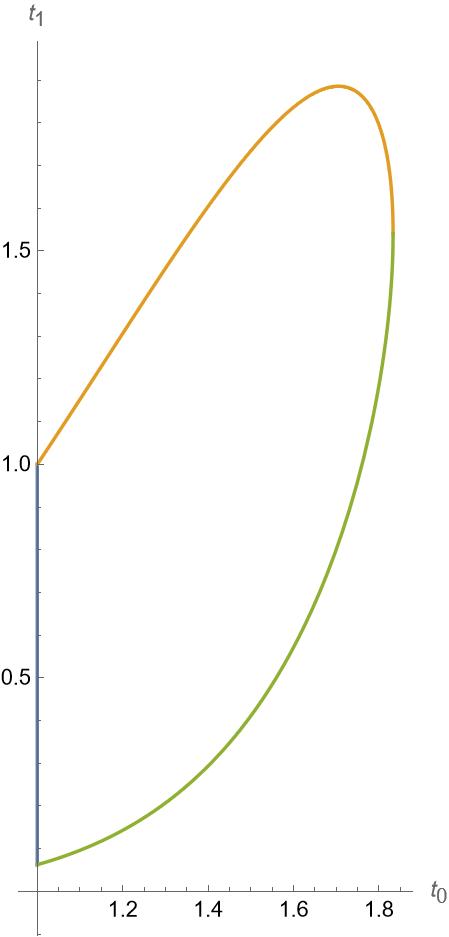}
\caption{The level set $S_1$}
%%\end{figure}
\end{minipage}
%\hfill
\begin{minipage}{.49\linewidth}%\label{fig:subfig:b}
\centering
%%\begin{figure}[htbp]
%%\centering
\includegraphics[width=0.50\textwidth]{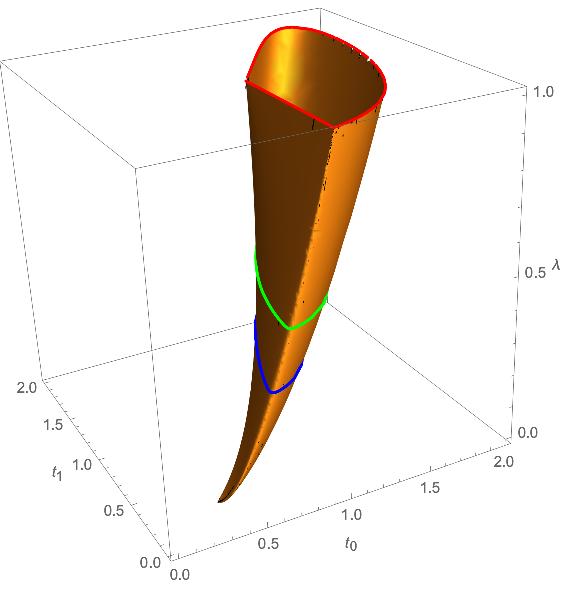}
\caption{The moduli space $\mathcal{M}$}
\end{minipage}
\end{figure}
\vskip -0.3cm
\begin{example}\label{example0} 
We point out that examples on the blue line segment coincide with the $1$-parameter family~\eqref{family33} in Remark~{\rm \ref{rk4}}. In fact, it follows from \eqref{t234} that
%{\small 
\begin{align*}
&t_0=1,~t_2=t_1,~t_3={5t_1^2}/{(4t_1+1)},~t_4=t_1^2,~t_5=t_1^3,~t_6=t_1^3,\\
&a_{00}=1,~a_{11}=1,~a_{22}=a_{33}={1}/{\sqrt{t_1}},~a_{44}={1}/{t_1}.
\end{align*}%}
Moreover, substituting all the data into \eqref{JP2}, we obtain that %see that this $1$-parameter family can be rewritten as {\small 
{\small\begin{equation}\label{family1}
\begin{pmatrix}
1 & 0 & -\sqrt{6}e^{\sqrt{-1}\theta_2}z^2 & -4\sqrt{\frac{t_3}{t_1}}e^{\sqrt{-1}\theta_3}z^3 & -3e^{\sqrt{-1}\theta_4}z^4\\
0 & 1 & \sqrt{6}e^{\sqrt{-1}\theta_1}z & 3e^{\sqrt{-1}\theta_2}z^2 & 2\frac{\sqrt{t_3}}{t_1}e^{\sqrt{-1}\theta_3}z^3\\
\end{pmatrix}. 
\end{equation}}
%which is in one-to-one correspondence with the $1$-parameter family~\eqref{family3} in Remark~{\rm \ref{rk4}}.

Set 
$t_1\triangleq {(5+3\cos\theta)}/{(20-12\cos\theta)}$, then $\cos\theta={(20t_1-5)}/{3(4t_1+1)}$. 
Then $\theta_0\triangleq 0, \;\theta_6\triangleq 0$, and
\[\aligned
&~\theta_1\triangleq \theta-\frac{\beta_0-\beta_1}{2},\;\theta_2\triangleq \theta ,\;\theta_3\triangleq \theta+\frac{\beta_0+\beta_1}{2}  ,\;\theta_4\triangleq \theta ,\;\theta_5\triangleq \theta-\frac{\beta_0-\beta_1}{2}  ,
\endaligned\]
satisfy \eqref{angle}, where % $\frac{\beta_0-\beta_1}{2} =(\beta_0-\beta_1)/2$ and 
$$\beta_0={\rm Arg}\left(\frac{3+e^{\sqrt{-1}\theta}}{\sqrt{10+6\cos\theta}}\right),~~~ \beta_1={\rm Arg}\left(\frac{3-e^{\sqrt{-1}\theta}}{\sqrt{10-6\cos\theta}}\right).$$  
%to satisfy\;
\iffalse
{\small $e^{\sqrt{-1}\beta_0}\triangleq {3+e^{\sqrt{-1}\theta}}/{\sqrt{10+6\cos\theta}},$}\;
{\small $e^{\sqrt{-1}\beta_1}\triangleq {3-e^{\sqrt{-1}\theta}}/{\sqrt{10-6\cos\theta}},$} \;
{\small $e^{\sqrt{-1}2\frac{\beta_0-\beta_1}{2} }\triangleq e^{\sqrt{-1}(\beta_0-\beta_1)},$\;
and set
 \[\algned
&\theta_0\triangleq 0,\;~\theta_1\triangleq -\frac{\beta_0-\beta_1}{2} -\theta ,\;\theta_2\triangleq -\theta ,\;\theta_3\triangleq \beta_1-\theta+\frac{\beta_0-\beta_1}{2},\\
&\theta_4\triangleq -\theta ,\;\theta_5\triangleq -\frac{\beta_0-\beta_1}{2} -\theta ,\;\theta_6\triangleq 0.
\endaligned\]
It is straightforward to verify that the  $t_l$ and $\theta_j$ as functions of $\theta$ satisfy \eqref{perturbed}, or equivalently, \eqref{angle}.
\fi
It is straightforward to verify that \eqref{family33} differs from \eqref{family1}  by multiplying its third and fourth columns by 
$e^{\sqrt{-1}(\beta_1-\beta_0+\theta )},$
 its last column by %%of \eqref{family3} by
 $e^{\sqrt{-1}(2(\beta_1-\beta_0)+\theta )},$
 and performing a reparameterization 
 $z\mapsto e^{\sqrt{-1}{(\beta_0-\beta_1)}/{2} }z.$ Note that $\pm \theta$ give the same $t_1$; they correspond to the two complex-conjugated solutions.
%%\begin{figure}[htbp]
%%\centering
%%\includegraphics[width=0.10\textwidth]{fixed-level-set.jpg}
%%\caption{Level set $S_1$}
%%\end{figure}
\end{example}

\begin{proposition}\label{nonhomogenous}
The second fundamental form $A$ of a constantly curved holomorphic $2$-sphere of degree $6$ belonging to the generally ramified family is not of constant norm, %non-homogeneous 
except for the standard Veronese curve \eqref{eq-standard}. %$PSL_2\cdot u^6$.
 \end{proposition} 
 \begin{proof}
 %We need only show that $||A||^2$, the norm squared of the second fundamental form, is not constant.  
 It follows from the Gauss equation that %{\small 
\begin{equation}\label{eq-2nd}
||A||^2=20/3-||{\partial F}/{\partial z}\wedge {\partial F}/{\partial z} ||^2/(9(1+|z|^2)^8),
\end{equation}
%}
where $F$ is the Pl\"{u}cker embedding of the holomorphic $2$-sphere in $G(2,5)$ into $\mathbb{C}P^9$ (see %%the beginning of Section \ref{secramifiedpoints}
\cite[p.6, p.9]{He2022}  for details).  Note that $||{\partial F}/{\partial z}\wedge {\partial F}/{\partial z} ||^2$ only vanishes at ramified points.  Therefore, using Lemma~\ref{usefulAlgeLemma} %Corollary \ref{ramified pts coro} 
we can derive that the second term on the right-hand side of $||A||^2$ is not constant. %, since its numerator vanishes at ramified points. %denominator $1+|z|^2$ is irreducible %%in $\mathbb{C}[z,\bar{z}]$ 
%while its numerator is nontrivial in ${\mathbb C}[z,\overline{z}]$. %%$f_0,\ldots,f_4$ in \eqref{fwfcase1} cannot vanish simultaneously.
 \end{proof}

\begin{example}\label{eg-exact} 
On the level set $S_1$, choose $t_0=11/6$. Then we can solve for $t_1=1331/864$. It gives an exact solution to \eqref{eq-alge},
{\small $$t_0=\frac{11}{6},~~~t_1=\frac{131}{864},~~~t_2=\frac{14641}{7776},~~~t_3=\frac{73205}{41472},~~~t_4=\frac{1771561}{1119744},~~~t_5=t_6=\frac{19487171}{17915904},$$}

It is checked that $X=Y=5\sqrt{5}/\sqrt{33}$ and $Z=2.$
from which the angles $\{\theta_1, \cdots, \theta_5\}$ can be solved. 
\end{example}

\begin{example}\label{eg-exact1} 
On the level set $S_1$, choose $t_0=\left(2 \sqrt{79}+20\right)/21$. Then we can solve for  $t_1=\left(2 \sqrt{79}+20\right)/21$. It gives an exact solution to \eqref{eq-alge}, 
$$\aligned &t_0=t_1=t_5=t_6= \left(2 \sqrt{79}+20\right)/21,~~~t_2=t_4=\left(23 \sqrt{79}+209\right)/189,\\&t_3= (9 + \sqrt{79})/8,\endaligned$$  
from which the angles $\{\theta_1, \cdots, \theta_5\}$ can be solved. Note that  for this example, the diagonal matrix $A$ has two distinct eigenvalues
$a_{00}=a_{44}\neq a_{11}=a_{22}=a_{33}.$
\end{example}
  
\begin{example}\label{eg-cusp} Start with the equations 
$P\triangleq X^2-4=0,\,Q\triangleq Y^2-4=0,\,R\triangleq Z^2-4=0,$ 
with $X,Y,Z$ given in \eqref{XYZt0t1t6} to express them in terms of the variables $t_0, t_1,g, $ with $t_6=t_1^3/(t_0^2g)$ by \eqref{sigmag}. Continue to compute the derived resultants of the refined numerators $P',Q',R'$ of $P,Q,R$,  in terms of $t_0, t_1, g$, after removing powers of $g-1$ and those single-variable factors without positive solutions by, e.g., Sturm's algorithm for counting the exact number of distinct positive roots of a real polynomial, while setting aside possible candidate polynomials before proceeding with the next level of resultant computation{\emph ;} 
along the way, we heed the constraint that $(gt_0,gt_1,1/g)$ is a set of solution if $(t_0,t_1,g)$ is, by Proposition {\em \ref{PROP7.1}}, to further narrow down the candidates. 
We end up with the exact equations for possible $t_0, t_1, g:$

%%\begin{tiny}\begin{equation}\label{exact}
\begin{tiny}$$
\aligned 
&p\triangleq 3004245721g^6 - 139634316726g^5 - 67838574585g^4 - 318786958820g^3 - 67838574585g^2 \\&- 139634316726g + 3004245721=0,\\
&q\triangleq 2537649t_0^6 - 40347234t_0^5 + 36454860t_0^4 - 19711080t_0^3 + 26076060t_0^2 - 17915544t_0 + 3452164=0,\\
&r\triangleq 6861904453295341780216896t_1^6 - 57789440847499427495680896t_1^5 - 3541432129528999644182160t_1^4\\& + 2695787548715827169923680t_1^3 
- 242591843875043061525060t_1^2 - 261056339362401426814176t_1 \\
&+ 53689575410338079139841=0.\endaligned$$\end{tiny}
%%&3496660660534477268390232249t_1^6 - 571240029103084624711190627844t_1^5 - 180334371204315235320396777540t_1^4 \\
%%&- 20690519947224917218021478880t_1^3 - 989489005236297587491889040t_1^2 - 17198222315646241060398144t_1 \\
%%&+ 577631510537395829824=0,
%%\endaligned
%%$$\end{tiny}
%%\end{equation}\end{tiny}

Compute the Gr\"{o}bner basis of the ideal $(P',Q,'R',p,q,r)$ to obtain the basis consisting of six elements of which we only record the two essential ones, %%$A,B,C,E,F,G$, 
\begin{tiny}$$
\aligned
%%&A \triangleq  46025066872492081067756g^2 + 4852886963441949937895445gt_0 - 11136245043517510673701080gt_1 \\&- 1669511136582870636511672g - 178840494500416295447115t_0 + 553492739894709138991560t_1\\& - 1516395520238245635400594=0,\\
%%&B \triangleq  -3505185607812573411262775046179165g^2 - 310397737581594038715309151616344122gt_1\\& + 84177653707712908172851478293254948t_1^2+ 227821426814783887832218950013251307g \\&+ 701722134896314276352682231147834210t_0 - 1921841171858818894254633052580523474t_1\\& - 212208017675696524298977688364142379=0,\\
%%&C \triangleq  -1889840533719569527975508g^2 - 138366509417798966097762510gt_1 \\&+ 14558660890325849813686335t_0^2+ 117495249698005458870457936g + 366332774802905466030375990t_0\\& - 1016419035967548320357952030t_1 \\&- 109045522401694057601295488=0,\\
&E \triangleq  30407219135534569920865279281g^2t_1 - 5684396631350441922486404084g^2 \\&+ 4826381508202691775218328738gt_1 + 8781109390742136392820835978g \\&+ 22087970177286319548246901485t_0 - 37952752504503427337193407559t_1 \\&- 10129670167010754418270796864=0,\\
%%&F \triangleq  -44440502586557410661318415833g^2 - 3705254673218497521344738988000gt_1\\&+ 608144382710691398417305585620t_0t_1 + 2846929753851981521603535157426g \\&+ 8895272990921454238235126362440t_0 - 24427152449198671502238278481360t_1\\& - 2678986150939933082063002417193=0,\\
&G \triangleq  323983664320381367395969030814241g^3 - 15097919249633508113716536736052777g^2 \\&+ 24001947052912436490532391777190000gt_1 - 10297270579570244241163795555112489g \\&- 21160216103727154670480065729425120t_0 + 38155570002907589892718590589124280t_1\\& - 10753529104240427995602453394128335=0.
\endaligned
$$\end{tiny}

We obtain $t_0\triangleq R/S$ and $t_1=T/U$ in closed form of $g$, where
\begin{tiny}$$
\aligned
&R\triangleq 323983664320381367395969030814241g^5 - 15046494988853004912329176221825959g^4\\& - 8611085577295995251867740593198034g^3+ 6658017307603866925677723269688366g^2 \\&+ 8122830950478969874129540484608001g + 26132918116090821757236925434099385,\\
&S\triangleq 21160216103727154670480065729425120g^2 + 20793797801629220801560324794395760g \\&+ 1305303435283084266467628002760120,\\
&T\triangleq  -423618308217230277983078980100353g^3 + 26861312395386909671099284789417865g^2\\& + 2464682459146076205358051730246729g + 26749087059945119323559494796984559,\\     
&U\triangleq 38088388986708878406864118312965216g^2 + 37428836042932597442808584629912368g \\&
+ 2349546183509551679641730404968216.
\endaligned
$$\end{tiny}

It is then checked that all the remaining equations in the basis are compatible with $p=0$.
Now, $p=0$ has two positive real roots reciprocal to each other as the coefficients of $p$ are symmetric, which are approximately
$g\sim 0.0212731522$ and $47.0076078738$ {\em (}Since all the above polynomial equations are exact, the listed numerical values are accurate up to the last digit, checked by the intermediate value theorem, for instance.{\em )} We then derive the corresponding values for $t_0$ and $t_1$ through $R, S, T, U$ to yield
{\small}\begin{equation}\nonumber
\aligned
(t_0,t_1,g)&\sim (0.3184944933,0.1803379951,47.0076078738),\;\text{or}\\
&\sim (14.9716642533,8.4772577609,0.0212731522),
\endaligned
\end{equation}
accurate up to the last digit, in accord with Proposition {\em \ref{PROP7.1}}{\emph ;} both give $X=Y=Z=2$.  %% and one triple is identified with the other by the involution $\sigma$ in \eqref{sigmag}. 
The second set gives the pointed end of the horn in Figure $2$. 

This is the third and the last example, aside from the two given in Remark {\em \ref{RMK}} with $g=1$, for which there is only one constantly curved $2$-sphere belonging to the generally ramified family in the corresponding Fano $3$-fold $A({\mathcal H}_0^3)$, where $A$ is computed by \eqref{eq-t01234}. %%to be expressed in terms of $g$ by $R,S,T,U$.
%%in the present case,
%%{\small $$)\sim (1,\;\;0.2584431111,\;\;0.3434569475,\;\;0.3434569475,\;\;0.2575909584)$$} is computed by \eqref{eq-t01234}; here, we use the smaller $g$ to represent the horn's pointed head.
\end{example}

\begin{example} Set $t_1\triangleq t_0^2/6$ in $F$ given in \eqref{eq-alge} and factor out positive terms to yield
\begin{tiny}$$
\aligned
&f(g,t_0)\triangleq 
190512g^4t_0^6 + 20736g^4t_0^5 + 95256g^3t_0^6 + 27g^4t_0^4 - 205416g^3t_0^5 - 401301g^3t_0^4\\ 
&- 104328g^2t_0^5- 6264g^3t_0^3 - 59319g^2t_0^4 + 168282g^2t_0^3 + 32913gt_0^4 + 202140g^2t_0^2 + 35388gt_0^3\\
& + 6720gt_0^2
 + 2034t_0^3 + 19504gt_0 + 2460t_02 + 688t_0 - 32=0.
\endaligned
$$\end{tiny}

It defines a plane algebraic curve $C$. We claim that $C^*\subset C$ falling in the rectangle ${\mathcal R}$ given by $8/15 \leq t_0\leq 5,\, 1475/10000 \leq g\leq 3,$ is a smooth, connected closed curve contained in $S$, the double of the moduli space ${\mathcal M}$. 

Firstly, observe that $(t_0,g)=(1,1)$ solves $f=0$ so that that $C^*$ is not empty. It is also directly checked that $\frac{\partial f}{\partial t_0}/\frac{\partial f}{\partial g}=2$ at $(t_0,g)=(1,1)$, so that the implicit function theorem implies that $f=0$ is locally a curve $(t_0,g(t_0))$ around $(t_0,g)=(1,1)$ with negative slope. 

%%In the following, we set the number of significant decimal digits to be $20$, verified by Sturm's algorithm and the intermediate value theorem, for all real roots of the engaged exact polynomial equations out of resultant computations, and truncate them to the first ten digits to warrant precision. 

Setting $t_0\triangleq 8/15$ or $5$, and $g\triangleq 1475/10000$ or $3$, respectively, we solve $f(g,t_0)=0$ to 
attain {\rm (}accurate up to the last digit for the exact polynomials{\rm )}
{\small$$
\aligned
&\text{for}\;\,t_0=8/15,\; \nexists\;\text{real}\; g,\;\text{while for}\; t_0= 5,\, g\sim -0.4687373438 ,\;\text{or}\;  -0.0109931977;\\
&\text{for}\;\,g=1475/10000,\;\,t_0\sim 0.0088038166,\; \text{while for}\;\, g=3,\\
&t_0\sim -0.5591240674, -0.4272041173, -0.0337884110, 0.0005317397.
\endaligned 
$$}

This means that the set $C^*$ never leaves the rectangle ${\mathcal R}$, so that by analytic continuation of an algebraic curve, $C^*$ consists of closed curves and, a priori, a few isolated points. The latter can be ruled out since these finitely many points must satisfy 
$f={\partial f}/{\partial t_0}={\partial f}/{\partial g}=0$ and the Gr\"{o}bner basis associated with the ideal $(f,{\partial f}/{\partial t_0},{\partial f}/{\partial g})$ is $\{g-1,3t_0+2\}$ whose zero locus $(t_0,g)=(-2/3,1)$ does not fall in the domain ${\mathcal R}$.
%%resultant calculation of the two equations produce two single-variable polynomials in $t_0$ and $g$, respectively, whose possible roots are
%%\begin{equation}\label{none}
%%\aligned
%%&t_0\sim0.5774990952,\; 1.1042080406,\;\text{or}\;\, 2.7004220933;\\
%%&g\sim0.2492816809,\; 0.3254047116,\;\text{or}\;\, 0.4352506873,
%%\endaligned
%%\end{equation}{\em )} 
As a result, it also implies that the finitely many closed curves constituting $C^*$ are smooth and disconnected in ${\mathcal R}$.

By calculating the resultants of $f={\partial f}/{\partial t_0}=0$ against $g$ and $t_0$ and solving for the roots, we verify that none of the possible pairs of $(t_0,g)$ satisfy \eqref{DH} 
{\em (}see the remark below for the engaged computational error analysis for rational functions{\rm )}, except possibly for two points $(t_0,g)$ approximately at
\begin{equation}\label{none}
(0.6547026351, 2.9099350324),\;\;\text{or}\;\; \;(4.5794327836, 0.1475263321),
\end{equation}
accurate up to the last digit. Since there exist at least two such points, this proves that $C^*$ is only tangent to the horizontal lines, $g=\text{constants}$, precisely at the two points; likewise, this is also true for the vertical line test. In particular, $C^*$ has only one connected component as, otherwise, we would have more than two points tangent to horizontal or vertical lines. 

We calculate the resultants of $f$ and the numerator of $R\triangleq Z^2-4$ against $g$ and $t_0$ and solve for the roots, to confirm that the only point of intersection of the curve $C^*$ and the boundary of $Z^2\leq 4$ occurs with tangency at 
$$
(t_0 ,g)\sim (1.5271772661,0.4663765333), 
$$
with the corresponding $X=Y\sim 1.8718004195$ and $Z=2$. It follows that $C^*$ lies completely in $Z^2\leq 4$ since $(t_0,g)=(1,1)$ satisfies $Z^2<4$. In particular, the three constraints in \eqref{inequality} are satisfied by Remark {\rm \ref{three inequality not indep}}.

Figure $5$ depicts the curve $C^*$ {\rm (}in red{\rm )} in $S$. Since it extends into the region with $g>1$, we apply the involution $\sigma$ to flip it back into ${\mathcal M}$ with $g\leq 1$. Figure $6$ shows the resulting self-crossing, flipped $C^*$ {\em (}in red{\rm )}, %%where the dotted curves define $Z^2=4$, and the flipped $C^*$  
which opens at $g=1$ for which $t_0=1$ or $t_0\sim 1.4542230103$. The region bounded by the three constraints is colored yellow.

\begin{figure}[htbp]
\begin{minipage}{.49\linewidth}%\label{fig:subfig:b}
\centering
\includegraphics[width=0.40\textwidth]{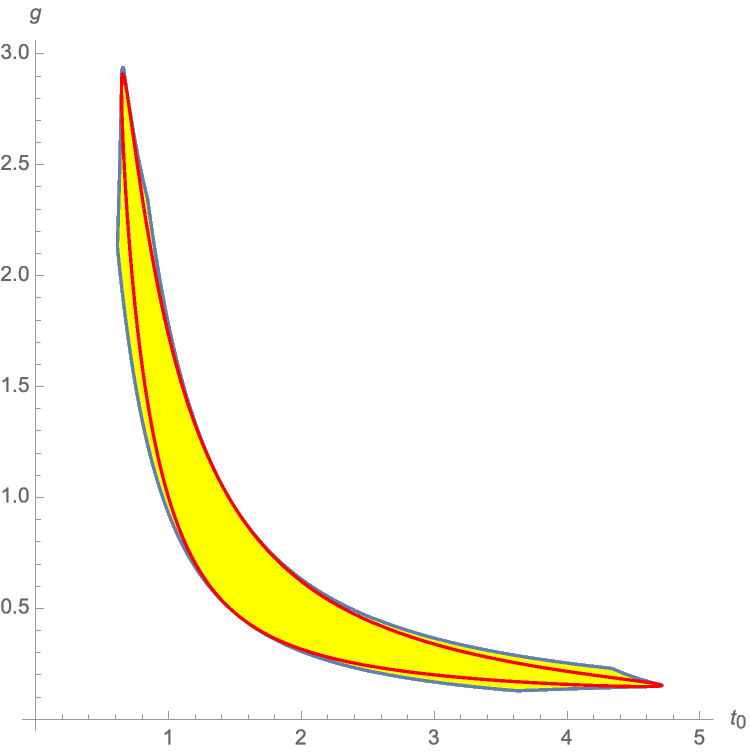}
\caption{The curve $C^*$ in $S$}
%%\end{figure}
\end{minipage}
%%\hfill
\begin{minipage}{.49\linewidth}%\label{fig:subfig:b}
\centering
%%\begin{figure}[htbp]
%%\centering
\includegraphics[width=0.40\textwidth]{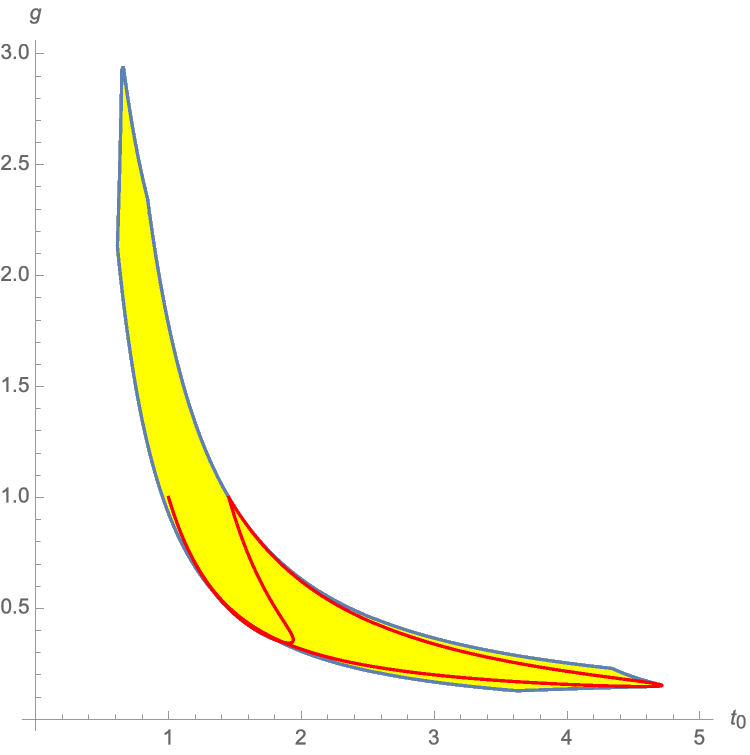}
\caption{Folded $C^*$ in $\mathcal{M}$}
\end{minipage}
\end{figure}

%%The $3$-dimensional version with $t_1=t_0^2/6$ is the curve viewed in the horn in Figure 2.

\end{example}

\begin{remark} %%Care must be taken when verifying that the obtained approximate values of $t_0, t_1, g$ do not satisfy the cubic equation \eqref{DH} that defines the moduli space, which is a rational function in $(t_0,t_1,g)$, even
%%when we are able to solve exact polynomials with integer coefficients as above to numerically approximate $t_0, t_1, g$ as precisely as we prefer.

Let $f(x,y)=\sum_{m,n=0}^{M,N} a_{mn}\, x^my^n$ and $l(x,y)\triangleq \sum_{i,j=0}^{I,J} b_{ij}\, x^iy^j$ over a rectangle ${\mathcal R}:[a,b]\times [c,d]$ with $a, c>0$. Assume $l(x,y)>0$ and define the positive function $||f||(x,y)\triangleq \sum_{m,n=0}^{M,N} |a_{mn}| \,x^my^n$ over ${\mathcal R}$. Given $(x_0,y_0), (x,y) \in {\mathcal R}$ with $0<|x_0-x|, |y_0-y|<h$, where $h>0$ is so small that 
$nh<<1$ for $n=M,N,I,$ or $J$,  then $p(x,y)\triangleq f(x,y)/l(x,y)$ satisfies the error estimate
\begin{equation}\label{error}
|p(x_0,y_0)-p(x,y)|\leq (C(M,N) + C(I,J)) \sup_{(x,y)\in{\mathcal R}}(||f||(x,y)/l(x,y)),
\end{equation}
where, for $n\in{\mathbb N}$ with $nh<1$, we define
$\gamma_n\triangleq nh/(1-nh)$, and 
$$
C(p,q)\triangleq (e^{1/a}-1)\gamma_p +(e^{1/c}-1)\gamma_q+(e^{1/a}-1)(e^{1/c}-1)\gamma_p\gamma_q
$$
for $p, q\in{\mathbb N}$. {\em (}We leave it to the reader to verify.{\em )}

In Example $5$, $x\triangleq g$ and $y\triangleq t_0$, ${\mathcal R}$ is the rectangle $[1475/10000, 3]\times [8/15, 5]$, and $f(g,t_0)$ is given in Example $5$. %%with $a=1475/10000$ and $c=8/15$. In \eqref{DH}, 
Write, for $H$ in \eqref{DH}, 
%%$H\triangleq f(g,t_0)/l(g,t_0), t_1=t_0^2/6,$
%%where $f(g,t_0)$ is given in Example $5$ 
$$
H=f(g,t_0)/l(g,t_0),\quad l(g,t_0)\triangleq 405000\,t_0^3\,g^2\,(3t_0 + 2)\,(3gt_0 + 2)>0,
$$

Since $M=4, N=6, A=3,$ and $B=5$, if we take $h\triangleq 10^{-20}$, the error estimate \eqref{error} gives that $C(M,N)+C(A,B)$ is in the magnitude of $10^{-17}$, and an elementary mini-max estimate derives $||f||(x,y)/g(x,y)\leq 1$ for all $(x,y)\in{\mathcal R}$, so that the error is in the magnitude of $10^{-17}$. Consequently, all the engaged computations for the data satisfying $H\neq 0$ to obtain, e.g., \eqref{none} are accurate up to the tenth decimal place if we set the last significant decimal place to be the twentieth; %% and truncate the obtained numerical values at the tenth, where 
all the undesired values above, in fact, are such that their third decimal digits are nonzero to satisfy $H\neq 0$.

\end{remark}

\vspace{2mm}

\noindent {\small Department of Mathematics, Washington University, St. Louis, MO 63130\\
School of Mathematical Sciences, Beihang University, Beijing 102206, China\\
School of Mathematical Sciences and LPMC, Nankai University, Tianjin 300071, China}

\begin{tiny}\noindent E-mail: chi@wustl.edu;~~~\phantom{,,}xiezhenxiao@buaa.edu.cn;~~~\phantom{,,}xuyan2014@mails.ucas.ac.cn.\end{tiny}

\end{document}